\documentclass[article] {memoir}
\usepackage{amsmath, amsthm, amsfonts,amssymb}

\usepackage{tikz} \usetikzlibrary{cd} 

\midsloppy 

\settypeblocksize{47\baselineskip}{360pt}{*} 
\setulmargins{1.5in}{*}{*} 
\setlrmargins{*}{*}{1} 
\setheadfoot{\onelineskip}{2\onelineskip} 
\setheaderspaces{*}{1.5\onelineskip}{*}  

\checkandfixthelayout

\theoremstyle{plain}
\newtheorem{proposition}{Proposition}
\newtheorem{theorem}[proposition]{Theorem}
\newtheorem{theorem*}{Main Theorem}
\newtheorem{corollary}[proposition]{Corollary}

\newtheorem{lemma}[proposition]{Lemma}

\theoremstyle{definition}
\newtheorem{definition}{Definition}

\theoremstyle{remark}
\newtheorem{example}{Example}
\newtheorem*{remark}{Remark}

\newcommand{\defeq}{\stackrel{\mathrm{def}}{\, = \,}}

\newcommand{\bZ}{{\mathbb{Z}}}
\newcommand{\bR}{{\mathbb{R}}}
\newcommand{\bC}{{\mathbb{C}}}
\newcommand{\bQ}{{\mathbb{Q}}}
\newcommand{\bF}{{\mathbb{F}}}

\newcommand{\bH}{{\mathbb{H}}}

\newcommand{\be}{{\mathbf{e}}}

\newcommand{\OO}{{\mathrm{O}}}
\newcommand{\SO}{{\mathrm{SO}}}

\newcommand{\Aut}{{\mathrm{Aut}}}
\newcommand{\SL}{{\mathrm{SL}}}
\newcommand{\GL}{{\mathrm{GL}}}
\newcommand{\N}{{\mathbf{N}}}

\begin{document} 

\title{Report on Freely Representable Groups}

\author{Wayne Aitken\thanks{Thanks to Shahed Sharif and his student Antony Savage for showing me the usefulness of
norm relations in the problem of finding units in algebraic number theory, 
and further thanks to Shahed Sharif for discussions in 2019 that led me to wonder what groups
have norm relations of unity. 
Thanks to the authors of \cite{biasse2020norm} for giving  nice
answers to such questions and for prompting my interest in freely representable groups, an interest that led to this report.
Finally, thanks to my student Jason Martin for his help in reviewing various drafts of this report.}}

\date{\today}

\maketitle


\begin{abstract}This report is an account of freely representable groups, which are finite groups admitting linear representations
whose only fixed point for a nonidentity element is the zero vector. The standard reference for such groups is Wolf~(1967)
where such groups are used to classify spaces of constant positive curvature. Such groups also arise in the theory
of norm relations in algebraic number theory, as demonstrated recently by Biasse, Fieker, Hofmann, and Page (2020).
This report aims to synthesize the information and results from these and  other sources
 to give a continuous, self-contained development of the subject. 
I  introduce new points of view, terminology,  results, and proofs in an effort to give a coherent, detailed, self-contained,
and accessible narrative.
\end{abstract}

This is an account of freely representable groups, which are finite groups admitting linear representations
whose only fixed point for a nonidentity element is the zero vector.
Such groups attracted my attention as being exactly
the groups that do not have a general type of norm relation (as recently shown in~\cite{biasse2020norm}).
Interestingly, such groups arose earlier as the key to the classification of Riemannian manifolds with constant positive curvature (see~\cite{wolf2011}).
Such groups were studied even earlier from a purely group theoretical point of view. For example, in 1905, 
W.~Burnside~\cite{burnside1905b} proved necessary conditions 
for a group to be freely representable. In~1955, S.~A.~Amitsur~\cite{amitsur1955} observed that all finite subgroups of a division ring are freely 
representable groups. Amitsur used this as the starting point of his classification of such groups; the actual classification required class field theory.
Finally, these groups are related to the subject of Frobenius complements in finite group theory and groups with periodic cohomology.\footnote{See Wall~\cite{wall2013} for periodic cohomology. 
This report does not consider Amitsur's classification, Frobenius complements, or periodic cohomology. I hope to include some of these topics in sequels or in a future version of the current report.}

Freely representable groups include cyclic groups, and in general behave like cyclic groups in many interesting
ways. 
Non-cyclic examples of freely representable groups can be
readily given.
In fact, the quaternion group with 8 elements is a freely representable group and is the smallest such group.
Freely representable groups can be thought of as playing  a role in the class of finite groups analogous to the role played by cyclic groups
in the class of finite Abelian groups. 
Because of this we regard freely representable groups as a kind of ``cycloidal group'', where the term ``cycloidal''  makes
use of the \text{suffix~``-oid''} meaning ``having the likeness of''.\footnote{
There are actually several types of ``cycloidal groups'' considered in this report besides freely representable groups.
One interesting class is the collection of  semiprime-cyclic groups. These are groups with the property that
any subgroup of order the product of two primes is cyclic. The remarkable fact is that 
freely representable group and  semiprime-cyclic groups correspond exactly for solvable groups.
The semiprime-cyclic groups will be considered in more detail, from another perspective, in the sequel, \emph{The Funakura Invariant and Norm Relations}.}

My main sources for this material are  \cite{wolf2011}, \cite{biasse2020norm}, and \cite{suzuki1955}.
I have also incorporated some insights of Allcock~\cite{allcock2018} which makes several improvements to Wolf \cite{wolf2011}.
I have also consulted several standard works on finite group theory.
In fact, a major purpose of this report is to synthesize the information and results from  these various sources
 to give a continuous, self-contained development of the subject. 
 There is still work to be done in this synthesis, but I think this report is a good start.
 I have tried to simplify and polish proofs whenever possible, and introduce new points of view, terminology, and results whenever they help the overall narrative.
New proofs are given to several of the results of \cite{wolf2011} 
with the goal of making the proofs of these results more accessible. For example Wolf uses Burnside's theorem
concerning Sylow subgroups central in their normalizers, where my proofs avoid this technique. I also avoid
transfer techniques more generally, and I avoid the Schur-Zassenhaus theorem (used in~\cite{allcock2018}). I explore Sylow-cyclic groups in more detail than Wolf and gives some results not found in Wolf~\cite{wolf2011} which might be new: for example 
Theorem~\ref{divisor_thm} concerning the conjugacy of subgroups of the same order may be new.
The characterization of freely representable groups in terms of the existence of unique subgroups of each prime order (see Theorem~\ref{sc_fr_thm}) may be new. I make more use of the idea of a maximum cyclic conjugate (MCC) subgroup, which leads
to a different perspective from Wolf, and the resulting theorem on the existence of a unique element of order 2
may be new (Corollary~\ref{new_corollary2}). 
The results from~\cite{biasse2020norm} cited here on norm relations are given more elementary proofs which use less representation
theory (for example, the proof here does not use central primitive idempotents), and they are generalized a bit.
In general the paper~\cite{biasse2020norm} has inspired me to use norm relations as a tool to study freely representable groups.
For example, I have caste the proof of the $pq$-theorem (Theorem~\ref{pq_thm}) using norm relations, which seems to be novel point
of view, and have used norm relations in a few other places  (Lemma~\ref{new_lemma} and Theorem~\ref{new_thm}) to 
show that groups are not freely representable.

I believe this report is reasonably self-contained at least for the solvable case; appendices have been given to help
make the report more accessible.
The non-solvable case requires me to draw from some high-level sources such as Suzuki~\cite{suzuki1955}.
I hope to give complete proofs even in the non-solvable case in a sequel or future edition of this report.


\chapter{Definitions and Examples}

Suppose a group $G$ acts on a set $S$. We say that $G$ \emph{acts freely} if for all $g \ne 1$ in $G$
and all $s \in S$ we have $g s \ne s$. In other words, no nonidentity element has a fixed point.
Such actions occur naturally, for example, in the theory of covering spaces.

Suppose a group $G$ acts linearly on a vector space $V$. We say that the action is a \emph{free
linear
representation}  if  $V\ne 0$ and if $G$ acts freely on the set of nonzero vectors of $V$. 
In other words, for all $g\in G$, the associated linear transformation~$V\to V$ has eigenvalue $1$ if and only if $g=1$.

Here are a few observations related to this definition:

\begin{lemma}\label{subrep_lemma}
If $G$ is a free linear representation of $V$ then any nonzero subrepresentation is also free.
\end{lemma}

\begin{lemma}\label{faithful_lemma}
A free linear representation of $V$ over $F$ is faithful. 
In other words, it yields an injective homomorphism $G \hookrightarrow \mathrm{GL}(V)$.
\end{lemma}

\begin{proof}
Suppose $g\in G$ is in the kernel of $G \to \mathrm{GL}(V)$. Let $v \ne 0$ be in $V$. Then~$g v = v$ since $g$ is in the kernel. So $g = 1$
since $G$ acts freely on nonzero vectors.
\end{proof}

A  group $G$ is \emph{freely representable} over a field $F$ if it possesses a free linear representation of an $F$-vector space $V$.
In this report we focus on finite groups only, so it is understood that all freely representable groups we consider are finite.
This allows us to restrict our attention, when convenient, to finite dimensional vector spaces:

\begin{lemma}\label{finite_dim_lemma}
Let $G$ be a finite group and let $F$ be a field. If $G$ is 
freely representable over~$F$ then there is a free linear representation of $G$ on a finite dimensional vector space $V$.
In fact we can choose $V$ to be of dimension at most $|G|$.
\end{lemma}

\begin{proof}
By assumption $G$ has a free linear representation on some vector space $W$.
By definition, $W$ is not the zero space; let $w \in W$ be a nonzero vector
and let $V$ be the span of $\{ \sigma w \mid \sigma \in G \}$. 
The result now follows from Lemma~\ref{subrep_lemma}
\end{proof}

\begin{corollary}\label{irreducible_cor}
Let $G$ be a finite group and let $F$ be a field. Then $G$ is 
freely representable over~$F$ 
if and only if there is an irreducible free linear representation of $G$ on a finite dimensional $F$-vector space $V$.
\end{corollary}

Combining Lemma~\ref{finite_dim_lemma} with Lemma~\ref{faithful_lemma} yields the following:

\begin{corollary}\label{finite_dim_cor}
Let $G$ be a finite group and let $F$ be a field. Then $G$ is freely representable over $F$ if and only if
it is isomorphic to a subgroup $\Gamma$ of $\GL_n(F)$, for some positive integer $n$, such that
the only element of $\Gamma$ with eigenvalue $1$ is the identity. 
\end{corollary}

Note we can even take $n$ be to at most $|G|$ in the above corollary if we wish.

If we do not specify $F$ it is understood that~$F$ is $\bC$. 
This default is reasonable since $\bC$ is in some sense the most basic field 
used in representation theory. But note that free representations over $\bR$ are central to the theory
of Riemannian manifolds with constant positive curvature, so one might argue that $\bR$ could have been a reasonable default.
The following proposition shows this question is moot:  there is no distinction between being freely representable over $\bC$
and being freely representable over $\bR$.

\begin{lemma}\label{all_char_lemma}
If a finite group $G$ is freely representable over a field $F$ then it is 
freely representable over all fields of the same characteristic at that of $F$.
\end{lemma}

\begin{proof}
Let $F$ be a field and let $F_0$ be its prime subfield. If we have a free representation of $G$ on an $F$-vector space $V$
then observe that it is a free representation of~$V$ regarded as an~$F_0$-vector space (where we allow the dimension to increase).

Conversely, suppose $G$ is freely representable over $F_0$. By Corollary~\ref{finite_dim_cor}
we can assume that $G$ is a finite subgroup of $\GL_n(F_0)$ where $n \ge 1$ and such that
the only element of $G$ with eigenvalue $1$ is the identity. In other words,
$\det (g - I) \ne 0$ for all~$g \in G$ not equal to the identity.
Note that $\GL_n(F_0)$ is a subgroup of~$\GL_n(F)$, and the determinant of $g - I$ with $g \in \GL_n(F_0)$
is the same whether we take determinants in $\GL_n(F_0)$ or in $\GL_n(F)$.
Thus, by Corollary~\ref{finite_dim_cor}, $G$ is freely representable over~$F$.
\end{proof}

Now we consider some basic examples that will prove to be central in what follows:

\begin{example} \label{cyclic_example}
Every cyclic group is freely representable.
To see this let $G$ be the group of $N$th roots of unity. This group  $G$ acts on $V=\bC^n$ by scalar multiplication
for all $n \ge 1$. This construction gives a free representation over~$\bC$ of dimension~$n$.
Now if we view $V=\bC^n$ as a real vector space, we get
instead a free representation over~$\bR$ of dimension~$2n$.
We can also find free representations  of a cyclic group $G$ over $\bR$ directly by considering finite groups of rotations of $\bR^2$
fixing the origin. Of course this is equivalent, in the case $n=1$, to the early construction.

\end{example}

\begin{example}\label{quaternion_example}
We can extend the stock of examples in a very interesting way by replacing $\bC$ in Example~\ref{cyclic_example} with the quaternions~$\bH$.
Let $G$ be a finite multiplicative subgroup of~$\bH^\times$. Then $G$ acts on $\bH^n$ by scalar multiplication.
Now identify $\bH^n$ with~$\bR^{4n}$ in the usual way, and observe we get a $4n$-dimensional
free representation over $\bR$.
(See the relevant appendix for more information on the division algebra~$\bH$.)

We conclude that every finite subgroup of $\bH^\times$  is a freely representable group.
In particular, the quaternion group is a freely representable noncyclic group with~8 elements.

The classification of  finite subgroups of $\bH^\times$ is easy to describe
in terms of finite rotations groups of $\bR^3$.
First observe that any finite subgroup of $\bH^\times$ sits 
in the group of elements $\bH_1$  of norm~$1$, which is topologically the $3$-sphere.\footnote{It turns out that the group $\bH_1$
is isomorphic to the Lie group $\mathrm{SU}(2)$. We do not require this fact in this report.}
There is a two-to-one surjective homomorphism $\bH_1 \to \mathrm{SO}(3)$
which sends~$h \in \bH_1$ to the rotation
$$
v \mapsto h v h^{-1}
$$
where $v \in \bR^3$ and where $\bR^3$ is identified with the  $\bR$-span of
the quaternions $\mathbf{i}, \mathbf{j}, \mathbf{k}$.
(See the appropriate appendix for more information).
The kernel is $\left\{\pm 1\right\}$. 

This two-to-one mapping $\bH_1 \to \mathrm{SO}(3)$  can be used to classify finite subgroups of~$\bH_1$ in terms of finite subgroups of $\mathrm{SO}(3)$.
The fact that $-1$ the only element of~$\bH_1$ of order 2 helps us describe the correspondence (see Lemma~\ref{doublecover_lemma}).

First consider any finite subgroup $G$ of $\bH_1$ of \emph{odd order}. Then the restriction of~$\bH_1 \to \mathrm{SO}(3)$
to $G$ has trivial kernel, so $G$ is naturally isomorphic to a subgroup of $\mathrm{SO}(3)$ of odd order.
All finite subgroups of $\mathrm{SO}(3)$ of odd order are cyclic, so
$G$ must be be cyclic in this case. Since $\bH_1$ contains the unit circle subgroup of $\bC^\times$ as a subgroup, we
have cyclic subgroup of all odd orders in $\bH_1$.

We can form finite subgroups $G$ of $\bH_1$ of even order by taking preimages of finite subgroups of $\mathrm{SO}(3)$.
Conversely, every finite subgroup $G$ of $\bH_1$ of even order 
must contain $-1$ since $-1$ is the unique element of $\bH_1$ of order 2. This implies that such a $G$ is the preimage of
its image (see Lemma~\ref{doublecover_lemma} for details).
So to classify the even order subgroups of $\bH_1$ we can just look at preimages of finite subgroups of $\mathrm{SO}(3)$.
See the appropriate appendix for a classification of subgroups of~$\mathrm{SO}(3)$. 

We start the classification of even ordered subgroups of $\bH_1$ by looking at the preimage $L$ of a cyclic subgroup $C_k$
of order $k$. The preimage of a group isomorphic to $C_k$ is isomorphic to $C_{2k}$ (Lemma~\ref{doublecover_lemma}).
Next we consider the preimage of noncyclic subgroups of $\mathrm{SO}(3)$. 
Observe that the preimage of a noncyclic group cannot be a cyclic group.
So the classification of noncyclic finite subgroups of $\mathrm{SO}(3)$ now gives us a classification 
of noncyclic finite subgroup of $\bH_1$.
Every noncyclic finite subgroup of $\bH_1$ is the preimage of one of the following: a dihedral group $D_m$ of order~$2m$ where~$m\ge 2$, the tetrahedral group $T$ (which is isomorphic to~$A_4$), the octahedral group $O$ (which is isomorphic to $S_4$),
or the icosahedral group $I$ (which is isomorphic to~$A_5$).
We call $G$ a \emph{binary dihedral group, binary tetrahedral group, binary octahedral group}, or \emph{binary iscosahedral group}
depending on its image in~$\mathrm{SO}(3)$.\footnote{The binary tetrahedral group is isomorphic to $\mathrm{SL}_2(\bF_3)$,
and the binary icosahedral group is isomorphic to~$\mathrm{SL}_2(\bF_5)$.}
In this classification we include $D_2$, the dihedral group of order 4; this is just the Klein four group $C_2 \times C_2$. The
quaternion group of size 8 has such a $D_2$ as its image.
\end{example}

\begin{remark}
With a little work, we can show that binary dihedral groups of equal order are isomorphic, and up to isomorphism there
is a unique binary tetrahedral group, a unique binary octahedral group, and a binary icosahedral group. (See the following proposition).
We denote these groups as $2 D_n, 2 T, 2 O,$ and $2 I$.
\end{remark}

\begin{proposition} \label{binary_conjugate_prop}
Let $G_1$ and~$G_2$ be subgroups of $\bH_1$ of the same order
and whose images under 
the standard map $\pi \colon \bH_1 \to \mathrm{SO}(3)$
are isomorphic. Then $G_1$ and $G_2$ are isomorphic. In fact, they are conjugate
subgroups of~$\bH_1$.
\end{proposition}

\begin{proof}
We take it as known that isomorphic subgroups of $\mathrm{SO}(3)$ are conjugate subgroups  of~$\mathrm{SO}(3)$.
Thus $\pi [G_2 ] = \gamma^{-1} \pi [G_1] \gamma$ for some $\gamma \in  \mathrm{SO}(3)$.
Now choose an element~$h \in \bH_1$ mapping to $\gamma$ and our goal is to show that $G_2 = h^{-1} G_1 h$.

We start with the case where $G_1$ and $G_2$ have even order.
Since $G_2$ and $h^{-1} G_1 h$ have the same order it is enough to show 
the inclusion $G_2 \subseteq h^{-1} G_1 h$. 
So let $g_2 \in G_2$.
Then $\pi(g_2) = \gamma^{-1} \pi(g_1) \gamma$ for some $g_1 \in G_1$.
Observe that  $g_2$ and~$h^{-1} g_1 h$ have the same image, so $g_2^{-1} h^{-1} g_1 h$ is in the kernel of $\pi$. 
But the kernel of $\pi$ is $\{ \pm 1 \}$.
So
$$
g_2 = h^{-1} (\pm g_1) h.
$$
But $\pm g_1 \in G_1$ since $-1\in G_1$ (since $G_1$ has even order). Thus $G_2 \subseteq h^{-1} G_1 h$ as desired.

The remaining case is where $G_1$ and $G_2$ are odd of order $k$. Let $G'_i$ be the preimage of $\pi [G_i]$.
In this case $G_i$ and~$\pi[G_i]$ are cyclic of order $k$ and so the~$G'_i$ are cyclic of order $2k$. 
By the above argument, $G'_1$ and $G'_2$
are conjugate groups. Since $G_i$ is the unique subgroup of $G_i'$ of index 2, it follows that $G_1$ and~$G_2$ must be
conjugate as well.
\end{proof}

The  facts we used about the
correspondence between subgroups and elements of $\bH_1$ and  their images in $\mathrm{SO}(3)$
are part of a more general phenomenon:

\begin{lemma} \label{doublecover_lemma}
Let $\pi\colon G \to M$ be a surjective homomorphism between groups with kernel $K$ of size $2$.
Then the following hold:
\begin{enumerate}
\item
The map $\pi\colon G \to M$ is a $2$-to-$1$ map: the preimage of each $t\in M$ has size $2$.

\item
The preimage in $G$ of any finite subgroup $L$ of $M$ is a finite subgroup of $G$
of order twice the order of $L$.

\item
Every finite subgroup $H$ of $G$ of odd order is isomorphic to its image in $M$.

\item
If $g \in G$ has odd finite order then $g$ and $\pi(g)$ have the same order.

\end{enumerate}
Furthermore if $G$ has a unique element of order $2$ then also 
\begin{enumerate}
\setcounter{enumi}{4}
\item
Every finite subgroup $H$ of $G$ of even order is the preimage of its image~$\pi[H]$.

\item
Every finite subgroup $H$ of $G$ of even order has order twice that of its image~$\pi[H]$.

\item
If $g \in G$ has even finite order then $\pi(g)$ has order one-half the order of $g$.

\item
Let $L$ be a  finite cyclic subgroup of $M$ of odd order $k$. 
Then its preimage in~$G$ is a cyclic subgroup of order $2k$.
\item
Let $L$ be a finite cyclic subgroup of $M$ of odd order $k$.
Then there is a unique subgroup of $G$ of order $k$ whose image in $M$ is $L$.
\end{enumerate}
\end{lemma}

\begin{proof}
Suppose $t \in M$ is given, and let $g \in G$ map to $t$. Then $g'$ maps to $t$ if and only if $g' = h g$ for some $h \in K$.
Since $K$ has order 2, there are two such elements~$g'$.
So $G \to M$ is two-to-one. From this (1) and (2) follow.

If $H$ is a subgroup of $G$ then the restriction of $G \to M$ to $H$ has kernel $H \cap K$.
If~$H$ is finite of odd order, then $H\cap K$ must be the trivial group since it is a subgroup of $K$.
Thus $H$ is isomorphic to its image in $M$ and (3) follows.

Note that if $g \in G$ then $\pi(g)$ generates the image of $\left< g \right>$ in $M$.
If in addition $g$ has finite odd order then,
by (3), $\left< g \right>$ and $\left< \pi(g) \right>$ are isomorphic, so $g$ and $\pi(g)$ have the same order.
Thus~(4) holds.

If $H$ is a finite subgroup of $G$ of even order then it has an element of order 2 by Cauchy's theorem. Thus $H$ contains $K$ since there is a unique
element of order 2.
The kernel of the restriction of $G \to M$ to~$H$ has kernel $H \cap K$, which in this case is $K$ itself.
Since the image $\pi[H]$ of $H$ is isomorphic to $H/K$ (first isomorphism theorem) this implies that $H$ has twice the order of $\pi[H]$.
By (2) the preimage $H'$ of $\pi[H]$ also has order twice that of $\pi[H]$. Since $H \subseteq H'$, this implies that $H = H'$.
Thus (5) and (6) hold.

Note that if $g \in G$ then $\pi(g)$ generates the image of $\left< g \right>$ in $M$.
If in addition~$g$ has finite even order then, by (6), 
$\left< g \right>$ and has twice the size of $\left< \pi(g) \right>$, so $g$ has order twice that of $\pi(g)$.
Thus~(7) holds.

Let $L$ be cyclic subgroup of $M$ of odd order $k$. By $(6)$ its preimage $H$ has order~$2k$.
Let $g\in H$ map to a generator of $L$. If $g$ has even order then $g$ has order~$2 k$ by (7), so $H$ is cyclic.
If $g$ has odd order then it has order $k$ by (4). Let $\tau$ be the unique element of order $2$ in $G$.
Then $g \tau g^{-1} = \tau$ by uniqueness, so~$g \tau = \tau g$. 
Note that $g$ and $\tau$ generate $H$, and so~$H$ is~Abelian.
This means that $\tau g$ has
order~$2k$. Thus $H$ is cyclic. So (8) holds.

Finally, let $L$ be cyclic subgroup of $M$ of odd order $k$. 
By (8), the preimage~$H$ of $L$ in $G$ is cyclic of order $2k$, and has a unique
subgroup $C$ of order $k$. The image of $C$ is $L$ by (3). Since any subgroup of $G$ with image $L$ must
be a subgroup of~$H$, this means that $G$ has a unique subgroup of order $k$ whose image is $L$.
So (9) holds.
\end{proof}

\begin{example}\label{divisionring_example}
We can extend Example~\ref{quaternion_example} from $\bH$ to a general division ring $D$.
Suppose~$G$ is a finite group of $D^\times$.
The prime subfield $F$ of $D$ is either $\bQ$ or $\bF_p$ for some prime $p$. In the first case we say that~$D$ has characteristic~0.
If the prime field is finite then we define the characteristic of $D$ to be the size of the prime field.
Note that $F$ is in the center of $D$ in the sense that $a v = v a$ for all $a \in F$ and~$v \in D$. This follows
from the observation that, for each nonzero $v \in D$, the map $a \mapsto v^{-1} a v$ is a ring homomorphism
$F \to D$ that maps $1$ to $1$. By definition of $F$ and properties of ring homomorphisms this map
must send any $a\in F$ to itself.

Observe that $D$ is an $F$-vector space. Observe also that multiplication defines a linear representation of $G$ on $V = D$.
Suppose $g \in G$ and $v \in D$. If $g\ne 1$ and if~$g$ has a fixed vector $v \in D$ then $g v = v$.
So $(g-1) v = 0$. Since $D$ is a division ring and since $g-1 \ne 0$, we have $v = 0$. 
Thus we get a free linear representation of $G$ on the~$F$-vector space $D$.
In other words, $G$ is freely representable over $F$. In characteristic $0$ we  conclude that $G$ is freely representable
in the usual sense (for fields of characteristic zero).

What about if $F$ has finite prime characteristic?  In that case the subset $R$ of finite linear combinations $\sum a_i g_i$ with $a_i \in F$ and $g_i \in G$
is closed under addition and multiplication. In fact $R$ forms a subring of $D$.
Observe that~$R$ is finite. Since~$R$ has no zero divisors, it must be a finite field by Wedderburn's theorem. It is well-known that $R^\times$
is a cyclic group if $R$ is a finite field, and since $G$ is a subgroup of~$R^\times$ it is also cyclic. So $G$ is freely representable (over $\bC$)
in this case as well.\footnote{The fact that $D^\times$ is cyclic in prime characteristic was noticed by Herstein in 1953.
Herstein then conjectured for in any characteristic that any subgroup of $D^\times$ of odd order is cyclic, because this holds in cases such as~$D = \bH$
and for $D$ of prime characteristic.
In 1955, Amitsur~\cite{amitsur1955} found a counter-example to the conjecture of size 63.}

So \emph{all finite subgroups of division rings are freely representable}, and in finite prime characteristic they are actually cyclic.\footnote{Here we used
the fact that (1) $F^\times$ is cyclic if $F$ is finite, and (2) Wedderburn's theorem. We will give 
independent arguments for these facts later in the document as a consequence of the structure theorem for Sylow-cyclic groups. See Section~\ref{Wedderburn} below.}
\end{example}

\chapter{Norm Relations of Unity}

Norm relations of unity are of  interest in algebraic number theory, both for theoretic reasons and for computational reasons. See \cite{biasse2020norm} for a discussion on their history and applications to algebraic number theory.
One of my motivations for studying freely representable groups is the  fact that such group are
exactly the finite groups without norm relations of unity (a result of \cite{biasse2020norm}, and a proof is provided in this section as well).

\begin{definition}
Let $G$ be a finite group and let $H$ be a subgroup. Let $R$ be a commutative ring (with unity). 
The \emph{norm of~$H$} 
is defined to be the formal sum of elements of $H$ in $R[G]$:
$$
\N H \; \defeq \; \sum_{\sigma\in H} \; \sigma.
$$
\end{definition}

If we fix a linear representation of a group $G$ on an $F$-vector space $V$ then we can view
 $V$ as an $F[G]$-module.
 If in addition the representation is a free linear representation
then norms of nontrivial subgroups have the interesting property that they 
annihilate $V$.

\begin{proposition}
Let $G$ be a finite group and let $F$ be a field.
If $G$ has a free linear representation on an $F$-vector space~$V$,
and if $H \ne \{1\}$ is a  subgroup of $G$, then~$(\N H) v = 0$ for all $v\in V$.
\end{proposition}

\begin{proof}
Let $v \in V$, and consider $v' = (\N H) v$. Since $H$ is nontrivial it has a nonidentity
element $\sigma$. Observe
$$
\sigma v' = \sigma  ((\N H) v) = (\sigma \N H) v = (\N H) v = v'.
$$
Since $G$ acts freely on nonzero vectors we have $v' = 0$. Thus $(\N H) v= 0$ for all vectors $v$ in $V$.
\end{proof}

 \begin{definition}
 Let $G$ be a finite group.
A \emph{norm relation of unity} for $G$ is an expression in $\bQ[G]$ of the form
$$
\mathbf{1} = \sum_{H \in \mathcal H} a_H \, (\N H) \, b_H
$$
where $\mathcal H$ is the collection
of nontrivial subgroups of $G$,  where $a_H, b_H \in \bQ[G]$, 
and where $\mathbf{1}$ is the unit in $\bQ[G]$ which can be viewed as the norm of the trivial group.

If we replace $\bQ$ in the above with a field $F$, then we call such a relation a \emph{norm relation of unity
relative to $F$}.
\end{definition}

\begin{example}
The first norm identity of unity that came to my attention was one exploited by my colleague S.~Sharif and heavily used by
his student A.~Savage in his thesis~\cite{savage2019}. This  norm relation involves the group $G = C_3 \times C_3$.
In addition to~$\{1\}$ and $G$ itself, there are four subgroups $H_1, H_2, H_3, H_4$ of $G$ which are all cyclic of order~3.
In $\bZ[G]$ we have the following (as in the proof of Theorem~\ref{pq_thm} below):
$$
3\cdot \mathbf{1} = \N H_1 + \N H_2 + \N H_3 +  \N H_4 - \N G
$$
giving us a simple norm relation of unity in $\bQ[G]$:
$$
\mathbf{1} = \frac{1}{3} \N H_1 +  \frac{1}{3} \N H_2 + \frac{1}{3} \N H_3 + \frac{1}{3} \N H_4 - \frac{1}{3} \N G.
$$
The above relation in $\bZ[G]$ was then applied to bicubic Galois extensions $K / \bQ$ with Galois group identified with~$G$. 
The above additive relation yields a multiplicative relation for~$\alpha \in K$:
$$
\alpha^3 = \frac{N_{K/K_1} (\alpha) N_{K/K_2} (\alpha) N_{K/K_3} (\alpha) N_{K/K_4} (\alpha) }{N_{K/\bQ} (\alpha)}
$$
where $N_{K/L}$ is the usual norm of field theory, and where $K_1, \ldots, K_4$ are the intermediate fixed fields
associated, as in Galois theory, with the subgroups $H_1, \ldots, H_4$.
Sharif and Savage were interested in the case where $\alpha$ is a unit in the ring of integers~$\mathcal O_K$.
For such $\alpha$ the equation becomes
$$
\alpha^3 = \pm {N_{K/K_1} (\alpha) N_{K/K_2} (\alpha) N_{K/K_3} (\alpha) N_{K/K_4} (\alpha) }.
$$
This identity helps one to identify units in $\mathcal O_K$ given a predetermination of units of each~$\mathcal O_{K_i}$.
See also \cite{biasse2020norm} for generalization of this identity and applications to finding units and other invariants of~$K$ for more general number fields~$K$.
\end{example}

\begin{example}

A  earlier example, from 1966, can be found in Wada~\cite{wada1966}
for the case where $G = C_2 \times C_2$ is the Klein four group. Let $\sigma_1, \sigma_2, \sigma_3$ be the 
nontrivial elements of $G$ and let $H_i$ be the subgroup generated by $\sigma_i$. Then we have 
$$
2 \cdot \mathbf{1} = \N H_1 + \N H_2 - \sigma_1  \N H_3
$$
which, when divided by $2$, yields a  norm relation of unity. In particular if $K$ is a biquadratic extension of $\bQ$, say,
with Galois group identified with $G$, and if~$K_1, K_2, K_3$ are the quadratic extensions associated with $H_1, H_2, H_3$,
then we get
$$
\alpha^2 =  \frac{N_{K/K_1} (\alpha) N_{K/K_2} }{\sigma_1 N_{K/K_3} (\alpha)}.
$$
As in the previous example this was used to describe units in $\mathcal O_{K_1}$ which was then used
to calculate class numbers of such biquadratic fields~$K$.
\end{example}

\begin{example}
A third example can be found in Parry~\cite{parry1977} for $G = C_3 \times C_3$. 
As in the earlier example, let~$H_1, H_2, H_3, H_4$ be the distinct cyclic subgroups of order~3.
Fix a  generator $\sigma$ of~$H_1$ and a  generator~$\tau$ of $H_2$. 
Switching $H_3$ and $H_4$ if necessary, we
can assume that $H_3$ is generated by~$\sigma \tau$ and that $H_4$ is generated by $\sigma \tau^2$. 
In $\bZ[G]$ we have the following:
$$
3\cdot \mathbf{1} = \N H_1 + \N H_2 + \N H_3 - (\sigma + \sigma \tau) \N H_4
$$
which gives a  norm identity of unity when we divide by $3$. Parry used these to study bicubic extensions of $\bQ$.
\end{example}

\begin{lemma} \label{normofunity_lemma}
Let $G$ be a finite group and let $\mathcal H$ be the collection of nontrivial subgroups of $G$. Then the following are equivalent:
\begin{enumerate}
\item
$G$ has a  norm relation of unity relative to~$F$.
\item
The two-sided ideal of $F[G]$ generated by $\left\{ \N H \mid H \in \mathcal H \right\}$ is all of $F[G]$.
\item
The left ideal of $F[G]$ generated by $\left\{ \N H \mid H \in \mathcal H \right\}$ is all of $F[G]$.
\item
The right ideal of $F[G]$ generated by $\left\{ \N H \mid H \in \mathcal H \right\}$ is all of $F[G]$.
\end{enumerate}
\end{lemma}

\begin{proof}
Clearly $(1) \iff (2)$, $(3) \implies (2)$, and $(4) \implies (2)$.
Next we note that the implication $(2) \implies (3)$ is a consequence of the following fact:
\emph{Let $R$ be a commutative ring and let $I$ be the left ideal of $R[G]$ generated by $\left\{ \N H \mid H \in \mathcal H \right\}$.
Then $I$ is a two-sided ideal of~$R[G]$.}
To establish this claim it is enough to show that~$(\N H) g \in I$ for all $H \in \mathcal H$ and $g \in G$. This
follows from the equation 
$$
(\N H) g = g g^{-1} (\N H) g = g \N (g^{-1} H g ).
$$
Similarly, $(2) \implies (4)$.
\end{proof}

\begin{theorem} \label{biasse2020norm_thm}
Let $G$ be a finite group and let $F$ be a field of characteristic not dividing~$|G|$. 
Let $I$ be the two-sided ideal of $F[G]$ generated by the norms $ \N H$
of nontrivial subgroups of $G$. Then $G$ is freely representable over $F$ if and only if $I$ is a proper ideal of $F[G]$.
\end{theorem}

\begin{proof}
Suppose that $G$ has a free linear representation on the $F$-vector space $V\ne 0$.
Let $K$ be the kernel of the action of $F[G]$. In other words, $K$ is the collection of all elements $\alpha \in F[G]$
such that $\alpha v = 0$ for all $v\in V$. Observe that $K$ is a two-sided ideal of $F[G]$.
By the above proposition, $\N H \in K$ for all nontrivial subgroups $H$.
Thus $I \subseteq K$. However, if $v\ne 0$ then $\mathbf{1} v \ne 0$. So $\mathbf{1}$ is not in $K$.
Thus $K$, and hence~$I$, must be a proper ideal of $F[G]$.

Conversely, suppose that $I$ is a proper ideal of $F[G]$. Let $V = F[G]/I$. Although~$V$ is an algebra
over $F$, we will think of it just as an $F[G]$-module; in other words, $V$ is an $F$-vector space with
a linear action of $G$. Since $I$ is a proper ideal of $F[G]$, the vector space $V$
is not the zero space. For convenience let's write~$[\alpha]$ for the coset $\alpha + I$ in $V$.

The goal is to show that the linear action of $G$ on $V$ is free. So
suppose that~$g\ne 1$ in $G$ and suppose that $g [\alpha] = [\alpha]$ for some $[\alpha] \in V$. 
Let~$C$ be the subgroup of~$G$ generated by $g$.
So $(\N C) [\alpha] = k [\alpha]$ where $k$ is the size of~$C$. 
We can rewrite this as~$[\N C \cdot \alpha] = k [\alpha]$ where now we think of $k$ is an element of~$F$.
 Since~$\N C \cdot \alpha \in I$, we have that $k [\alpha] = [0]$.
Since~$k$ is invertible in $F$ we have~$[\alpha] = [0]$. This show that the representation of $G$  on~$V$
is a free linear representation.
\end{proof}

As a corollary we get the following:

\begin{theorem} \label{biasse2020norm_thm2}
A finite group $G$ has a  norm relation of unity if and only if it is not freely representable.
\end{theorem}

\begin{proof}
Let $\mathcal H$ be the collection of nontrivial subgroups of $G$. Let $I$ be the two-sided ideal of $\bQ[G]$ generated
by $\{ \N H \mid H \in \mathcal H\}$.

Suppose $G$ has a  norm relation of unity. By Lemma~\ref{normofunity_lemma}, $I$ is all of~$\bQ[G]$.
 By Theorem~\ref{biasse2020norm_thm}, $G$ is not freely representable over~$\bQ$,
 hence is not freely representable over~$\bC$ (Lemma~\ref{all_char_lemma}).

Conversely, suppose $G$ is not freely representable. In other words, suppose~$G$ is not freely
representable over $\bC$, and hence over $\bQ$ (Lemma~\ref{all_char_lemma}).
By Theorem~\ref{biasse2020norm_thm} the ideal $I$ must be all of $\bQ[G]$.
So $G$ has a  norm relation of unity by Lemma~\ref{normofunity_lemma}.
\end{proof}

\chapter{Basic Properties of Freely Representable Groups}

We start with some important properties of freely representable groups that can be proved
without too much effort. We will see that the Abelian groups that are freely representable are just the cyclic groups. And we
will see various  ways in which freely representable groups behave like cyclic groups, confirming
the idea that freely representable groups can play the same role for finite groups in general that the cyclic groups
play for the finite Abelian groups. In other words, they can be regarded as ``cycloidal'' groups.

For example, every subgroup of a cyclic group is a cyclic group. The same holds for freely representable groups:

\begin{proposition} \label{freely_representable_subgroup_prop}
Every subgroup of a freely representable group is freely representable.
\end{proposition}

\begin{proof}
This follows from the definition.
\end{proof}

\begin{lemma}\label{abelian_lemma}
If $A$ is a freely representable finite Abelian group then $A$ is cyclic.
\end{lemma}

\begin{proof}
(We will give here a proof that uses a bit of representation theory over~$\bC$ and the fact that 
finite subgroups of $\bC^\times$ are cyclic. However,  more algebraic arguments can be given that do not use $\bC$ in any essential way. See the remark after 
Corollary~\ref{pq_cor}.)

A freely representable group has an irreducible free linear representation over~$\bC$ (Corollary~\ref{irreducible_cor}).
Every irreducible linear representation of an Abelian group is one-dimensional and so gives a homomorphism into the circle group of $\bC^\times$ (see the appendix on group representations).
Also, free linear representations are necessarily faithful (Lemma~\ref{faithful_lemma}). So any freely representable finite Abelian group is isomomorphic to a subgroup
of the circle group in~$\bC^\times$. Such groups are cyclic.
\end{proof}

Combining this lemma with Example~\ref{cyclic_example} gives the following:

\begin{corollary}\label{abelian_cor}
A finite Abelian group $A$ is freely representable if and only if it is cyclic.
\end{corollary}

Combining the above lemma with Proposition~\ref{freely_representable_subgroup_prop} yields the following:

\begin{corollary}\label{abelian_cor2}
If a finite group~$G$ is freely representable then all its Abelian subgroups are cyclic.
\end{corollary}

\begin{proof}
Every subgroup of $G$ is freely representable, so every Abelian subgroup is cyclic
by the previous corollary.
\end{proof}

\begin{remark}
In Example~\ref{divisionring_example} we observed that if $F$ is a division ring then any finite subgroup of $F^\times$ is freely representable.
Thus all finite subgroups of $F^\times$ are cyclic if~$F$ is a field by Lemma~\ref{abelian_lemma}.
If we use the above proof of Lemma~\ref{abelian_lemma} then the argument is  somewhat circular
since it uses the result for $F = \bC$. Also, in Example~\ref{divisionring_example} we used the result for finite $F$.
However, the remark after Corollary~\ref{pq_cor} gives an argument for Lemma~\ref{abelian_lemma} that does not assume the result for $F = \bC$, and
in Section~\ref{Wedderburn} we will give an independent argument for the result that $F^\times$ is cyclic without assuming
the result for finite~$F$.
\end{remark}

Recall that finite subgroups of $\bH^\times$ give examples of freely representable groups, and all such 
 groups all have at most one element of order 2.
This occurs more generally. (Note also that this is a basic property of cyclic groups as well).

\begin{proposition}\label{even_prop}
Suppose that $G$ is a finite group of even order that is freely representable over a field~$F$.
Then~$G$ has a unique element of order two, and
this element is in the center of $G$.
\end{proposition}

\begin{proof}
By Corollary~\ref{finite_dim_cor} we can think of $G$ as a subgroup of $\mathrm{GL}_n(F)$
such that
the only element of $\Gamma$ with eigenvalue $1$ is the identity. 
Since $G$ has even order it has at least one element of order $2$ (Cauchy's theorem).
We will prove the result by showing that if $g \in \mathrm{GL}_n(F)$ has order $2$
and does not have eigenvalue $1$ then~$g = - I$.

To see this, observe that $(g-I)(g+I) = 0$ since $g$ has order $2$. Also observe that the null-space of $g-I$ is $\{0\}$ since $g$ fixes only the
zero vector. Thus $g-I$ is invertible, and so $g+I = 0$. In other words, $g = - I$ as desired.
\end{proof}

\begin{remark}
The above proof can be adapted to prove the following: if $F$ has characteristic $2$ then every freely
representable group over $F$ is of odd order.
\end{remark}

\begin{remark}
By the above proposition, every freely representable group of even order has a normal subgroup of order 2.
The Feit-Thompson theorem states that every non-Abelian simple group has even order.
So the above proposition implies that a simple group is  freely representable if and only if
it is cyclic (and necessarily of prime order).
\end{remark}

\begin{remark}
The above proposition implies that many familiar non-Abelian groups, such as the $A_n$ if $n \ge 4$
or $S_n$ if $n \ge 3$ are not freely representable. Similarly dihedral groups are not freely representable.
\end{remark}

\begin{remark}
The quaternion group with 8 elements is the classic example of a non-Abelian group with a single element of
order two. Later we will see that every $2$-group with only one such element is freely representable,
and will form a well-known class of groups called the \emph{generalized quaternion groups}.
\end{remark}

We give another class of groups where freely representable and cyclic coincide.

\begin{theorem}\label{pq_thm}
Let $G$ be a group of order equal to the product of at most two primes where the two primes in question can be equal.
Then $G$ is freely representable if and only if it is cyclic. 
\end{theorem}

\begin{proof}
One direction is clear from previous results, thus it is enough to show that if $G$ not cyclic
then it is not freely representable. 
So suppose that $G$ is not cyclic. This means that all the nontrivial cyclic subgroups of~$G$ have prime order, and
every nonidentity element of $G$ is in a unique cyclic group.
 Let $\mathcal C$
be the collection of nontrivial cyclic subgroups of~$G$. Then every nonidentity element of~$G$ is in exactly
one $C \in \mathcal C$. So
$$
\sum_{C \in \mathcal C} \N C = (k-1) \mathbf1 + \N G
$$
where $k$ is the size of $\mathcal C$.
This yields a norm relation of unity since~$k>1$.
So $G$ is not freely representable by Theorem~\ref{biasse2020norm_thm2}
\end{proof}

\begin{remark}
The above theorem cannot be extended to all groups whose size is a product of three primes: we have seen that the 
quaternion group with $8 = 2^3$ elements is freely representable, but not cyclic.
\end{remark}

\begin{remark}
The proof of the above theorem generalizes as follows: suppose $G$ is the union of two or more nontrivial proper subgroups 
such that when you remove the identity element from each you get a partition of~$G - \{e\}$. Then $G$ has a norm relation
of unity, and so is not freely representable.
For example, dihederal groups must have such norm relations and so we see (in a second way) that they are not freely representable.
This idea will arise later (See Lemma~\ref{new_lemma}).
\end{remark}

\begin{corollary}\label{pq_cor}
If $G$ is a finite freely representable group then every subgroup of $G$ of order $pq$, where $p$ and $q$ are primes, is cyclic.
\end{corollary}

\begin{remark}
We can use this result to give another proof of Lemma~\ref{abelian_lemma} (all finite Abelian freely representable groups are cyclic)
using the structure theorem for finite Abelian groups. 
 
We can even an even more elementary algebraic proof of  Lemma~\ref{abelian_lemma} based on the above result without appealing to the structure theorem for finite Abelian groups. 
\emph{We assume we have a  Abelian subgroup $A$ with the property that all subgroups
of size~$p^2$ are cyclic where~$p$ is any prime dividing $|A|$ and we show that all such groups are cyclic.}
Start by fixing a prime $p$ dividing~$|A|$ (if $|A| = 1$ we are done of course) and consider the endomorphism~$x \mapsto x^p$.
Let~$K$ be the kernel and let~$I$ be the image. By Cauchy's theorem (in the easy case of Abelian groups), the group~$K$ is a nontrivial $p$-group.
Suppose $K$ contains two distinct subgroups~$C_1$ and $C_2$ of order~$p$. Then $C_1 C_2$ is a group of order $p^2$,
and so is cyclic. This means that~$C_1 = C_2$ since a cyclic group of order $p^2$ has a unique subgroup of order~$p$.
This is a contradiction, thus $K$ is a group of order $p$. This means that $I$ has index~$p$ in $G$, so by induction we can
assume $I$ is cyclic. If $I$ has order prime to $p$ then~$G = K I \cong K \times I$ is cyclic. So we assume $p$ divides the
order of $I$. Let $g \in G$ be an element mapping to a generator of $I$. So $p$ divides the order of $g$, which means
that $g$ has order $p$ times the order of $I$. Thus $g$ has order $|I | \cdot |K| = |G|$.
\end{remark}

We now consider the Cartesian products of groups of relatively prime order.

\begin{proposition}
Suppose $A$ and $B$ are finite groups of relatively prime order, and let $F$ be a field.
If $A$ has a free linear representation on a
finite dimensional $F$-vector space $V_A$, and $B$ has a free representation on a finite dimensional $F$-vector
space on $V_B$, then the associated representation of $A \times B$ on $V_A \otimes V_B$ is a free linear representation,
where  $A\times B$ acts on $V_A \otimes V_B$ according to the rule
$$(a, b) (v_1 \otimes v_2) = a(v_1) \otimes b(v_2).$$
\end{proposition}

\begin{proof}
Suppose $g v = v$ where $g$ is a nonidentity element of $A\times B$.
Then $g^k v = v$ for all powers of $g$. 
In particular, if $p$ is a prime divisor of the order of $g$ then there is an element $\sigma$ of order $p$
such that $\sigma v = v$.
Since $A$ and $B$ have relatively prime orders, $\sigma$ must
be in $A$ or $B$ (where $A$ and $B$ are regarded as subgroups of~$A \times B$).

Suppose $\sigma \in A$. Let $e_1, \ldots, e_n$ be a basis for $V_B$. Then 
$$
v = v_1 \otimes e_1 + \ldots + v_n \otimes e_n
$$
where $v_i \in V_A$. Since $\sigma v = v$,
$$
(\sigma v_1) \otimes e_1 + \ldots + (\sigma v_n)  \otimes e_n = v_1 \otimes e_1 + \ldots + v_n \otimes e_n.
$$
Thus $\sigma v_i = v_i$ for each $i$. So $v_i = 0$ since the representation on $V_A$ is a free linear representation.
Thus $v=0$ as desired. Similarly, $v = 0$ if $\sigma \in B$.
\end{proof}

This yields a generalization of a result about cyclic groups:  if $A$ and $B$ have relatively prime orders
then $A \times B$ is cyclic if and only if both $A$ and $B$ are cyclic.

\begin{corollary}\label{rel_prime_cor}
Suppose $A$ and $B$ are finite groups of relatively prime order. 
Then~$A\times B$ is freely representable if and only if $A$ and $B$ are both freely representable.
\end{corollary}

Next we investigate freely representable $p$-groups. 
If $p$ is odd these turn out to cyclic, as we will soon see.
In the case where $p=2$ we can get generalizations of quaternion groups in addition
to cyclic groups, but the details will have to wait for a later section.

We start with a few lemmas. The first is motivated by the
following question: if a prime~$p$ divides 
the order of a freely representable group $G$,  is there a unique subgroup of order $p$?
This is a natural question since we always have uniqueness when $G$ is cyclic, and we have established uniqueness
for general freely representable groups~$G$ in the case  $p=2$. The following shows that $p$ dividing the order of the center
is a sufficient condition.
(We can even weaken the hypothesis a bit: 
instead of requiring that $G$ be freely representable, we just assume $G$ has the property that all 
subgroups of order~$p^2$ are cyclic.)

\begin{lemma} \label{p2top3}
Suppose $p$ is a prime and $G$ is a finite group with the property that every subgroup of $G$ of order $p^2$
is cyclic. If $p$ divides the order of the center $Z$ of~$G$ then 
$G$ has exactly one subgroup of order $p$, and that group is a subgroup of $Z$.
\end{lemma}

\begin{proof}
By Cauchy's theorem $Z$ has a subgroup $Z_p$ of order $p$.
Let $C_p$ be a subgroup of order $p$. Then $H = C_p Z_p$ is a subgroup with at most $p^2$
elements. Since every element of $H$ has order $1$ or $p$, the order of $H$ is $p$ or $p^2$. 
Thus $H$ is cyclic, and so~$H$ has a unique subgroup of order $p$. Thus $C_p = Z_p$.
\end{proof}

\begin{lemma}\label{hom_lemma}
Suppose $G$ is a group of order $p^k$ where $p$ is a prime.
Suppose that~$H_1$ and $H_2$ are  distinct Abelian subgroups of $G$ of index $p$.
If $p\ne 2$ then the map~$x \mapsto x^p$ is a homomorphism from $G$ to $H_1 \cap H_2$.
If $p=2$ then the map $x \mapsto x^4$ is a homomorphism from $G$ to $H_1 \cap H_2$.
\end{lemma}

\begin{proof}
Since $H_1$ and $H_2$ have index $p$ in $G$, they are normal subgroups of $G$ (Proposition~\ref{indexp_prop}).
Thus  $Z = H_1 \cap H_2$  is also a normal subgroup of $G$.
Observe that~$G = H_1 H_2$ and that $Z$ is in the center of $G$ since $H_1$ and $H_2$ are Abelian.\footnote{In the
interesting case where
$G$ is non-Abelian $Z$ is all of the center. Hence the notation $Z$ for this subgroup. (Why? If $Z'$
is  the center then $Z' H_i$ is Abelian, so is $H_i$.)}
Since~$H_1$ and $H_2$ both have index $p$, if $g\in G$ then $g^p \in Z$.

Given $x, y \in G$ we form the commutator $[x, y] = x^{-1} y^{-1} x y$. So~$[x, y]$ is defined by
$$x y  = y x [x, y].$$
In the case where $a\in H_1$ and $b \in H_2$ we have $a^{-1} b^{-1} a b \in H_1 \cap H_2$
since $H_1, H_2$ are both normal. So $[a, b] \in Z$.
This means that $G/Z$ is Abelian since $G=H_1 H_2$ and $H_1, H_2$ are 
Abelian.
In particular, $[x, y]\in Z$ for all~$x, y \in G$.

If $x, y \in G$ then $x^p y = y x^p$ since $x^p \in Z$ and $Z$ is in the center.
However, 
$$
x^n y = x \cdots x y = x \cdots x y x [x, y] = \ldots = y x^n [x, y]^n
$$
since $[x, y]$ is in the center. Comparing these expressions when $n = p$ gives
$$
[x, y]^p = 1.
$$
Another consequence of the fact that $[x, y]$ is in the center of $G$ is that
$$
(xy)^p = x yxy \cdots xy xy = x^p y^p [y, x]^m
$$
where $m = 1 + \ldots + (p-1) = p(p-1)/2$. 
When $p$ is odd, we have that $m$ is a multiple of $p$, so $[y, x]^m = 1$ for all $x, y \in G$
and so
$$
(xy)^p = x^p y^p.
$$
If $p = 2$ then $(xy)^2 = x^2 y^2 [y, x]$ which doesn't necessarily give a homomorphism.
But
since $[y, x]^2 = 1$ and since $[y, x]$ is in the center,
$$
(xy)^4 = xyxyxyxy = x^4 y^4 [y, x]^6 = x^4 y^4.
$$
\end{proof}

Here is an interesting group theoretical consequence (interesting beyond just the theory of freely representable
groups):

\begin{proposition}\label{p3top4}
If $p$ is an odd prime then every $p$-group $G$ with only one subgroup of order $p$
is cyclic.
\end{proposition}

\begin{proof}
We proceed by induction on $k \ge 1$ where $p^k$ is the order of $G$. 
The base case~$k=1$ is clear so assume $k \ge 2$.
So suppose $G$ has a unique
subgroup $C$ of order $p$.
Let $H_1$ be a subgroup of index $p$ in $G$ (which exists by Proposition~\ref{cseries_prop}).
Let~$g \in G$ be any element of $G$ not in $H_1$. If $g$ generates~$G$ we are done, so 
assume $g$ generates a proper subgroup. By Proposition~\ref{cseries_prop}
there is a subgroup $H_2$ of index~$p$ in $G$ containing $g$.

Observe that $H_1$ and $H_2$ can have only one subgroup of order $p$ since that holds of $G$.
So $H_1$ and $H_2$ are cyclic by the induction hypothesis.
By the above lemma,~$x \mapsto x^p$ is a homomorphism from~$G$ to $H_1 \cap H_2$.
Observe that the kernel is $C$, the unique subgroup of $G$ of order $p$.
Thus the image of $x \mapsto x^p$ has $p^{k-1}$ elements. However $H_1 \cap H_2$ has order bounded by $p^{k-2}$,
a contradiction.
\end{proof}

These results can be collected together to give another family of groups where freely representable coincides with cyclic:

\begin{theorem}\label{pgroup_thm}
Suppose $G$ is a $p$-group where $p$ is an odd prime. Then the following are equivalent.
\begin{enumerate}
\item\label{p1}
$G$ is freely representable.
\item\label{p2}
Every subgroup of $G$ of order $p^2$ is cyclic.
\item\label{p3}
$G$ has a unique  subgroup of order $p$.
\item\label{p4}
$G$ is cyclic.
\end{enumerate}
\end{theorem}

\begin{proof}
We have $(\ref{p1}) \implies (\ref{p2})$ by Theorem~\ref{pq_thm}.
We have $(\ref{p2}) \implies (\ref{p3})$ by Lemma~\ref{p2top3} (the center
is nontrivial by Proposition~\ref{center_prop}).
We have $(\ref{p3}) \implies (\ref{p4})$ by Proposition~\ref{p3top4}.
Finally  $(\ref{p4}) \implies (\ref{p1})$  as in Example~\ref{cyclic_example}.
\end{proof}

\begin{remark}
The quaternion group on 8 elements shows that this theorem does not
hold for $p=2$. We will investigate $2$-groups in a later section.
\end{remark}

\begin{remark}
It follows from the above theorem that if $G$ is freely representable then for every odd prime $p$ the $p$-sylow
subgroup is cyclic. An important case is where every Sylow subgroup is cyclic (even for $p=2$).
Such groups have a long history.
Frobenius and Burnside
proved such groups are solvable (and even metacyclic). 
In fact, such groups constitute one of the three classes of groups studied by Burnside in his classification
of groups that could be freely representable (\cite{burnside1905a} and \cite{burnside1905b}).
We will review the theory of such groups in Section~\ref{sylow_cyclic_chapter}.
\end{remark}

\chapter{Manifolds of Constant Positive Curvature}

Now that we have explored the more accessible properties of freely representable groups and 
have some feel for them, we will take a break from group theory and discuss their central
role in Riemannian geometry in order to further motivate their study.
After this we will return to group theoretic matters and survey the classification of such groups.

Freely representable groups are of importance in Riemannian geometry because of the following
(where $\mathbf{S}^n$ is the unit sphere in $\bR^{n+1}$):

\begin{theorem}
Suppose $\Gamma$ is a finite group of isometries of $\mathbf{S}^n$ acting freely on~$\mathbf{S}^n$.
Then $\Gamma$ is freely representable. Conversely if $G$ is a freely representable group, then for some $n$
there is a  finite group of isometries $\Gamma$ of $\mathbf{S}^n$ isomorphic to $G$ acting freely on~$\mathbf{S}^n$.
\end{theorem}

\begin{proof}
We appeal to the following standard result in differential geometry: the isometry group
of $\mathbf{S}^n$ can be identified with the orthogonal group $\mathbf{O}(n+1)$.
In particular, an orthogonal linear tranformation $\bR^{n+1} \to \bR^{n+1}$ restricts to an isometry
of any sphere $\mathbf{S}^n \subseteq \bR^{n+1}$ centered at the origin.

Assume $\Gamma$ is a finite subgroup of $\mathbf{O}(n+1)$ acting freely on a 
$\mathbf{S}^n \subseteq \bR^{n+1}$. This means that the inclusion map gives a free representation of $\Gamma$,
so $\Gamma$ is freely representable.

Conversely, if $G$ is a freely representable, then $G$ is freely representable over $\bR$ by Lemma~\ref{all_char_lemma}.
So there is a free representation of $G$ on a real vector space $V$.
We can form a positive definite inner product on $V$ that is $G$-invariant (using the standard averaging technique).
In particular, every element of $G$ acts as an orthogonal transformation with respect to such an inner product.
Fix an orthonormal basis and identify $V$ with $\bR^{n+1}$ for some $n \ge 0$.
Since the representation is free, it is faithful. So $G$ is isomorphic its image $\Gamma$ in $\mathbf{O}(n+1)$.
Since the representation is free, $\Gamma$ acts freely on the unit sphere $\mathbf{S}^n$.
\end{proof}

\begin{remark}
If $n$ is $0$ or $1$ then any finite group of isometries acting freely of $S^n$ is cyclic.
However, in Example~\ref{cyclic_example}
we showed how to construct free linear representation over~$\bR$ of cyclic groups for any even dimension.
So we can safely remove the cases $n=0$ and $n=1$ in the above theorem, and thus assume $\mathbf{S}^n$
is simply connected: \emph{if $G$ is a freely representable group then for some $n\ge 2$
there is a  finite group of isometries~$\Gamma$ of~$\mathbf{S}^n$ isomorphic to $G$ acting freely on
the simply connect sphere~$\mathbf{S}^n$.}
\end{remark}

\begin{remark}
If $n+1$ is odd then any element of $\mathbf O(n+1)$ must have a real eigenvalue, and that eigenvalue must equal to $\pm 1$.
Thus the square of every element must have eigenvalue $+1$. This means that for any finite subgroup $\Gamma$ of $\mathbf O(n+1)$
acting freely on~$\mathbf{S}^n$, the square of every element is the identity. So \emph{all} complex eigenvalues of elements
of~$\Gamma$ are $\pm 1$, but we cannot have nonidentity elements of $\Gamma$ with eigenvalue~$+1$
since $\Gamma$ acts freely. So $\Gamma$ is a subgroup of $\{ \pm I \}$. 
In other words, if $n$ is even then the only  finite subgroup of $\mathbf{O}(n+1)$ acting freely on
$\mathbf{S}^n$
are $\{ \pm I \}$ and the trivial group $\{  I \}$.

This means we can restrict our attention to $\mathbf{S}^n$ for odd $n$:
\emph{if $G$ is a freely representable group, then for some odd $n\ge 3$
there is a  finite group of isometries~$\Gamma$ of~$\mathbf{S}^n$ isomorphic to $G$ acting freely on
the simply connect sphere~$\mathbf{S}^n$.}

Observe that Example \ref{quaternion_example} gives interesting non-Abelian freely representable groups
action on $\mathbf{S}^3$.
\end{remark}

The groups $\Gamma$ of the above theorem are of importance in the classification of spaces of positive curvature because
of the following theorem:

\begin{theorem}
Suppose $M$ is a complete, connected Riemannian manifold of constant positive sectional curvature $K$
and of dimension $n \ge 2$.
Then $M$ is isometric to~$\mathbf{S}^n/\Gamma$ 
where $\Gamma$ is a finite subgroup of isometries of~$\mathbf{S}^n$ acting freely on 
the sphere~$\mathbf{S}^n$ of curvature $K$.

If both $\Gamma$ and $\Gamma'$ are finite subgroup of isometries of~$\mathbf{S}^n$ acting freely on $\mathbf{S}^n$,
then $\mathbf{S}^n/\Gamma$ is isometric to $\mathbf{S}^n/\Gamma'$ if and only if 
$\Gamma$ and $\Gamma'$ are conjugate subgroups of~$\mathbf O(n+1)$.
\end{theorem}

\begin{corollary}
A finite group $G$ is a freely representable group if and only if it occurs as the fundamental group
of a complete, connected Riemannian manifold of constant positive sectional curvature
of dimension $n \ge 2$.
\end{corollary}

\begin{remark}
As pointed out in an earlier remark, if $n\ge 2$ is even then $\Gamma$ in the above theorem
is limited to either $\{ \pm I \}$ or the trivial group $\{  I \}$. This means that there are only
two  complete, connected Riemannian manifold of constant positive sectional curvature $K$ of dimension $n$:
$\mathbf S^n$ itself and real projective space.
\end{remark}

\begin{remark}
Since cyclic groups are freely representable, they occurs as the fundamental group of such manifolds.
Such manifolds are \emph{lens spaces} of constant curvature.
\end{remark}


\section{Constant curvature: some history}

The Riemannian manifolds of constant curvature play a central role in geometry as a whole. 
The study of such manifolds is geometry \emph{par excellence}. It arose out of the hyperbolic geometry of Lobatchevsky-Bolyai-Gauss.
In fact, it predates hyperbolic geometry since Euclidean and spherical geometry are special cases.
In~1854 Riemann gave his account of elliptic geometry, and began the development of Riemann geometry (as a generalization
of Gauss's differential geometry of surfaces).
Beltrami played an essential role by providing models for non-Euclidean geometry in 1868.

In the 1870s, William Clifford began exploring interesting topologies that
can occur with surfaces of constant curvature. He was an early proponent
of Riemann's formulation of differential geometry, which was revolutionary at the time.  
Clifford was a remarkable figure, who tragically died at age 33
in 1879. Clifford
algebras are named for him, and 
some of Clifford's ideas about curvature
anticipate Einstein's theory of gravity.

Starting in the early 1870s
Felix Klein became interested non-Euclidean geometry and in 
classifying surfaces of constant curvature. He pioneered the techniques of using covering
spaces for studying such surfaces in 1891.
Klein coined the term ``hyperbolic'' and ``elliptic'' geometry (and even ``parabolic'' geometry for Euclidean geometry, but that
term has not caught on).

In 1891, Killing posed the classification question for
spaces of positive curvature in $n$-dimensions and called this the Clifford-Klein spherical space form
problem (The \emph{Clifford-Kleinschen Raumproblem}).

The term ``space form'' is a term for complete, connected manifold of constant curvature or
equivalently (the Killing-Hopf theorem) manifolds of the form $M/\Gamma$ where $M$ is the (simply connected) standard  hyperbolic, Euclidian,
or spheric spaces and where $\Gamma$ is a group of isogenies acting freely and properly discontinuously on~$M$.
This last condition means that for each $x \in M$ there is a neighborhood $U$ such that~$\gamma U$ and $U$ are
disjoint if $\gamma \in \Gamma$ is not the identity.
In the case where $M$ is a sphere, this just means that $\Gamma$ is finite since the sphere is compact.
The term ``space form'' can also refer to this analogous situation in differential topology or the topology of manifolds
more generally, where $\Gamma$ is not restricted to isometries.

H.~Hopf (1926) modernized Killing's spherical space form problem,
and highlighted the use of group theoretical methods. 
He studied the examples where $\Gamma$ is a binary dihedral group, thus establishing an infinite
family of topologically distinct three-dimensional spaces of positive curvature.

Threfall and Seifert classified three dimensional spherical spaces in 1930 which uses the interesting subgroups
of $\bH^\times$. 
Georges Vincent made significant progress toward the general classification in 1947. Vincent classifed all the solvable
freely representable groups and their free representations over~$\bR$.
This lead to the classification of spherical space forms in all dimensions not congruent to $3$ mod~$4$.\footnote{As far as I know, Vincent only
published  one paper.}
Wolf in the 1960s completed the classification of freely representable groups, classified their real free representations,
and so classified all complete, connected Riemannian manifolds of constant curvature \cite{wolf2011}.
He simplied Vincent's approach by using work by Zassenhaus, and was thus able to expand the results
to non-solvable freely representable groups.

As mentioned above Burnside had studied
examples of freely representable groups as early as the 1905, but outside the context of geometry.
More specifically he studied groups whose Sylow subgroups are cyclic except for $2$-Sylow
subgroups which are of quaternion type, and showed all freely representable groups have these properties.
Vincent (and earlier Hopf?) brought this work on freely representable groups into the classification problem of Clifford-Klein space forms.
(Holder proved the solvability in 1895).


\chapter{Freely Representable $2$-Groups}

For odd primes $p$ the only freely representable $p$-groups are cyclic.
However, in Example~\ref{quaternion_example} we saw that any finite subgroup of~$\bH^\times$
is freely representable and many of these are non-Abelian $2$-groups. We will see that all 
freely representable $2$-groups are in fact isomorphic to a subgroup of $\bH^\times$.

\begin{definition}
A \emph{generalized quaternion group} is a group isomorphic to a noncyclic finite group of~$\bH^\times$ whose order
is a power of $2$. For example, the standard quaternion group~$\{\pm 1, \pm \mathbf{i}, \pm \mathbf{j}, \pm \mathbf{k}\}$
qualifies as a generalized quaternion group.
\end{definition}

It is immediate from this definition that generalized quaternion groups are freely-representable groups.
As discussed in Example~\ref{quaternion_example}, every noncyclic finite subgroups of $\bH^\times$
is the preimage of a finite noncyclic subgroup of $\mathrm{SO}(3)$ under the standard two-to-one homomorphism
$\bH_1 \to \mathrm{SO}(3)$. So a generalized quaternion group must be the preimage of a noncyclic
subgroup
of $\mathrm{SO}(3)$ whose order is a power of two. The only such subgroups of $\mathrm{SO}(3)$ are dihedral groups $D_n$
where~$n \ge 2$ is a power of~$2$.
Thus we have the following:

\begin{proposition}
A group is a generalized quaternion group if and only if it is isomorphic to $2 D_n$ where $n\ge 2$ is some power of $2$.
Thus every generalized quaternion group has order $2^k$ for some $k\ge 3$, and for every $k\ge 3$ there is a unique
generalized quaternion group of order $2^k$ up to isomorphism.
\end{proposition}

\begin{proof}
Most of this is clear from the above discussion. The uniqueness is a consequence of the fact that dihedral 
subgroups of $\SO(3)$ of a fixed order are isomorphism (see appendix on finite subgroups of $\SO(3)$), together with
Proposition~\ref{binary_conjugate_prop}.
\end{proof}

Let $Q$ be a generalized quaternion group of order $2^n$ with $n \ge 3$. We identify~$Q$ with a subgroup of $\bH^\times$.
As mentioned in Example~\ref{quaternion_example} (and Lemma~\ref{doublecover_lemma}), $Q$
has a unique element of order 2, namely $-1$. Also $Q$ is a subgroup of $\bH_1$, 
and restricting the map~$\bH_1 \to \mathrm{SO}(3)$ to~$Q$ gives a two-to-one homomorphism $Q \to D$ where $D$ is a dihedral subgroup of $\mathrm{SO}(3)$
of order~$2^{n-1}$. The kernel of this homomorphism is~$\{ \pm 1\}$.

As in our discussion in Example~\ref{quaternion_example}, the image of any nontrivial cyclic subgroup~$C$ of $Q$ of order $k$
is a cyclic subgroup of $D$ of order $k/2$. In other words, an element of~$Q$ of order $k>1$ maps to an element of order $k/2$.
Since $D$ is dihedral of order~$2^{n-1}$, it has generators $\rho$ and $\tau$ such that~$\rho$ has order $2^{n-2}$ and $\tau$ has order~$2$,
and such that~$\tau \rho \tau = \rho^{-1}$. Let $R \in Q$ map to $\rho$ and let $T \in Q$ map to $\tau$.
So~$R$ has order~$2^{n-1}$ and~$T$ has order~$4$. Note that~$R$ generates a cyclic subgroup of~$Q$ of index 2.
Since~$R^{2^{n-2}}$ and $T^2$ are elements of order 2, they are both equal to~$-1$, and so~$T^2 = R^{2^{n-2}}$.
The subgroup $\left<R, T\right>$ of $Q$ generated by $R$ and $T$ maps onto~$D$ 
since~$\rho$ and $\tau$ generate $D$, so $\left<R, T\right>$ has order $2^n$ (Lemma~\ref{doublecover_lemma}),
meaning that~$\left<R, T\right> = Q$.


In $D$ we have that $\tau \rho \tau  = \rho^{-1}$ so that $(\tau \rho)^2 = 1$. This means that $\tau \rho$ has order~2 (it cannot
have order 1 since $\tau$ and $\rho$ are not inverses; otherwise $D$ would be generated by $\tau$ and would have order 2).
Thus $TR$ has order 4, and $(TR)^2$ must be -1. Note also that $T^{-1} = - T$ since $T^2 = -1$.
Thus
$$
TRT^{-1}R = TR(-T) R = -(TR)^2 = 1, \qquad\text{so}\quad TRT^{-1} = R^{-1}. 
$$
In summary, we have the following relations between the generators $R$ and $T$:
$$
R^{2^{n-1}} = 1, \quad T^2 = R^{2^{n-2}},   \quad TRT^{-1} = R^{-1}.
$$
These relations are sufficient to characterize $Q$:

\begin{proposition}\label{relations_prop}
The quaternion group group $Q$ with $2^n$ elements is isomorphic to the following abstract group $G$ described
by two generators $R$ and $T$ and three relations
$$
R^{2^{n-1}} = 1, \quad T^2 = R^{2^{n-2}},   \quad TRT^{-1} = R^{-1}.
$$
\end{proposition}

\begin{proof}
Using the given relations we see that we $TR = R^{-1} T$, and so
every element of $G$ can be written as~$R^i T^j$ where~$0 \le i \le 2^{n-1}-1$ and where~$0 \le j \le 1$. So $G$ has at most $2^n$
elements.

Next observe that the specific choice of elements we also call $R$ and $T$ in $Q$ given above satisfies the desired relations,
giving us a homomorphism $G \to Q$ (where we view $G$ as a quotient of a free group generated by two elements).
Since $R$ and $T$ in~$Q$ generates~$Q$, the homomorphism $G \to Q$ is surjective. 
Since~$Q$ has order $2^n$ and $G$ has order at most $2^n$, this map is an isomorphism
and~$G$ has order exactly~$2^n$.
\end{proof}

Our next goal is to show that every freely representable $2$-group is a generalized quaternion group. 
It turns out that we can do so by exploiting the fact that freely representable groups of even order have
a unique element of order~$2$. To this end we will require a few lemmas about $2$-groups with a unique
element of order~$2$.

\begin{lemma} \label{2_abelian_lemma}
If $A$ is an Abelian group of order $2^k$ with a unique element of order~$2$ then
$A$ is cyclic.
\end{lemma}

\begin{proof}
This follows from the structure theorem for finite Abelian groups. 

A direct proof can be given. We proceed by induction. The base case~$k=1$ is clear, 
so suppose $A$ is Abelian with order $2^k$ where $k \ge 2$, and assume~$A$ has a unique element of order 2. Consider the homomorphism $A\to A$ defined by the rule~$x \mapsto x^2$.
The kernel of this map has two elements, so the image~$C$ has order~$2^{k-1}$.
Since $C$ is a subgroup of $A$ of order $2^{k-1}$, it has a unique element of order 2.
So, by induction, $C$ is cyclic.
Let $h \in A$ map to a generator of~$C$, so $h^2$ has order $2^{k-1}$.
Since the order of $h$ is even, the order of $h^2$ is half the order of $h$.
So~$h$ has order~$2^k$, and $A$ must be cyclic with generator $h$.
\end{proof}

\begin{lemma}\label{2_small_lemma}
Suppose that $G$ is a $2$-group with a unique element of order $2$.
\text{If $|G| \le 8$} then $G$ is cyclic or is the quaternion group with $8$ elements.
\end{lemma}

\begin{proof}
Suppose that $G$ is not cyclic. Then $|G| = 4$ or $|G|=8$, and all elements of~$G$ have order 1, 2, or 4.
Note that $G$ has one element of order 1 and one element of order~2, so $G$ has $|G|-2$ elements have order 4.
If $|G| = 4$, then any of the elements of order 4 generate $G$ so $G$ is in fact a cyclic group.
So. we can assume $G$ is a noncyclic group of order 8 with 6 elements of order 4.

Let $R$ be any element of order 4. Note that the subgroup generated by $R$ has~2 elements
of order 4, so there are 4 additional elements of order 4. Let~$T$ be one of these additional elements of order 4.
Note that the group generated by~$R$ and $T$ has more than 4 elements, so must be all of $G$.

Finally we work out enough relations between $R$ and $T$ to determine $G$.
Since~$R$ has order 4, we have the relation $R^4 = 1$.
Since $R^2$ and $T^2$ have order~2, they must be equal. So we get a second relation $R^2 = T^2$. 
The group generated by $R$ has index two in $G$ so must be a normal subgroup of $G$.
Thus the conjugate $TRT^{-1}$ of~$R$ is a power of $R$. It must have order $4$ and so this conjugate is $R$ or $R^3 = R^{-1}$.
If $TRT^{-1} = R$ then $R$ and $T$ commute and so $G$ is Abelian. This can't happen 
by the previous lemma.
So we must have $TRT^{-1} = R^{-1}$ as a third relation.
By Proposition~\ref{relations_prop}, these three relations insure that the group $G$ is a 
quotient of a quaternion group of order $8$. Since $G$ has size $8$ this means that $G$ is isomorphic to the quaternion
group with 8 elements.
\end{proof}

Next we see that with one exception $2$-groups with a unique element of order $2$ must have a unique cyclic subgroup of index $2$.

\begin{lemma} \label{2_unique_lemma}
Let $G$ be a group of order $2^n$ with a unique element of order $2$
and with at least two distinct cyclic subgroups of index $2$.
Then 
$G$ is isomorphic to the quaternion group with $8$ elements.
\end{lemma}

\begin{proof}
Observe that $G$ is non-Abelian: otherwise $G$ would be cyclic by Lemma~\ref{2_abelian_lemma} and would only have one subgroup of index $2$.
Let $H_1$ and $H_2$ be distinct cyclic subgroups of index $2$. Because $H_1$ and $H_2$ have index 2 in $G$, they are normal subgroups of $G$.
Thus the intersection~$H_1 \cap H_2$ is also normal. Observe that $G$ must have order at least $8$ since otherwise it could not have two distinct
subgroups of index 2.

Let $Z$ be the center of $G$. 
Observe that $Z H_i$ is Abelian, and so cannot be all of~$G$. This means $Z H_i = H_i$ and so $Z \subseteq H_i$.
Hence $Z \subseteq H_1 \cap H_2$.  Observe that~$G = H_1 H_2$, and that~$H_1 \cap H_2 \subseteq Z$ is in the center of $G=H_1 H_2$
since $H_1$ and~$H_2$ are cyclic. We conclude that  $Z = H_1 \cap H_2$.

From this we see that the inclusion $H_1 \hookrightarrow G = H_1 H_2$ induces an isomorphism~$H_1/Z \to G/H_2$ (since $Z = H_1 \cap H_2$).
So $[H_1: Z] = 2$. Since $H_1$ has order~$2^{n-1}$ we have that $Z$ has order $2^{n-2}$.
Since $Z$ is Abelian, it is cyclic by Lemma~\ref{2_abelian_lemma}.
Let~$C$ be the unique subgroup of $Z$ of index 2 in $Z$. So $C$ is a cyclic group of order $2^{n-3}$.

Since $G/H_i$ has order 2, if $x \in G$ then $x^2 \in H_i$. So $x^2 \in Z = H_1 \cap H_2$. 
Furthermore, since $Z/C$ has order 2, we have $x^4 \in C$ for all $x \in G$.
In other words,~$x \mapsto x^4$ is a map $G \to C$.
By Lemma~\ref{hom_lemma} the map $x \mapsto x^4$ is actually a homomorphism $G \to C$.
Let $K$ be the kernel. Observe that $K$ consists of all elements of $G$ of order 1, 2, or 4.
Since $C$ has $2^{n-3}$ elements, the image of~$G \to C$ has at most $2^{n-3}$ elements, so $K$ has at least $8$ elements.

As noted above $G$ has at least $8$ elements. Suppose $G$ has more than $8$ elements, so $n \ge 4$.
Since $Z$ has order $2^{n-2}$ this means $Z$
contains a subgroup $Z_4$ of order~$4$. Then $Z_4 \subseteq K$ by definition of $K$. Since $K$ has order at least~$8$,
there is an element~$b \in K$ not in $Z_4$. Since $Z_4$ is in the center, $Z_4 \left< b \right>$ would be an Abelian subgroup of $K$ of order
at least $8$. By Lemma~\ref{2_abelian_lemma}, this means~$Z_4 \left< b \right>$ would be cyclic of order $8$ or more. But every element of $K$ has
order at most $4$, a contradiction. We conclude that~$G$ has~8 elements.
Since it is not cyclic, it must be the quaternion group by the previous lemma.
\end{proof}

We know that a generalized quaternion group $Q$ has a cyclic subgroup of index~2 
and a unique element of order $2$.  Now we prove the converse for noncyclic groups:

\begin{lemma}\label{index2_cyclic_lemma}
Suppose $G$ is a noncyclic group of order $2^n \ge 8$ with 
a cyclic subgroup~$C$ of index $2$ and a
unique element of order~$2$.
Then $G$ is a generalized quaternion group.
\end{lemma}

\begin{proof}
Because $G$ has order $2^n$, 
it is enough to find generators in $G$ satisfying the relations of Proposition~\ref{relations_prop}.
We simply choose $R$ to be any generator of $C$ and $T$ to be any element outside of $C$.
These elements $R$ and $T$ generate~$G$ since $C$ has index $2$ in $G$.
We just need to show that $R$ and $T$ satisfy the desired relations.
The first relation, $R^{2^{n-1}} = 1$, is immediate.

Observe that~$C$ is normal in $G$ since its index in $G$ is 2, and  so $T^2 \in C$
since~$G/C$ has order 2. Let~$C'$ be the cyclic subgroup of $C$ generated by $T^2$.
If $C'=C$ then~$T$ generates $G$, so $G$ is cyclic, a contradiction. So $C'$ is a proper subgroup of $C$.
Let~$H_1$ be the unique subgroup of $C$ containing $C'$ with $[H_1: C'] = 2$. Let $H_2$ be the subgroup of~$G$ generated
by~$T$. So~$C'$ also has index $2$ in $H_2$ and $|H_1| = |H_2|$.

Every subgroup of a normal cyclic subgroup of $G$ must be normal in $G$ (since there is at most one subgroup of a cyclic group of any given order).
Thus $H_1$ is a normal subgroup of $G$. This implies that $Q = H_1 H_2$ is a subgroup of $G$. Every element of~$Q$
can be written as $h$ or $h T$ with $h\in H_1$ since $T^2 \in H_1$.
Thus $Q$ has order~$2|H_1|$.
Since $H_1$ and $H_2$ have equal order, they both have index 2 in $Q$. By the previous lemma $Q$ is the quaternion group of 8 elements.

Since $Q$ is the quaternion group with 8 elements, 
$T$ has order $4$. Now $T^2$ and~$R^{2^{n-2}}$ have order 2, and so  are equal:
$$
T^2 = R^{2^{n-2}}.
$$
Also $H_1$ has order $4$ so $C$ has order at least $4$.

Note that $T$ acts via conjugation on $C$ since $C$ is normal in $G$. The square~$T^2$ is in $C$ so acts trivially on $C$.
However, $T$ cannot act trivially on $C$ since $G$ is not Abelian (otherwise it would be cyclic by Lemma~\ref{2_abelian_lemma}). 
Thus $T$ acts as an automorphism of $C$ of order~$2$.
One such automorphism is $x \mapsto x^{-1}$. If $|C| \ge 8$ there are two other
automorphism of order $2$, namely $x \mapsto x^{2^{n-2} + 1}$
and $x \mapsto x^{2^{n-2} - 1}$. 

Assume for now that $C$ has order $8$ or more. 
Then $H_1$ (which has four elements) is a proper subgroup of $C$.
Observe that $x \mapsto x^{2^{n-2} + 1}$ maps $R^2$ to $R^2$, so restricts to the identity on 
any proper subgroup of $C$, including $H_1$. But $T$ acts nontrivially on $H_1$ (since $H_1H_2 = Q$ is non-Abelian),
so this gives a contradiction. 
Suppose instead that~$T$ acts as $x \mapsto x^{2^{n-2} -1}$. Then
$$
(TR)^2  = T^2(T^{-1} R T) R =  T^2 R^{2^{n-2}} = 1
$$
since $T^2$ and $R^{2^{n-2}}$ are each the unique element of order $2$ in $G$.
Thus $TR$ has order $2$. But $TR$ is not in $C$, so cannot be the unique element of order $2$.

Thus in any case  $T$ acts as $x \mapsto x^{-1}$ on $C$.
So we get the third and last relation
$$
TRT^{-1} = R^{-1}.
$$
\end{proof}

As before, let $G$ be a $2$-group with a unique element of order 2.
The above tells us that if $G$ has a cyclic subgroup of index $2$ it is either cyclic or a generalized
quaternion group. The following tells us that if $G$ has a generalized quaternion subgroup
of index 2 then $G$ is itself a generalized quaternion group.
Since $G$ must have a subgroup of index 2, we can combine these two cases to construct a simple induction
to argue that $G$ is either cyclic or is a generalized quaternion group. 

\begin{lemma}\label{index2_quaternion_lemma}
Let $G$ be a $2$-group with a unique element of order $2$.
If $G$ has a 
generalized quaternion subgroup $Q$ of index $2$ then $G$ is itself a generalized quaternion group.
\end{lemma}

\begin{proof}
The subgroup $Q$ has a unique cyclic subgroup $C$ of index $2$ in  $Q$ in the case where $|Q|>8$.
If~$|Q|=8$ then $Q$ has three such cyclic subgroups $C_1, C_2, C_3$. 
Note that $Q$ is normal in $G$ since it has index 2. By uniqueness of $C$ for $|Q|>8$ we see
that $C$ is also a normal subgroup of $G$.
Now consider the case $|Q| = 8$. Any~$\tau \in G$ permutes
the set~$\{C_1, C_2, C_3\}$ via conjugation, but $Q$ acts trivially on this set since each $C_i$
has index 2 and so is normal in $Q$. In other words, the 
$2$-element group~$G/Q$
acts on the three element set $\{C_1, C_2, C_3\}$. 
At least one $C_i$ must be fixed by this action, and so must be normal in $G$.
So in any case $Q$ has a cyclic subgroup $C$ of index 2 that is a normal subgroup of $G$.

Since $C$ is a normal subgroup of $G$, 
if $\tau \in G$ then $\tau$ acts on $C$ by conjugation. It must act as~$x \mapsto x^k$ for some odd $k$,
since all automorphisms of~$C$ are of this form.
So $\tau^2$ acts as $x \mapsto x^{k^2}$.
But $\tau^2 \in Q$ since $Q$ has index 2 in $G$. So~$\tau^2$ acts as 
either~$x \mapsto x^{- 1}$ or $x \mapsto x$ since $Q$ is a generalized quaternion group.
Note that~$k^2 \not\equiv -1$ modulo $2^j$ if $2^j \ge 4$ (this can be verified by checking modulo~$4$).
So~$\tau^2$ acts as the trivial map~$x \mapsto x$. 
This implies that $\tau^2 \in C$.

If $\tau\not\in C$ then $H_\tau = C \left< \tau \right>$ is a subgroup of $G$ (since $C$ is normal in $G$).
As above we have $\tau^2 \in C$, so every element of $H_\tau$ is of the form $h$ or $h \tau$ with $h \in C$.
\text{So $[H_\tau: C] = 2$} and hence $[G: H_\tau] = 2$. If $H_\tau$ is cyclic then
$G$ is a generalized quaternion group by the previous lemma, and we are done ($G$ is not cyclic since it contains
$Q$ as a subgroup).
If $H_\tau$ is not cyclic, then it must be a generalized quaternion group by the previous lemma.
In particular, $\tau$ acts on $C$ by the
automorphism~$x \mapsto x^{-1}$.

Let $\tau \in G-Q$ and $\sigma \in Q - C$. If 
either $H_{\tau}$ or $H_{\tau \sigma}$ is cyclic, then as pointed out above, $G$ is a general quaternion
group. If neither are cyclic, then $\tau$ and $\tau \sigma$ both act on $C$ by the
automorphism $x \mapsto x^{-1}$. But this means $\sigma \in Q-C$ acts trivially on $C$, a contradiction.
\end{proof}

We are now ready to assert the main results about $2$-groups.

\begin{theorem}\label{unique_2_theorem}
Let $G$ be a $2$-group with a unique element of order $2$. Then $G$ is cyclic or is isomorphic
to a general quaternion group.
\end{theorem}

\begin{proof}
We proceed by induction on $k \ge 1$ for $|G| = 2^k$. The base case $k=1$ is clear so assume $k \ge 2$.
Let $H$ be a subgroup of index $2$ (Proposition~\ref{cseries_prop}).
By the induction hypothesis, $H$ is cyclic or isomorphic to a generalized quaternion group.
If $H$ is cyclic, the result follows from Lemma~\ref{index2_cyclic_lemma}.
If $H$ is a generalized quaternion group, the result follows from Lemma~\ref{index2_quaternion_lemma}.
\end{proof}

\begin{corollary} \label{unique_2_cor}
Let $G$ be a nontrivial $2$-group. Then the following are equivalent:
\begin{enumerate}
\item
$G$ is freely representable.
\item
Every subgroup of $G$ of order $4$ is cyclic.
\item
$G$ has a unique element of order $2$.
\item
$G$ is a cyclic or generalized quaternion group.
\end{enumerate}

\begin{proof}
The implication $(1) \Rightarrow (2)$ is covered by Theorem~\ref{pq_thm}.
Note that $2$ divides the order of the center of $G$ (Proposition~\ref{center_prop}),
so the implication~$(2) \Rightarrow (3)$ is covered by Lemma~\ref{p2top3}.
The implication~$(3) \Rightarrow (4)$ is covered by the above theorem.
The implication~$(4) \Rightarrow (1)$ is covered by the discussion of Example~\ref{quaternion_example}.
\end{proof}
\end{corollary}

It will be convenient to give the groups of the above corollary a special name:

\begin{definition}
A \emph{$2$-cycloidal} group is a $2$-group that is either cyclic or a generalized quaternion group.
Equivalently a 2-cyloidal group is a freely representable 2-group.
\end{definition}

Here is another way in which $2$-cycloidal groups behave like cyclic groups:

\begin{proposition}\label{2_cyloidal_prop}
Let $G$ be a $2$-group. Then $G$ is a $2$-cycloidal group if and only if every Abelian subgroup $A$ of $G$ is cyclic.
\end{proposition}

\begin{proof}
We assume $|G| \ge 2$ since the case $|G| = 1$ is clear.
Suppose $G$ is a  2-cycloidal group.
Then~$G$ and hence every nontrivial subgroup~$A$ of $G$ has a unique element of order 2.
By Lemma~\ref{2_abelian_lemma}, any Abelian subgroup $A$ of $G$ is cyclic.

Conversely, suppose every Abelian subgroup $A$ of $G$ is cyclic. 
Since every subgroup of order 4 is Abelian,  every subgroup of order 4 is cyclic. Thus $G$ is a 2-cycloidal group by 
Corollary~\ref{unique_2_cor}.
\end{proof}

Here are some additional properties of $2$-cycloidal groups

\begin{proposition}\label{2_cycloid_subgroup_prop}
Every subgroup $H$ of a $2$-cycloidal group $G$ is a $2$-cycloidal group.
\end{proposition}

\begin{proof}
Every nontrivial subgroup of $G$ must contain a unique element of order~$2$.
\end{proof}

\begin{proposition}
Let $G$ be a $2$-cycloidal group. Then $G$ contains a cyclic subgroup of index $2$. If $G$
is not the quaternion group with 8 elements then the cyclic subgroup of index $2$ is unique.
\end{proposition}

\begin{proof}
We can identify $G$ with a subgroup of $\bH_1$. The image of $G$ under the double cover $\bH_1 \to \SO(3)$
has image $G'$ which is either dihedral or cyclic. So $G'$ has a cyclic subgroup $H'$ of index $2$, and its
preimage $H$ has index 2 in $G$ (see Lemma~\ref{doublecover_lemma}).

Lemma~\ref{2_unique_lemma} covers the uniqueness claim for generalized quaternion groups.
\end{proof}

\begin{proposition}\label{2_quaternion_subgroup_prop}
Let $G$ be a generalized quaternion group. For all~$8 \le 2^k \le |G|$
there is a generalized quaternion subgroup $H$ of $G$ of order $2^k$.
\end{proposition}

\begin{proof}
We can think of $G$ as a subgroup of~$\bH_1$.
As in Example~\ref{quaternion_example}, there is a natural homomorphism $\bH_1 \to \mathrm{SO(3)}$,
and the preimage of any subgroup of order $n$ of~$\mathrm{SO(3)}$ is a subgroup of order $2n$ of $\bH_1$.
The image of $G$ is a dihedral group $D$ of order~$|G|/2$, and the preimage of $D$ is $G$. A standard property of a dihedral group of order $2n$
is that it contains a dihedral subgroup of order $2m$ for each $m >1$ dividing $n$. So let~$E$ be a dihedral 
subgroup of $D$ of order $2^{k-1}$, and let $H$ be the preimage of~$E$. Since $H$ is the preimage of a dihedral $2$-group, $H$ is
a generalized quaternion group.
\end{proof}

What is required to force a $2$-group to be cyclic? The following gives an answer:

\begin{proposition}
Let $G$ be a $2$-group of order at least $4$. Then $G$ is cyclic if and only if $G$ has a unique  subgroup of 
order~$4$ and that subgroup is cyclic.
\end{proposition}

\begin{proof}
If $G$ is cyclic, then it has a unique subgroup of any order dividing $|G|$, and that subgroup is unique.
So the result follows.

Now suppose $G$ has a unique subgroup $H$ of order $4$, and that $H$ is cyclic.
Note that $2$ divides the order of the center of $G$ (Proposition~\ref{center_prop}),
so $G$ contains a unique element of order 2~by Lemma~\ref{p2top3}.
By Theorem~\ref{unique_2_theorem},
$G$ is either cyclic or a generalized quaternion group. We just need to eliminate the second possibility.
If $G$ is a generalized quaternion group, then it has a subgroup isomorphic to the quaternion subgroup with $8$ elements,
and that subgroup contains three subgroups of order 4.
\end{proof}


\chapter{Toward the General Classification} \label{classification_section}

Now that we have classified the $p$-groups that are freely representable, we turn
to the general case. 
This document gives a self-contained treatment in the solvable case, but  when we get to the non-solvable case 
 some results will be given without proof,  with citations to the group-theoretic literature instead.\footnote{At least in this version of this document. The hope is to give a unified account that incorporates the non-solvable case in a sequel, or a
 future version of this report.}

A necessary but far from sufficient condition for a group $G$ to be freely representable is that every Sylow subgroup
of $G$ is freely representable. We have by our previous results that a~$p$-Sylow subgroup $G$ with $p$ odd will be freely representable if and only if $G$ is cyclic.
A $2$-Sylow subgroup of $G$ is freely representable if and only if it is cyclic or a generalized quaternion group. 
We have an informal notion of
\emph{cycloidal} where we want this adjective to apply to families of groups that behave in important ways like cyclic groups.
We won't commit to a final definition of ``cycloidal'' here, but will keep the notion a bit open. Let's agree to (1) admit
all finite subgroups of $\bH^\times$ as cycloidal. Let's also agree that (2) every subgroup of a cycloidal group is cycloidal, and (3) an Abelian group is cycloidal if and only if it is actually cyclic.
 These stipulations are enough to force a definition of cycloidal in the case of $p$-groups (see Theorem~\ref{pgroup_thm}, 
 Proposition~\ref{p^2_prop}, and Corollary~\ref{unique_2_cor}):

\begin{definition}
Let $p$ be an odd prime.
Then a $p$-group is \emph{cycloidal} if it is cyclic.
As above, a $2$-group is \emph{cycloidal} if it is cyclic or a generalized quaternion group.
A group $G$ is \emph{Sylow-cycloidal} if every Sylow subgroup is cycloidal in the above sense.
A group $G$ of even order is \emph{Sylow-cycloidal-quaternion} if every Sylow subgroup is cycloidal and if the $2$-Sylow subgroups are quaternion.
A group $G$ is \emph{Sylow-cyclic} if every Sylow subgroup is cyclic.\footnote{As far as I know, the terminology introduced
in this definition is new.}
\end{definition}

So when we classify freely representable groups we can limit our attention to Sylow-cycloidal groups. We will divide our classification of Sylow-cycloidal groups, and hence freely representable groups, into three  categories:

\begin{enumerate}
\item
\textbf{Sylow-cyclic groups}. These turn out to be
solvable groups.

\item
\textbf{Solvable Sylow-cyclic-quaternion groups}. 

\item
\textbf{Non-Solvable Sylow-cyclic-quaternion groups}. 
\end{enumerate}

\begin{remark}
According to the Sylow theorems,
all $p$-Sylow subgroups are conjugate hence isomorphic. So to determine if $G$ is Sylow-cycloidal or Sylow-cyclic, it
is enough to look at one $p$-Sylow subgroup for each prime $p$.
\end{remark}

\begin{remark}
It is interesting to note that every freely representable group has the property that
all its Sylow subgroups are isomorphic to subgroups of $\bH^\times$. Also note that Sylow-cyclic groups
are those whose Sylow subgroups are isomorphic to subgroups of $\bC^\times$.
\end{remark}

\begin{remark}
Although not all Sylow-cycloidal groups are freely representable, the collection of  Sylow-cycloidal groups is
of great interest since such groups are exactly the groups whose Abelian subgroups are all cyclic.
A proof of this fact is provided below (Theorem~\ref{abelian_cyclic_thm}).
Such groups arise in the classification of groups of periodic cohomology (see Wall~\cite{wall2013}).\footnote{Unfortunately
I have not explored periodic cohomology in this document, but I hope to cover this topic either in a sequel or
in a future version of this report.}
\end{remark}


\section{Sylow-Cycloidal Groups} 

Now we consider some basic observations concerning Sylow-cycloidal groups.

\begin{proposition} \label{cycloidal_subgroup_prop}
Let $G$ be a Sylow-cycloidal group. Then 
every subgroup $H$ of $G$ is a Sylow-cycloidal group.
\end{proposition}

\begin{proof}
Let $H_p$ be a $p$-Sylow subgroup of $H$ for a prime $p$ dividing $|G|$.
By a Sylow theorem  (Theorem~\ref{sylow2_thm})  $H_p$ is a subgroup of a $p$-Sylow subgroup $G_p$ of $G$.
When~$p$ is odd, $G_p$ is cyclic, and so the subgroup $H_p$ of $G_p$ must also be cyclic.
Similarly, 
if $p=2$ then $H_p$ must be a~$2$-cycloidal group (Proposition~\ref{2_cycloid_subgroup_prop}).
\end{proof}

\begin{proposition} \label{cycloidal_abelian_prop}
Let $A$ be an Abelian Sylow-cycloidal group. Then $A$ is cyclic, hence $A$ is freely representable.
\end{proposition}

\begin{proof}
Every Abelian group is the product of its Sylow subgroups.\footnote{Suppose $A$ and $B$ are subgroups 
of an Abelian group with $A\cap B = \{1\}$, and consider the natural map~$A \times B \to AB$.
This must be an isomorphism.} 
Since $A$ is Abelian, all its Sylow subgroups must be Abelian, and hence cyclic.
The product of cyclic subgroups of pairwise relatively prime orders is cyclic.
\end{proof}

\begin{remark}
This gives another proof of Lemma~\ref{abelian_cor} that is less dependent on representation theory
or the fact that all finite subgroups of $\bC^\times$ is cyclic.
\end{remark}

\begin{proposition}  \label{cycloidal_quotient_prop}
If $G$ is a Sylow-cycloidal group, and $N$ is a normal subgroup of odd order, then $G/N$ is a Sylow-cycloidal group.
\end{proposition}

\begin{proof}
If $G_p$ is a $p$-Sylow subgroup, then its image $G_p'$ in $G/N$ is a $p$-Sylow subgroup of $G/N$,
assuming $p$ divides the order of $G/N$.\footnote{This is a nice general fact.
To see it, note that the restricted canonical map $G_p \to G/N$ has kernel
$G_p \cap N$. The largest power of $p$ dividing $|G/N|$ times the  largest power of $p$ dividing~$|N|$ is equal to the order of $G_p$.
Similarly, $|G'_p|$ times~$|G_p \cap N|$
is also equal to the order of $G_p$.
An inequality forces equality: $|G_p'|$ is equal to the largest power of $p$ dividing $|G/N|$.
Note also that~$G_p \cap N$ must be the $p$-Sylow subgroup of $N$.}
If $p$ is odd then $G_p'$ is the image of a cyclic group, so $G_p'$ is itself cyclic.
If $p = 2$ then the restricted canonical map~$G_p \to G/N$ has trivial kernel $G_p \cap N$ since $|N|$ is odd, so $G_p'$ is isomorphic to~$G_p$, and so
is a $2$-cycloidal group.
\end{proof}

\begin{proposition} \label{cycloidal_product_prop}
Suppose $A$ and $B$ are subgroups of $G$ such that $|A|$ and $|B|$ are relatively prime
and such that $|A||B| = |G|$. Then $G$ is Sylow-cycloidal if and only if both $A$ and $B$ are Sylow-cycloidal.
\end{proposition}

\begin{proof}
One direction follows just because $A$ and $B$ are subgroups of $G$. So suppose that~$A$ and $B$ are Sylow-cycloidal.
Let $p$ be a prime dividing $|A|$. Every~$p$-Sylow subgroup $P$ of $A$ is in fact a $p$-Sylow subgroup of $G$ (since $p$ cannot divide $|B|$), and is 
either cyclic or is a generalized quaternion group. The same holds for any prime dividing $|B|$.
Since $|G| = |A| |B|$, for any prime $p$ dividing the order of $G$ there is a 
is $p$-Sylow subgroup that is
either cyclic or is a generalized quaternion group. This implies that $G$ is Sylow-cycloidal.
\end{proof}

\begin{corollary}\label{cycloidal_product_cor}
The product or even the semi-direct product of two Sylow-cycloidal groups of relatively prime orders is itself Sylow-cycloidal.
\end{corollary}

\begin{theorem}\label{abelian_cyclic_thm}
Let $G$ be a finite group. Then $G$ is Sylow-cycloidal if and only if every Abelian subgroup of $G$ is cyclic.
\end{theorem}

\begin{proof} 
Suppose that $G$ is Sylow-cycloidal with Abelian subgroup $A$.
Then $A$ is Sylow-cycloidal (Proposition~\ref{cycloidal_subgroup_prop}). Hence $A$ is
cyclic (Proposition~\ref{cycloidal_abelian_prop}).

Now suppose every Abelian subgroup of $G$ is cyclic.
Let $P$ be a $p$-Sylow subgroup of $G$. 
If $H$ is a subgroup of $P$ of order $p^2$ then $H$ is Abelian (Proposition~\ref{p^2_prop})
and so is cyclic.  If $p$ is odd then $P$ must itself be cyclic (Theorem~\ref{pgroup_thm}). 
If $p=2$ then~$P$ must be cyclic or a generalized quaternion group (Corollary~\ref{unique_2_cor}).
Thus $G$ is Sylow-cycloidal.
\end{proof}

We saw above that every even-ordered freely representable group has a unique element of order $2$. This generalizes to many, but not all, Sylow-cycloidal groups. The following lemma is useful in this respect:

\begin{lemma} \label{order2_lemma}
Let $G$ be a finite group. Suppose $G$ has a $p$-Sylow subgroup having a unique subgroup order $p$. Then $G$ has a unique subgroup of order $p$
if and only if~$G$ has a normal subgroup of order $p$.

In particular, if $G$ has a $2$-Sylow subgroup that has a unique element of order $2$ then
$G$ has a unique element of order $2$ if and only if $G$ has a normal subgroup of order $2$.
\end{lemma}

\begin{proof}
One direction is clear: a unique subgroup of order $p$ is characteristic hence normal in $G$.

Conversely, suppose $A$ is a normal subgroup of $G$ of order $p$, and let $B$ be any subgroup of $G$ order $p$.
By a Sylow theorem (Theorem~\ref{sylow2_thm}), $A$ is contained in a~$p$-Sylow subgroup~$P_1$, and $B$ is contained in 
a $p$-Sylow subgroup~$P_2$.
By a Sylow theorem (Theorem~\ref{sylow3_thm}),~$P_1$ and $P_2$ are conjugate. We also know that $A$ is the unique subgroup of~$P_1$ of order $p$,
and $B$ is the unique subgroup of $P_2$ of order $p$ (uniqueness for all $p$-Sylow subgroups follows from the fact that all $p$-Sylow subgroups are conjugate).
This uniqueness property implies that
the conjugation map that carries $P_1$ to $P_2$ must  carry~$A$ to~$B$. But $A$ is normal and so conjugation
carries $A$ to itself. Thus $A = B$ as desired.
\end{proof}


\chapter{Sylow-Cyclic Groups} \label{sylow_cyclic_chapter}

The term ``Sylow-cyclic'' might be new, but these groups have a long history
with contributions by early group theory pioneers such as H\"older, Frobenius,
Burnside, and later in the 1930s by Zassenahaus.\footnote{Suzuki~\cite{suzuki1955} and a few other authors call these groups ``$Z$-groups'', perhaps as a pun on ``Zassenhaus'' and ``Zyklische Gruppe''. I started using
``Sylow-cyclic'' before learning of the term ``$Z$-group''. I have a slight preference for the more
descriptive term ``Sylow-cyclic''.}
They have been classified since the early
20th century, and have a relatively simple structure: they are the semi-direct products of cyclic
groups of relatively prime order.
Some, but not all, of these groups are freely representable:

\begin{example}
Any group of square-free order is Sylow-cyclic.
This includes dihedral groups with $2 n$ elements where $n$ is square free and odd.
Such dihedral groups cannot be freely representable since they do not have a unique element of order~$2$.
\emph{Thus there are an infinite number of orders such that there exists Sylow-cyclic groups
of that order that are not freely representable.}
Later we will see that if a non-Abelian group has square-free order then it cannot be freely representable.
\end{example}

\begin{example}
Let $G$ be a binary dihedral group $2 D_n$ where $n$ is a square free odd integer. Such groups occur 
as subgroups of $\bH^\times$ so are freely representable. They must then be Sylow-cyclic (since the $2$-Sylow subgroups
have four elements). So there are infinite number of orders of Sylow-cyclic groups that are freely representable.
Later we will see there are an infinite number of odd orders of non-Abelian Sylow-cyclic groups that are freely representable,
and none of these can be isomorphic to a subgroup of~$\bH^\times$.
These will be our first examples of freely representable groups that are not isomorphic to subgroups of~$\bH^\times$.
\end{example}

We continue with some basic observations about Sylow-cyclic groups:

\begin{proposition} \label{abelian_cyclic_prop2}
Any Abelian Sylow-cyclic group is cyclic, and hence is freely representable.
\end{proposition}

\begin{proof}
This is a special case of Proposition~\ref{cycloidal_abelian_prop}.
\end{proof}

\begin{proposition}
Any subgroup $H$ of a Sylow-cyclic group $G$ is also a Sylow-cyclic group.
\end{proposition}

\begin{proof}
Let $H_p$ be a $p$-Sylow subgroup of $H$ for a prime $p$ dividing $|G|$.
By a Sylow theorem  (Theorem~\ref{sylow2_thm}), $H_p$ is a subgroup of a $p$-Sylow subgroup $G_p$ of $G$.
Since~$G_p$ is cyclic, the subgroup $H_p$ of $G_p$ must also be cyclic.
\end{proof}

\begin{proposition}
Any quotient group $G/N$ of a Sylow-cyclic group $G$ is also a Sylow-cyclic group.
\end{proposition}

\begin{proof}
The proof is similar to the proof of Proposition~\ref{cycloidal_quotient_prop}. The basic idea is that any Sylow subgroup 
of $G$ maps to a Sylow subgroup of $G/N$ (or a trivial group).
\end{proof}

\begin{proposition} 
Suppose $A$ and $B$ are subgroups of $G$ such that $|A|$ and $|B|$ are relatively prime
and such that $|A||B| = |G|$. Then $G$ is Sylow-cylic if and only if both $A$ and $B$ are Sylow-cyclic.
\end{proposition}

\begin{proof}
The proof is similar to that of Proposition~\ref{cycloidal_product_prop}.
\end{proof}

\begin{corollary}\label{semi-direct_cor}
The product or semi-direct product of two Sylow-cyclic groups of relatively prime orders is itself Sylow-cyclic.
\end{corollary}

\begin{example}
Let $p$ be an odd prime and let $q$ be any prime dividing $p-1$.
The cyclic group $C_p$ of size $p$ has $\bF_p^\times$ for its automorphism group, which is cyclic and contains
a unique subgroup of order $q$. Let $C_q$ be a cyclic group of size $q$, and fix an isomorphism
with the subgroup of $\bF_p^\times$ of size $q$. This gives a nontrivial action of~$C_q$ on $C_p$.
Thus the semidirect product $G = C_p \rtimes C_q$ is a non-Abelian Sylow-cyclic group whose order is a multiple of $p$.
By Theorem~\ref{pq_thm}, $G$ is not freely representable. 
For example, we can construct a Sylow-cyclic group of order 21 that is not freely representable; here $p = 7$ and $q = 3$.\end{example}

The converse of Corollary~\ref{semi-direct_cor} true, and is a deeper result (proved by Burnside by~1905). 
Toward this end, we start with the following interesting result:\footnote{The proof is adapted from M.~Hall~\cite{hall1959}, 
proof of Theorem~9.4.3.}

\begin{theorem}\label{largest_prime_thm}
Suppose that $G$ is Sylow-cyclic and that $q$ is the largest prime dividing $|G|$.
Then $G$ has a unique $q$-Sylow subgroup and this subgroup is normal.
\end{theorem}

\begin{proof}
We start by defining a \emph{large-prime divisor} of $|G|$ to be a divisor $d$ of $|G|$
such that the primes dividing $d$ are as large as or larger than any prime dividing $|G|/d$.
Then we establish the following claim: for every large-prime divisor $d$ of $|G|$, the set
$E_d = \{ x \in G \mid x^d = 1\}$ has exactly $d$ elements. 

To establish the claim, let $d$ be the largest large-prime divisor that is a counter-example.
Note that $d < |G|$ since the claim holds for $d = |G|$.
Let $p$ be the largest prime divisor of 
the complementary divisor~$|G|/d$.
Then $p d$ is a large-prime divisor, and we get to assume that $E_{p d}$ has $pd$ elements.
Let $p^k$ be the largest power of $p$ dividing~$d$. Any~$p$-Sylow subgroup of $G$
is cyclic, and so
has an element of order~$p^{k+1}$. Such an element is in $E_{pd}$ but not in $E_d$.
So $E_{pd}$ is strictly larger than $E_d$.

We partition the nonempty set $E_{pd} - E_d$ by declaring two elements equivalent if they generate the same
cyclic subgroup. If $g \in E_{pd} - E_d$ has order $t$, then the set of elements that generate $\left< g \right>$
has $\phi(t)$ elements. Now $p-1$ divides $\phi(t)$ since the order of $t$ is divisible by $p$.
So $E_{pd} - E_d$ has size~$m (p-1)$ for some positive~$m$.
Next, we note that $E_d$ has order $n d$ for some positive $n$ by Frobenius's theorem (Theorem~\ref{frobenius_thm}).
Thus
$$
p d = n d + m(p-1), \qquad\text{so}\quad d \mid m (p-1).
$$
By choice of $p$, there is no prime divisor of $d$ strictly smaller than $p$. Thus $p-1$ is relatively prime to $d$.
So $d$ actually divides $m$. We write $m = m_0 d$ for $m_0> 0$, and
$$
p d = n d + m_0 d (p-1), \qquad\text{so}\quad p  = n + m_0 (p-1).
$$
The only solution to this last equation with positive $n$ and $m_0$ is $n=m_0 = 1$.
This shows that $E_d$ has $d$ elements, so $d$ cannot be a counter-example.

Now let $p$ be the largest prime divisor of $|G|$ and let $p^l$ be the largest prime power of $p$ dividing $|G|$.
Every $p$-Sylow subgroup of $G$ is contained in $E_{p^l}$.
Since $p^l$ is a large-prime divisor of $|G|$, we must have that every $p$-Sylow subgroup of $G$ is 
the set $E_{p^l}$
since they have the same size.
So there is a unique $p$-Sylow subgroup, and it is normal since $E_{p^l}$ is invariant under conjugation.
\end{proof}

By repeated application of the above result we get the following:

\begin{corollary}\label{cyclic_quotient_cor}
Suppose that $G$ is Sylow-cyclic and that $p$ is the smallest prime dividing the order of $G$.
Then $G$ has a quotient of order $p$. In particular, $G$ has a nontrivial Abelian quotient.
\end{corollary}

We also get the following result by repeatedly applying the above theorem.\footnote{I have seen this result attributed to H\"older (1859--1877).}

\begin{corollary}
Every Sylow-cyclic group is solvable.
\end{corollary}

An important step in the classification is the following:

\begin{proposition}\label{sylow_cyclic_commutator_prop}
Let $G'$ be the commutator subgroup of a Sylow-cyclic group $G$. Then
$G'$ and $G/G'$ are cyclic.
\end{proposition}

\begin{proof}
Consider the derived series. In other words, 
recursively define $G^{(n+1)}$ to be the commutator subgroup of $G^{(n)}$ starting with~$G^{(1)}=G'$.
By Corollary~\ref{cyclic_quotient_cor}, if~$G^{(n)}$ is nontrivial then $G^{(n+1)}$ is a proper subgroup
of $G^{(n)}$. Thus $G^{(n)} = 1$ for sufficiently large $n$.

The result is clear for Abelian $G$, so assume
$G$ is non-Abelian. Let $m$ be the largest integer such that $G^{(m)}$ is nontrivial.
Since $G^{(m)}$ is an Abelian Sylow-cyclic group, it is cyclic.
Observe that $G$ acts on~$G^{(m)}$ by conjugation.\footnote{Any automorphism of a group restricts
to an automorphism of its commutator subgroup. Thus, by induction,  conjugation automorphisms (inner automorphisms) on $G$ restricts to
automorphisms of~$G^{(m)}$.} 
Thus $G/A$ is isomorphic to a subgroup of the automorphism group of $G^{(m)}$ where
$A$ is the centralizer of $G^{(m)}$. Since $G^{(m)}$ is cyclic, its automorphism group 
is~$(\bZ/t\bZ)^\times$ where~$t$ is the order of~$G^{(m)}$.
Thus $G/A$ is Abelian, which means that $G' \subseteq A$. In other words, if $g \in G'$ and~$h\in G^{(m)}$ 
then $g$ and $h$ commute.

Suppose $m > 1$. Then $G^{(m-1)}/G^{(m)}$ is an Abelian Sylow-cyclic group, so is cyclic. 
Let $g \in G^{(m-1)}$
be such that its image in $G^{(m-1)}/G^{(m)}$ is a generator. Since~$g\in G'$,
it commutes with every element of~$G^{(m)}$.
So $G^{(m-1)}$ is Abelian since~$G^{(m-1)} = \left<g\right> G^{(m)}$. This 
implies $G^{(m)}$ is trivial, a contradiction. So~$m=1$, as desired.

Since $m=1$,  $G' = G^{(m)}$ is cyclic. Since $G/G'$ is Abelian and Sylow-cyclic, it is cyclic as well.
\end{proof}

Now we are ready for Burnside's result on the structure of Sylow-cyclic groups:

\begin{theorem}\label{semidirect_thm}
Let $G$ be a  Sylow-cyclic group and let $A$ be the commutator subgroup of~$G$.  
Then, as above, both $A$ and $G/A$ are cyclic. Let $m$ be the order of $A$ and let~$n$ be the 
order of~$G/A$,  so $G$ has order $mn$. 
Let $b\in G$ be any element mapping to a generator of $G/A$, and let $B$ be the subgroups of $G$
generated by $b$.
Then the following hold:
\begin{itemize}
\item
The subgroup $B$ has order $n$, and so its generator $b$ has order $n$.
\item
The subgroup $B$ is a complement of $A$ in $G$, so $G = A B = A \rtimes B$.
\item
The subgroups $A$ and $B$ have relatively prime orders $m$ and $n$.
\item
The order $m$ of $A$ is odd.
\item
The normalizer of $B$ and centralizer of $B$ in $G$ are both equal to $B$.
\item
There are $m$ subgroups of $G$ of order $n$, and they are all conjugate in $G$ to $B$.
\item
There there is a pair of subgroups of order $n$ that together generates all of $G$.
\end{itemize}
\end{theorem}

\begin{proof}
By the previous proposition, $A$ and $G/A$ are cyclic. Choose $b\in G$ so that its image in $G/A$ is a generator,
and choose $a$ to be a generator of $A$.
Let $B = \left< b \right>$. Observe that $a$ and $b$ together generate $G$ and $G=AB$.
Consider the commutator 
$$a' \; \defeq \; [a, b] = a^{-1} b^{-1} a b \in A.$$
Every subgroup of $A$ is normal in $G$ (since $A$ is normal in $G$, and has at most one subgroup of any given order),
so the group $A'$ generated by $a'$ is normal in $G$.
Note that $G/A'$ is generated by the images of $a$ and $b$,
and these images commute. So~$G/A'$ is Abelian, which means that $A'$ contains the commutator subgroup~$A$.
In other words, $A = A'$. 

Since $a$ and $a'$ commute, we have $[a^t, b] = (a')^t$ for all positive integers $t$.
Suppose now that~$a^t$ commutes with $b$. Then
$$
1 = [a^t, b] = (a')^t.
$$
So $t$ must be a multiple of the order of $A$ since $a'$ is a generator of $A$.
Hence $a^t = 1$. So~$1$ is the only element of $A$ that centralizes $B$.
In other words $A \cap Z(B)= \{1 \}$. Of course the center $Z(G)$ is a subgroup of $Z(B)$, and
$B$ is also a subgroup of $Z(B)$ since $B$ is cyclic.
So $A \cap Z(G) = \{1\}$ and $A \cap B = \{1\}$.
This last equation shows that $B$  is a complement to $A$
and so $B$ has size $n$. In particular, the restriction of~$G \to G/A$ to $B$
is an isomorphism $B \to G/A$.

Now suppose $g = a^s b^t$ normalizes $B$. Since $B$ is cyclic, this means that $a^s$ normalizes $B$.
Note that $a^{s} b a^{-s} \in B$ and $b \in B$ map to the same element under the isomorphism $B \to G/A$
so $a^{s} b a^{-s} = b$. 
Thus $a^s \in A \cap Z(B) = \{1\}$. We conclude that the normalizer of $B$
is just $B$ itself: $N(B) = Z(B) = B$.
By the stabilizer-orbit theorem, there are exactly $m$ distinct conjugates of $B$ in $G$.

Next we show that the orders of $A$ and $B$ are relatively prime. Suppose otherwise that a prime $p$ divides the order
of both $A$ and $B$. Then $A$ and $B$ each have a subgroup of order $p$, call them $A'$ and $B'$.
Note that $A'$ and $B'$ are distinct since~$A \cap B = \{ 1 \}$.
Since $A'$ is the unique subgroup of the normal subgroup~$A$ of order~$p$, the subgroup $A'$ must be normal in $G$.
This means $A' B'$  is a group of order~$p^2$. Let $C$ be a $p$-Sylow subgroup of $G$ containing $A'B'$.
Then $C$ contains at least two subgroups ($A'$ and $B'$) of order $p$, contradicting the fact that $C$ is cyclic.

Next we show that $A$ has odd order. This is clear if $G$ has odd order so suppose that~$G$ has even order.
By Corollary~\ref{cyclic_quotient_cor}, $G$ has a quotient  $G/K$ isomorphic to the~$2$-Sylow subgroup of $G$, so $K$ has odd order.
The result follows from the fact that $G/K$ is Abelian and that $A$ is the commutator subgroup so $A \subseteq K$.

Now we show that every subgroup of $G$ of order $n$ is conjugate to $B$.
Since $B$ has exactly $m$ conjugates, it is enough to show that there are at most $m$ subgroups of~$G$ of order $n$.
Suppose $H$ is a subgroup of order $n$. Then its image under the map~$G \to G/A$ must be all of $G/A$ since
$A \cap H = \{1\}$, and the restriction $H \to G/A$ is an isomorphism. The image of $b$ in $G/A$ generates $G/A$, so $H$ has an element of the form $a^t b$ which generates $H$. There are only $m$ elements of the form $a^t b$,
so there can be at most $m$ subgroups of order $n$.

Finally, note that the above arguments apply to any other choice of $b$. In particular, if $b$ is the original
choice, then $ab$ also has order $n$. The group generated by $\left<b\right>$ and $\left< ab \right>$
contains $b$ and $a$, so is all of $G$.
\end{proof}

This yields an interesting corollary:

\begin{corollary}\label{center_cor}
Let $G$ be a Sylow-cyclic subgroup with commutator subgroup $A$.
Then the center $Z(G)$ of $G$ has the property that $A \cap Z(G) = \{1\}$.
In fact, $Z(G)$ is the intersection of all subgroups $B$ of $G$ of order $n = |G/A|$.
\end{corollary}

\begin{proof}
By the above theorem there is a cyclic complement $B_0$ of $A$ such that~$Z(B_0) = B_0$.
Thus $Z(G) \subseteq Z(B_0) = B_0$. All subgroups of $G$ of size $n$ are conjugate to $B$,
so~$Z(G)$ is contained in the intersection $Z_0$ of all subgroups of~$G$ of order $n$.
Also~$A \cap Z(G) = \{1\}$ since $A \cap B_0= \{1\}$.

Note that $Z_0$ commutes with each subgroup of order $n$ since such groups are cyclic.
Since groups of order $n$ generate~$G$, we conclude that $Z_0 = Z(G)$.
\end{proof}

Here is another application of the above theorem, one that shows a strong parallelism between
Sylow-cyclic groups and cyclic groups:

\begin{theorem}\label{divisor_thm}
Let $G$ be a Sylow-cyclic group of order $N$. For any divisors~$d$ of~$N$, there is a subgroup 
of $G$ of order $d$, and all subgroups of order $d$ are conjugate in~$G$.
%
\end{theorem}

\begin{proof}
By the above theorem, we can write $G$ as $A \rtimes B$ where $A$ and $B$ are cyclic of relatively prime orders
and where $A$ is the commutator subgroup of $G$.

Let $d$ be a divisor of $|G|$ and write $d = e f$ where $e$ divides $|A|$ and 
$f$ divides~$|B|$. Let $A_{e}$ be the unique subgroup of $A$ of order $e$. Since $A_{e}$ is the unique subgroup of~$A$
of order $e$, and since $A$ is normal in $G$, it follows that~$A_{e}$ is normal in $G$.
Let~$B_{f}$ be the unique subgroup of $B$ of order $f$. 
Since $A_{e} \cap B_{f} = \{1\}$, the subgroup~$A_{e} B_{f}$ has order $d = e f$.

Now suppose that $H_1$ and $H_2$ are subgroups of order $d$. Replacing $G$ if necessary,
we can assume that $G$ is generated by $H_1$ and $H_2$. (And we choose $A$ and $B$ for the new $G$).
As before,  write $d = e f$ where~$e$
divides $|A|$ and $f$ divides $|B|$. Note that under the map $G \to G/A$, both~$H_1$ and~$H_2$
have images of size $f$, and the restriction $H_i \to G/A$ has kernel $A \cap H_i$ of size~$e$.
Since $A$ and $G/A$ are cyclic, this means that $H_1$ and $H_2$ have the same image in~$G/A$
and that $A \cap H_1 = A \cap H_2$. 

Let $A_0 = A \cap H_1 = A \cap H_2$ and 
let $G_0/A$ be the image of $H_i$ in $G/A$, where~$G_0$ is a subgroup of $G$ containing~$A$.
Since $G$ is generated by $H_1$ and $H_2$, we must have~$G_0 = G$, and so $f = n = |G/A|$. 
Let $B_i$ a subgroup of $H_i$ of order $n$ (which exists since $H_i \to G/A$ is surjective). Note that $A_0 B_1 = H_1$ and $A_0 B_2 = H_2$.
By Theorem~\ref{semidirect_thm}, $B_1$ and $B_2$ are conjugate subgroups of $G$. Thus $H_1$ and $H_2$ 
are conjugate.
\end{proof}

\begin{remark}
Let $G$ be a Sylow-cyclic group of order $N$.
In light of the above theorem, 
we observe that for any divisor~$d$ of $N$ the following are equivalent: (1)~$G$ has a unique subgroup of order~$d$,
(2) $G$ has a characteristic subgroup of order $d$, and (3)~$G$ has a normal subgroup of order~$d$.

We can classify divisors $d$ of $|G|$ based on whether or not there is a unique subgroup of order~$d$. Alternatively
we can classify divisors $d$ based on 
whether or not the subgroups of order $d$ are cyclic.  
We are particularly interested in the divisors $d$ for which both conditions hold: there is a normal cyclic subgroup of order~$d$.
Since every subgroup of a cyclic normal subgroup is cyclic normal, we conclude
that if there is a normal cyclic subgroup of of order $d$, then the same is true for each divisor of $d$.
The following proposition and corollaries give us further result about such subgroups that are normal and cyclic.
\end{remark}

\begin{lemma} 
Let $G$ be a finite group with the property that any two subgroups of the same prime power order are conjugate.
If $A$ and $B$ are cyclic normal subgroups of~$G$ then~$AB$ is a cyclic normal subgroup of order equal to the 
least common multiple of $|A|$ and $|B|$.
\end{lemma}

\begin{proof}
Note that $AB$ is a normal subgroup of $G$ since $A$ and $B$ are normal. So we just need to show $AB$
is cyclic of the specified order.

First we consider the case where $A$ and $B$ are of relatively prime order. 
Since~$A$ and $B$ are normal, and since $A \cap B = \{1\}$, we have that $AB$ 
 is isomorphic to $A \times B$. Here $A \times B$ is cyclic
of order equal to $|A| |B|$ and so the result holds. 

Next we consider the case where $A$ and $B$ are $p$-groups for the same prime~$p$.
Either~$|A|$ divides $|B|$ or $|B|$ divides $|A|$. Suppose, say, that $|A|$ divides $|B|$.
Then~$A$ has the same order as a subgroup of $B$ (Theorem~\ref{sylow1_thm}), but there is a unique subgroup of order $|A|$ in $G$ since $A$ is normal.
So $A\subseteq B$ and so $AB = B$. The  result holds in this case as well.

In the general case, we have
$$
A B = P_1 \cdots P_k Q_1 \cdots Q_l
$$
where each $P_i$ is a Sylow subgroup of $A$ and each $Q_j$ is a Sylow subgroup of $B$.
If~$X$ and $Y$ are two Sylow subgroups in the above product then $XY = YX$ since~$X$ and~$Y$ are normal
subgroups of $G$ (all subgroups of a cyclic normal subgroup are cyclic normal subgroups). Thus we can rearrange the terms and reduce to the special cases listed above
\end{proof}

\begin{corollary} \label{subgroups_cg_corollary}
Let $G$ be a Sylow-cyclic subgroup of order $N$. 
Then there is a maximum cyclic characteristic (MCC) subgroup $\mu(G)$ of $G$
that contains all normal cyclic subgroups. Every subgroup of $\mu(G)$ is normal cyclic, and is characteristic.
\end{corollary}

\begin{proof}
By Theorem~\ref{divisor_thm}, a subgroup of $G$ is normal if and only if it is characteristic.
Define $\mu(G)$ to be the product of all normal cyclic subgroups of $G$. 
Theorem~\ref{divisor_thm} allows us to use the above lemma to conclude that $\mu(G)$ is a normal
cyclic subgroup, so is characteristic.
By construction it contains all normal cyclic subgroups of~$G$, so is the maximum (under the inclusion relation) among cyclic characteristic subgroups.
Note that every subgroup of a normal cyclic group is a normal cyclic group.
\end{proof}

The following shows that a necessarily condition for a Sylow-cyclic group $G$  of even order to be freely representable is that 
$\mu(G)$ have even order.

\begin{corollary}\label{unique_cyclic_cor}
Let $G$ be a Sylow-cyclic subgroup of order $N>1$ and let $D$ be the order of the
maximal cyclic characteristic subgroup $\mu(G)$ of $G$.
Then the maximum prime divisor of $N$ divides $D$ to the same order as it divides $N$, so $D>1$.
Also $G$ has a unique cyclic subgroup of order $d$ if and only if $d$ divides $D$.
Thus $G$ has a unique element of order two if and only if $D$ is even.
\end{corollary}

\begin{proof}
The first claim follows from Theorem~\ref{largest_prime_thm} and the previous corollary.
The other claims follow directly from the previous corollary. 
\end{proof}

We can give other useful characterizations of~$\mu(G)$. We start with the following lemma:

\begin{lemma}
Let $G$ be a Sylow-cyclic group with commutator subgroup $G'$. Then the centralizer $Z(G')$
of $G'$ in $G$ is equal to $G' \cdot Z(G)$ where $Z(G)$ is the center of $G$. Furthermore $G'$ and $Z(G)$ have relatively prime orders and 
$G' \cdot Z(G)$ is a characteristic cyclic subgroup isomorphic to $G' \times Z(G)$.
\end{lemma}

\begin{proof}
Let $B$ be a cyclic complement of $G'$. Suppose $ab \in Z(G')$ where $a \in G'$ and~$b \in B$.
Of course $a \in Z(G')$ since $G'$ is cyclic, so $b \in Z(G')$. Also $b\in Z(B)$ since~$B$ is cyclic,
so $b \in Z(G'B) = Z(G)$. We conclude that $Z(G') \subseteq G'\cdot Z(G)$. 
The other inclusion is clear.

Since $Z(G) \subseteq B$ and so $G' \cap Z(G) = \{ 1\}$ (Corollary~\ref{center_cor}) we have that $Z(G)$ is  normal  of relatively prime order to $|G'|$ and so
$G' \cdot Z(G)$ must be isomorphic to the cyclic group $G'\times Z(G)$.
\end{proof}

\begin{proposition}\label{MCC_prop}
Let $G$ be a Sylow-cyclic group with commutator subgroup $G'$
and  maximum cyclic characteristic subgroup $\mu(G)$. Then
$$
\mu(G) = Z(G') = G' \cdot Z(G)
$$ 
where $Z(G')$ is the centralizer of $G'$ in $G$, and where $Z(G)$ is the center of $G$.
\end{proposition}

\begin{proof}
Since $\mu(G)$ is an Abelian group containing $G'$, we have $\mu(G) \subseteq Z(G')$.
By the above lemma, $Z(G')$ is a characteristic cyclic subgroup of $G$, which forces equality.
\end{proof}

A necessary condition for a Sylow-cyclic group of even order to be freely representable is that its center
have even order:

\begin{corollary}\label{new_corollary2}
Let $G$ be a Sylow-cyclic group with 
  maximal cyclic characteristic subgroup $\mu(G)$ and center $Z(G)$.
Then $G$ has a unique element of order $2$ if and only if the center $Z(G)$ has even order.
\end{corollary}
\begin{proof}
This follows from the fact that $\mu(G) = G' \cdot Z(G)$ (as above) and the fact that
the commutator subgroup $G'$ has odd order (Theorem~\ref{semidirect_thm}).
\end{proof}

Above we considered subgroups that are normal and cyclic. Now we add a third requirement
that the quotient be cyclic as well:

\begin{definition}
Let $G$ be a finite group. A \emph{metacyclic kernel} $C$ is any normal subgroup of $G$
such that both $C$ and $G/C$ are cyclic.\footnote{This is not a standard term in the literature, as far as I know. I came up with this terminology  based on the term \emph{metacyclic group} which is a fairly standard term
for a group $G$ with a normal subgroup $C$ 
such that $C$ and $G/C$ are both cyclic.}
For example, we have shown
that the commutator subgroup of a Sylow-cyclic
group is a metacyclic kernel.
Observe that any metacyclic kernel is a characteristic cyclic subgroup of $G$ and so is contained in the MCC subgroup $\mu(G)$.
\end{definition}

\begin{lemma} \label{super_lemma}
Let $G$ be a finite group with a metacyclic kernel $K$. If $C$ is a cyclic subgroup  of $G$ containing $K$ 
then $C$ is also a metacyclic kernel.
\end{lemma}

\begin{proof}
There is a natural correspondence between subgroups of $G/K$ and subgroups of $G$ containing $K$, and that this correspondence 
is well-behaved
for normal subgroups. Let $\overline C$ be the image of $C$ in $G/K$.
Since $G/K$ is cyclic, $\overline C$ is a normal subgroup of $G/K$. It follows that $C$ is a normal subgroup of $G$. The quotient
$G/K$ by $\overline C$ is cyclic, since $G/K$ is cyclic. But this quotient is isomorphic to $G/C$.
\end{proof}

\begin{lemma} \label{metacycle_range_lemma}
Let $G$ be a Sylow-cyclic subgroup. Let $G'$ be the commutator subgroup
and let $Z(G')$ be the centralizer of $G'$. Then a subgroup $C$ of $G$ is a metacyclic
kernel if and only if 
$$
G' \subseteq C \subseteq Z(G').
$$
\end{lemma}

\begin{proof}
First suppose that $C$ is a metacyclic kernel. Since $G/C$ is cyclic, hence Abelian, it follows that $G' \subseteq C$.
Since $C$ is cyclic, hence Abelian, $C \subseteq Z(G')$. So one implication is established.

Suppose $G'\subseteq C \subseteq Z(G')$. Since $Z(G')$ is cyclic (Proposition~\ref{MCC_prop}),
we have that~$C$ is cyclic. So $C$ is a metacyclic kernel by the above lemma.
\end{proof}

\begin{remark}
Let $G$ be a Sylow-cyclic group and let $Z(G')$ be the centralizer of the commutator group $G'$ in $G$,
which is also equal to the MCC subgroup of $G$ (Proposition~\ref{MCC_prop}).
Then  $Z(G') = \mu(G)$ is the \emph{maximum metacyclic kernel}, and 
$G'$ the \emph{minimum metacyclic kernel}. These groups are important invariants of $G$,
and~$\mu(G)$ will be used to identify whether or not $G$ is freely representable.
\end{remark}

\begin{corollary} \label{MCC_maximal_cor}
Let $G$ be a Sylow-cyclic group
and let $\mu(G)$ be the MCC subgroup of $G$.
Then $\mu(G)$ is a maximal cyclic subgroup of $G$.
\end{corollary}

\begin{proof}
Since $\mu(G)$ is the maximum among metacyclic kernels, it is maximal among all cyclic groups by
Lemma~\ref{super_lemma}.
\end{proof}

In practice $\mu(G)$ can be calculated using automorphisms from any semidirect decomposition of $G$
into cyclic groups of relatively prime order:

\begin{lemma}\label{mck_lemma}
Suppose $G$ is the semidirect product $A \rtimes B$ of two cyclic groups of relatively prime order.
As usual we identify $A$ and $B$ with subgroup of $G$.
Then the MCC subgroup $\mu(G)$ of $G$ is $AK$ where $K$ is the kernel of the
associated action map $B \to \mathrm{Aut}(A)$.
\end{lemma}

\begin{proof}
Note that $G$ is Sylow-cyclic and that $A$ is a metacyclic kernel.
Since~$K$ acts trivially on $A$ via conjugation, the 
group $AK$ is Abelian. Since $AK$ is Sylow-cyclic,~$AK$ is cyclic. 
Thus $AK$ is a metacyclic kernel (Lemma~\ref{super_lemma}).

By Lemma~\ref{metacycle_range_lemma}, $AK$ is a subgroup of $Z(G') = \mu (G)$.
Let $g \in \mu(G)$. Write~$g$ as~$ab$ where~$a \in A$ and~$b\in B$.
Observe that $b \in \mu(G)$ since $a \in AK \subseteq \mu (G)$. 
Since $\mu(G)$ is an Abelian group containing $A$, we see that $b$ acts trivially on $A$ and so~$b \in K$. Thus~$g = ab$ is in~$AK$.
We conclude that $A K = \mu (G)$.
\end{proof}

We will need the following later in our proof of Wedderburn's theorem.

\begin{corollary}\label{mck_size_cor}
Let $G$ be a Sylow-cyclic group with MCC subgroup $\mu(G)$. Then $$|\mu(G)| > [G:\mu(G)].$$
\end{corollary}

\begin{proof}
Write $G$ as the  semi-direct product $A \rtimes B$ of two cyclic groups of relatively prime order $m = |A|$ and $n = |B|$.
Let $K$ be the kernel of the associated action map~$B \to \mathrm{Aut}(A)$, and let $I$ be the image.
The automorphism group of $A$ is isomorphic to~$(\bZ/m\bZ)^\times$ where $m$ is the order of $A$, so $|I| <  m$ and
$$
n = |B| = |K| |I| < |K| m.
$$
By the above lemma, 
$|\mu(G)| = |A| |K|$, so
$$
|\mu(G)| = m |K| > n.
$$
However, $G/AK$ is a quotient of $G/A \cong B$. So $|G/\mu(G)| \le n$.
\end{proof}

\begin{remark}
Let $A$ be a cyclic group of order $m$ and let $B$ be a group of order $n$, where~$m$ and $n$ are relatively prime.
Let $a$ be a generator of $A$ and let $b$ be a generator of $B$.
We construct a semi-direct product $G$ of $A$ with $B$ by choosing an automorphism~$\sigma$ of $A$ associated
with conjugation by $b$. Such $\sigma$ is of the form~$x\mapsto x^r$ where $r$ is 
any~$r\in \left( \bZ / m \bZ \right)^\times$ of order dividing $n$. So $r^n = 1$ in $\bZ / m \bZ$.
In $G$ we have
$$
b^{-1} a b = a^r, \qquad [a, b] = a^{-1} b^{-1} a b = a^{r-1}.
$$
Clearly the commutator subgroup $G'$ will be contained in $A$ since the quotient is Abelian (isomorphic to $B$).
Since $[a, b] = a^{r-1}$ we have $\left<a^{r-1}\right> \subseteq G'$, and 
since~$G/\left<[a, b]\right>$ is Abelian, we have $G' \subseteq \left<[a, b]\right>$.
Thus 
$$
G' = \left<a^{r-1}\right>.
$$
 If we want $G' = A$ we will
also need $r-1$ to be relatively prime to $m$.
\end{remark}

\begin{example} As an illustration we classify groups $G$ of order $210 = 2 \cdot 3 \cdot 5 \cdot 7$. 
Note that all such groups are Sylow-cyclic. First we sort such groups by 
the size of $A = G'$.
Since $A$ is a proper subgroup of $G$ of odd order, its possible orders are~$1, 3, 5, 7, 15, 21, 35,$ or $105$.
We will use the notation of the above remark.

If $A$ has order $1$ then $G$ is cyclic and $\mu(G) = G$. This gives us the only Abelian example.

If $A$ has order $3$ then we choose $r \in (\bZ / 3 \bZ)^\times$ so that $r-1$ is relatively prime to $3$.
So $r=-1$.
Here $B$ has order $70$ and
the kernel $K$ of the map $B \to \mathrm{Aut}(A)$ has order $35$. So $\mu(G)$ has order $105$
and so all subgroups of odd order are cyclic and normal. Note that in this case there is a subgroup
isomorphic to the dihedral group $D_3$.

If $A$ has order $5$, then $B$ has order $42$. 
We choose $r \in (\bZ / 5 \bZ)^\times$ so that $r-1$ is relatively prime to $5$ and so that $r$ has order dividing $42$.
This means that~$r$ is~$-1$.
The kernel $K$ of the map $B \to \mathrm{Aut}(A)$ has order $21$. So $\mu(G)$ has order $105$
and so all subgroups of odd order are cyclic and normal. Note that in this case there is a subgroup
isomorphic to the dihedral group $D_5$.

If $A$ has order $7$, then $B$ has order $30$. 
Every element of $r \in (\bZ / 7 \bZ)^\times$ has order dividing $30$, so we just need $r-1$ not to be
divisible by $7$. Thus $r = 2, 3, 4, 5, -1$. We divide by subtypes:
\begin{itemize}
\item
If $r = -1$ then  the kernel $K$ of the map $B \to \mathrm{Aut}(A)$ has order $15$.
So $\mu(G)$ has order $105$ and all subgroups of odd order are cyclic and normal. 
Note in this case there is a subgroup isomorphic to $D_7$.
\item
If $r = 2, 4$ then the kernel $K$ of the map $B \to \mathrm{Aut}(A)$ has order $10$.
So $\mu(G)$ has order $70$. 
\item
If $r= 3, 5$ then then the kernel $K$ of the map $B \to \mathrm{Aut}(A)$ has order $5$.
So~$\mu(G)$ has order $35$. 
Note in this case there is a subgroup isomorphic to~$D_7$.
\end{itemize}
Note that if we replace our generator $b$ of $B$ with its inverse
the corresponding~$r$ changes to its multiplicative inverse.
So $r=3,5$ give isomorphic results, and $r=2, 4$ give isomorphic results. So we have
three examples up to isomorphism where $A$ has order $7$.

If $A$ has order $15$ then $B$ has order $14$.
Then we choose $r \in (\bZ / 15 \bZ)^\times$ which is isomorphic to $C_4 \times C_2$.
This means that $r$ should have order $2$, which means that~$r= 4, 11, -1$ . But we also want $r-1$
to be prime to $15$ so $r=-1$ is the only possibility.
The kernel $K$ of the map $B \to \mathrm{Aut}(A)$ has order $7$. So $\mu(G)$ has order~$105$
and so all subgroups of odd order are cyclic and normal. Note that in this case there is a subgroup
isomorphic to the dihedral group $D_{15}$.

If $A$ has order $21$ then $B$ has order $10$.
Then we choose $r \in (\bZ / 21 \bZ)^\times$ which is isomorphic to $C_6 \times C_2$.
This means that $r$ should have order $2$, which means that~$r=  8, 13, -1$. But we also want $r-1$
to be prime to $21$ so $r=-1$ is the only possibilities.
The kernel $K$ of the map $B \to \mathrm{Aut}(A)$ has order $5$. So $\mu(G)$ has order~$105$
and so all subgroups of odd order are cyclic and normal. Note that in this case there is a subgroup
isomorphic to the dihedral group $D_{21}$.

If $A$ has has order $35$ then $B$ has order $6$.
Then we choose $r \in (\bZ / 35 \bZ)^\times$  where~$r-1$ is prime to $35$ and where $r$
has order dividing $6$.
This means that 
$$r \equiv -1 \pmod 5, \qquad r \equiv 2, 3, 4, 5, -1 \pmod 7
$$
This gives 5 possibilities:
$$r=4, 9, 19, 24, -1.$$
We can divide this into two subtypes depending on the order of $r \in (\bZ / 35 \bZ)^\times$:
\begin{itemize}
\item
If  $r=-1$ then $r$ has order 2 and the kernel $K$ of $B \to \mathrm{Aut}(A)$ has order $3$. 
So~$\mu(G)$ has order $105$, and all subgroups of odd order are cyclic and normal.
Here $G$ contains $D_{35}$ as a subgroup.
\item
Otherwise (for $r=4, 9, 19, 24$) we have $r$ of order 6 and  
 the kernel $K$ of the map $B \to \mathrm{Aut}(A)$ has order $1$. So $\mu(G) = A$.
 Here $G$ also contains $D_{35}$ as a subgroup. Note that $r = 4$ and $r = 9$ are inverses
 and so yield isomorphic semidirect products. Similarly, $19$ and $24$ give similar results.
\end{itemize}

Finally if $A$ has order $105$, then $B$ has order $2$.  In order to satisfy the requirements
that $r$ have order at most $2$ and that $r-1$ be relatively prime to $3, 5, 7$, we must have $r \equiv -1$ modulo
$3, 5$ and $7$. So $r=-1$ modulo $105$. In this case~$\mu(G) = A$ has 105 elements and $G$ is $D_{105}$.

\medskip\noindent
All in all we have 12 groups of order 210 up to isomorphism. Later we will see that none of them are freely representable
except the cyclic group.
In fact, only one non-Abelian example (where $A$ has order 7) satisfies the necessary condition that~$\mu(G)$ has even order and so there is a unique element of order two (Corollary~\ref{new_corollary2}).
\end{example}

\section{Freely Representable Sylow-Cyclic Groups}

Now we take-up the question of which Sylow-cyclic groups are freely representable.

\begin{lemma} \label{new_lemma}
Suppose $G$ is the semi-direct product $A \rtimes B$ of two cyclic groups where~$B$ has prime order $p$
and where $A$ has order $q^k$ for a prime $q$ not equal to $p$.
Then if $G$ is freely representable, $G$ must be cyclic.
\end{lemma}

\begin{proof}
By
Corollary~\ref{semi-direct_cor},
$G$ is Sylow-cyclic. So by Theorem~\ref{semidirect_thm} the commutator subgroup $G'$ of $G$ is cyclic,
and $G'$ and $G/G'$ have relatively prime orders. Since~$G/A$ is isomorphic to $B$ and so is Abelian, $G' \subseteq A$.
So either $G' = 1$ or $G' = A$.

Start by assuming $G' = A$.
By Theorem~\ref{semidirect_thm}, any element of $g\in G$ whose image in $G/A$ generates $G/A$
must generate a group $\left< g \right>$ of order $p$. Since every nontrivial element of $G/A$ generates this group, every element $g \in G - A$
is contained in a unique subgroup of $G$ of order $p$.

 Let $\mathcal C$
be the collection of subgroups of $G$ consisting of $A$ together with all subgroups of order $p$.
Then every nonidentity element of $G$ is in exactly
one $C \in \mathcal C$. So
$$
\sum_{C \in \mathcal C} \N C = (k-1) \mathbf{1} + \N G
$$
where $k$ is the size of $\mathcal C$.
This contradicts Theorem~\ref{biasse2020norm_thm2} since $k>1$.

So we conclude that $A$ is not $G'$. This means that $G' = 1$. So $G$ is Abelian.
Since $G$  is Sylow-cyclic and Abelian, $G$ must be cyclic.
\end{proof}

\begin{theorem} \label{new_thm}
Suppose $G$ is the semidirect product $A \rtimes B$ of two cyclic groups where $B$ has prime order $p$
and where $p$ does not divide the order of $A$. 
Then $G$ is freely representable if and only if $G$ is cyclic.
\end{theorem}

\begin{proof}
One implication is clear, so suppose $G$ is freely representable. Since $A$ is normal and cyclic in $G$,
all the subgroups of $A$ are normal and cyclic in $G$ (since there is at most one subgroup of $A$
of any given order). So if $q$ is a prime dividing 
the order of $A$ then the $q$-Sylow subgroup $A_q$ of $A$ will be a normal $q$-Sylow subgroup of $G$.
So viewing $B$ as a subgroup of $G$, we have that $H_q = A_q B = A_q \rtimes B$ is a subgroup of~$G$.
Since $H_q$ is a subgroup of a freely representable group, $H_q$ is a freely representable group.
By the above lemma~$H_q$ is cyclic. 

In particular $B$ acts trivially by conjugation on $A_q$ for each Sylow subgroup of~$A$.
Since $A$ is cyclic, it is generated by its Sylow subgroups. Thus $B$ acts  trivially on all of $A$. 
This implies that $G$ is Abelian. Since $G$ is Sylow-cyclic this means that~$G$ is cyclic.
\end{proof}

We can leverage this to get a
 necessary condition for a Sylow-cyclic group $G$ to be freely representable.
  Basically the condition is that $G$ is a semi-direct product of cyclic groups using a relatively
 weak action:

\begin{lemma}\label{necessary_lemma}
Let $A$ and $B$ be cyclic subgroups 
of relatively prime orders of a group~$G$ such that~$G = AB = A \rtimes B$.
If $G$ is freely representable then the kernel $K$ of the associated action homomorphism
$$
B \to \mathrm{Aut} (A)
$$
must contain all subgroups of $B$ of prime order.
\end{lemma}

\begin{proof}
Suppose $B_p$ is a subgroup of $B$ of prime order $p$.
Observe that $A B_p$ is a subgroup of $G$ of order $|A| p$ (since $A$ is normal in $G$) and is the semi-direct product~$A \rtimes B_p$
where the action of $B_p$ on $A$ is just the restriction of the action of~$B$ on $A$.

Since $G$ is freely representable, $A B_p$ is as well. So by the above theorem, $A B_p$ is cyclic, hence Abelian.
Thus $B_p$ acts trivially on $A$. In other words, $B_p$ is in the kernel $K$ of the action of $B$ on $A$.
\end{proof}

\begin{corollary} \label{new_corollary}
Suppose $G$ is the semi-direct product $A \rtimes B$ of two cyclic groups of relatively prime order.
Suppose $B$ has square free order. 
Then $G$ is freely representable if and only if $G$ is cyclic.

In particular, if $G$ is a group of square free order then $G$ is freely representable
if and only if $G$ is cyclic.
\end{corollary}

\begin{proof}
One direction is clear, so we assume $G$ is freely representable. 
So by the above theorem, the order of the kernel $K$ of the action homomorphism $B \to \mathrm{Aut} (A)$ must
be divisible by exactly the primes that divide the order of $B$. This means that $K = B$, so $G = AB = A \times B$  is cyclic.
\end{proof}

We can restate the above lemma in terms $\mu(G)$:

\begin{corollary} \label{necessary_cor}
Let $G$ be a freely representable Sylow-cyclic group, and let $\mu(G)$ be the MCC subgroup of $G$.
Then every element of $G$ of prime order is  in~$\mu(G)$. In particular,
every prime dividing the order of $G$ must divide the order of~$\mu(G)$.
\end{corollary}

\begin{proof}
Write $G$ as $AB = A \rtimes B$ where $A$ and $B$ are cyclic subgroups of $G$ of relatively prime orders. 
Let $g = ab$ be an element of prime order where $a\in A$ and~$b\in B$. Under the projection
$A \rtimes B \to B$ the element $g$ maps to $b$, so $b$ is of prime order or is the trivial element.
In either case, by the above lemma, $b \in K$ where $K$ is the kernel of the action map~$B \to \mathrm{Aut}(A)$.
So $g  = a b \in AK$. But~$AK = \mu(G)$ (Lemma~\ref{mck_lemma})
so $g \in \mu(G)$. 
\end{proof}

This gives a necessary condition for a Sylow-cyclic group to be freely representable. In order to
show it is sufficient we use induced representations, but in a very basic manner.
The following, which we take as given, is all that we need to know about induced representations here:

\begin{proposition}
Let $G$ be a finite group with subgroup $H$, and let $F$ be a field.
Suppose that $H$ acts linearly on an~$F$-vector space $W$.
Then there is a  linear action of $G$ on an $F$-vector space $V$ containing $W$ such that (1) the action of $G$ on $V$ restricts to the
given action of $H$ on $W$, (2) if~$g_1 H, \ldots, g_k H$ are the distinct left cosets in~$G/H$ then
the vector space $W$ is the direct sum of the spaces $g_i W$:
$$
V = \bigoplus g_i W.
$$
This representation, called the induced representation, is
unique up to a $F[G]$-module  isomorphism fixing $W$.
\end{proposition}

\begin{proposition} \label{sufficient_prop}
Let $G$ be a finite group with a subgroup~$H$ that contains all elements of $G$ of prime order.
Suppose $W$ is an $F$-vector space with a linear representation of $H$ on $W$, and suppose $V$
is a $F$-vector space containing $W$ with a representation of~$G$ induced by the representation of $H$ on $W$.
If the linear representation of $H$ on $W$ is a free linear representation, then the linear representation
of $G$ on $V$ is also a free linear representation.
\end{proposition}

\begin{proof}
Let $g_1 H, \ldots, g_k H$ be the distinct left cosets in~$G/H$.
Suppose $\sigma \in G$ is not the identity and that $\sigma (v) = v$ where $v\in V$ is equal to 
$$
v = g_1 w_1 + \ldots + g_k w_k
$$
where $w_i \in W$.
Let $m>1$ be the order of $\sigma$ and let $p$ be a prime dividing $m$.
Then~$\tau = \sigma^{m/p}$ has order $p$ and also fixes $v$.
By assumption $\tau \in H$. Observe that the conjugate $\tau_i = g_i^{-1} \tau g_i$ also has order $p$
so is in $H$. Since $\tau g_i = g_i \tau_i$
$$
\tau v = \tau g_1 w_1 + \ldots + \tau g_k w_k = g_1 (\tau_1 w_1) + \ldots + g_k(\tau_k w_k).
$$
Since $\tau v = v$, 
$$
g_1 (\tau_1 w_1) + \ldots + g_k(\tau_k w_k) = g_1 w_1 + \ldots + g_k w_k.
$$
By the direct sum property of induced representations,
$$
g_i(\tau_i w_i) = g_i w_i
$$
for each $1 \le i \le k$. Multiplying by the inverse of $g_i$ gives $\tau_i w_i = w_i$
for the induced representation.
Since $\tau_i \in H$ and $w_i \in W$, we have $\tau_i w_i = w_i$ in the original representation of $H$ on $W$.
This representation is a free linear representation by assumption, so $w_i = 0$ for each $i$. This implies
$v = 0$, showing that the induced representation is a free linear representation.
\end{proof}

\begin{corollary}\label{sufficient_cor}
Let $G$ be a finite group with freely representable  subgroup $H$.
If $H$ contains all elements of $G$ of prime order, then $G$ is also freely
representable.
\end{corollary}

Now we are ready for the main theorem:

\begin{theorem}\label{sc_fr_thm}
Let $G$ be a Sylow-cyclic group and let $\mu(G)$ be its maximal characteristic cyclic subgroup.
The following are equivalent:
\begin{enumerate}
\item 
$G$ is freely representable.
\item
Every prime dividing the order of $G$ also divides the order of $\mu(G)$.
\item
For every prime $p$ dividing the order of $G$ there is a unique subgroup of $G$ of~order~$p$.
\end{enumerate}
\end{theorem}

\begin{proof}
Suppose $G$ is freely representable. 
Then by Corollary~\ref{necessary_cor} every prime dividing the order of $G$ must
divide the order of $\mu(G)$. So $(1) \implies (2)$.

Now assume $(2)$. If $p$ divides $|G|$ then, by assumption, $p$ divides $|\mu(G)|$.
Thus there is a unique subgroup of $G$ of order $p$ by Corollary~\ref{unique_cyclic_cor}.
So $(2) \implies (3)$ holds.

Finally assume $(3)$. Let $g \in G$ be an element of prime $p$ order.
By assumption
the subgroup~$\left< g \right>$ is the unique subgroup of $G$ of order $p$. This implies that $\left< g \right>$ is 
a characteristic cyclic subgroup, and so $\left< g \right>$ is contained in $\mu(G)$.
Since $\mu(G)$ is cyclic, it is freely representable. Thus $G$ is freely representable by Corollary~\ref{sufficient_cor}.
\end{proof}

We can now strengthen Lemma~\ref{necessary_lemma}.

\begin{corollary}\label{ns_cor}
Let $A$ and $B$ be cyclic subgroups 
of $G$ of relatively prime orders such that~$G = A \rtimes B$ (so $A$ is normal in $G$).
Let $K$ be the kernel of the action homomorphism  $B \to \Aut (A)$ associated to the semidirect product.
Then $G$ is freely representable if and only if every prime dividing $|B|$  divides $|K|$.
\end{corollary}

\begin{proof}
Recall that $\mu(G)$ is $AK$ (Lemma~\ref{mck_lemma}).

If every prime dividing $|B|$ divides $|K|$ then every prime dividing the order of~$G$ must
divide $|A|$ or $|K|$. Hence every prime dividing the order of $G$ divides 
the order of $\mu(G)$, and so $G$ is freely representable by the above theorem.

Conversely, if $G$ is freely representable, then every prime $p$ dividing the order of $B$
must divide the order of $\mu(G) = AK \cong A \times K$ by the above theorem. But $p$ does not divide the order of $A$,
so $p$ divides the order of $K$ as desired.
\end{proof}

\begin{example}
Suppose $A$ a cyclic group of odd prime order $p$. Let $q$ be any prime dividing $p-1$.
Let $B$ be a cyclic group of order $q^k$ with $k>1$. We can identify~$\mathrm{Aut}(A)$ with $\bF_p^\times$
which is cyclic of order $p-1$. Let $B'$ be the unique quotient of $B$ of size $q$, and fix an injective homomorphism
$B' \to \mathrm{Aut}(A)$. Now have $B$ act on $A$ by the composition
$$
B \to B' \to \mathrm{Aut}(A)
$$
and let $G$ be the associated semidirect product $A \rtimes B$.
The kernel $K$ of this action homomorphism has order $q^{k-1}$.
By Lemma~\ref{mck_lemma} we have that $\mu(G) \cong A \times K$, and this has $p q^{k-1}$ elements.
By the above theorem, $G$ is a freely representable non-Abelian group of order~$p q^k$.
In the case of order $p \cdot 2^2$ such groups arose already as binary dihedral groups $2 D_p$ in $\bH^\times$.
But if we take $q \ne 2$ we can conclude the following:
\emph{There are an infinite number of odd orders such that there exists non-Abelian Sylow-cyclic groups
of that order that are freely representable} (for example, 
for a fixed $q$ take an infinite sequence of primes $p \equiv 1 \bmod q$). Note that such Sylow-cyclic groups cannot be isomorphic to subgroups of $\bH^\times$
since all finite non-Abelian subgroups of~$\bH^\times$ have order divisible by $4$.

The smallest such order of this type of group of odd order is $7 \cdot 3^2 = 63$. Note that if $G$ 
is a noncyclic freely representable group of order $5 \cdot 3^2$ then $G$ is $A \rtimes B$ where $A$ is cyclic 
subgroup of $G$ of order $5$ and $B$ is a cyclic subgroup of $G$ of order~$9$.
Furthermore,
$\mu(G)$ has order $15$ or~$45$.
The second case cannot happen since $G$ is not cyclic.
The first case cannot happen either:
 the automorphism group of $A$ has order 4,
so the kernel $K$ of the action of $B$ on $A$ must be all of $K$, so~$\mu(G) = A K = G$.
We conclude that 63 is the smallest odd order possible for a noncyclic freely representable group.
\end{example}

Here is another interesting application of Theorem~\ref{sc_fr_thm}.

\begin{proposition}\label{product_prop}
Let $G$ be a freely representable Sylow-cyclic group and let $N$ be a normal
subgroup of $G$ of index $p$. If $p^2$ does not divide the order of $G$ then
$$G = NC_p \cong N \times C_p$$
where $C_p$ is a subgroup of $G$ of order $p$ (and is the unique subgroup of order $p$).
\end{proposition}

\begin{proof}
Since $G$ is freely representable, there is a unique subgroup $C_p$ of order $p$, and so $C_p$
must be normal in $G$. Since $N \cap C_p = \{1\}$ we  have $N C_p \cong N \times C_p$.
Finally,  $G = N C_p$ since $[G: N] = p$.
\end{proof}

We cannot hope to generalize Theorem~\ref{sc_fr_thm} to all freely representable groups.
For example, the binary tetrahedral group $2 T$ is freely representable (as a subgroup of~$\bH^\times$),
but does not have a unique subgroup of order 3. However, one implication holds in general:

\begin{proposition}\label{unique_freely_rep_prop}
Let $G$ be a finite group with the property that for each prime $p$ dividing the order of $G$ there
is a unique subgroup of order $p$.
Then $G$ is freely representable.
\end{proposition}

\begin{proof}
Let $p_1, \ldots, p_k$ be the primes dividing the order of $G$. Let $C_{p_i}$ be the 
 be the unique subgroup of order $p_i$.
Then each $C_{p_i}$ is normal and
$$
H \; \defeq \; C_{p_1} \cdots C_{p_k} \cong C_{p_1} \times \cdots \times C_{p_k} 
$$
is a cyclic subgroup of $G$. So $H$ is freely representable. Now use Corollary~\ref{sufficient_cor}.
\end{proof}

\begin{remark}
The classification of Sylow-cyclic groups is enough to yield significant applications
to differential geometry. In fact, by a theorem of Vincent (1947), every complete connected Riemannian manifold of 
constant positive curvature of dimension not congruent to $3$ modulo $4$ has a fundamental group 
that is Sylow-cyclic. From this Vincent was able to give a full classification of such manifolds when the dimension
is not congruent to~$3$ modulo $4$. Wolf \cite{wolf2011} completed the classification to all dimensions by classifying
freely representable groups beyond the Sylow-cyclic groups.
\end{remark}


\section{Application to Automorphisms of Sylow-cyclic groups}

Consider the automorphism group $\Aut(G)$ where $G$ is a
Sylow-cyclic group of odd order, and let $O(\Aut(G))$ be the maximal odd normal subgroup of $\Aut(G)$.
Then we can use the above results to show that $\Aut(G)/ O(\Aut(G))$ is an Abelian $2$-group.
This is clear if $G$ is cyclic since  $\Aut(G)$ is an Abelian group.
The key  to generalizing this is to relate this quotient to the corresponding quotient for 
the cyclic subgroup $\mu(G)$. We start with a lemma.

\begin{lemma}
Let $G$ be a Sylow-cyclic group of odd order and let
 $\phi$ be an automorphism of $G$ such that $\phi^2$ is the identity map.
If $\phi$ fixes the MCC subgroup $\mu(G)$  then $\phi$ fixes all of $G$.
\end{lemma}

\begin{proof}
Let $P$ be a nontrivial Sylow subgroup of $G$. We will show that $\phi$ acts trivially on~$P$.
Since the Sylow subgroups of $G$ generate $G$, this gives the result.
Let $A$ be the commutator subgroup of $G$.
Since $A$ and $G/A$ have relatively prime orders, 
either~$P \subseteq A$ or $P \cap A = \{ 1 \}$. In the first case $\phi$ acts trivially on $P$
since $A\subseteq \mu(G)$.
So from now on we assume that $P \cap A = \{1\}$.

Observe that the image $\overline P$ of $P$ in $G/A$ is isomorphic to $P$.
Since $\overline P$ is a cyclic group of odd order, it has a unique automorphism of order $2$.
So $\phi$ acts on $\overline P$ either as~$x \mapsto x$ or as $x \mapsto x^{-1}$.

First suppose that  $\phi$ acts on $\overline P$ as~$x \mapsto x$. So if $b \in P$
then $\phi(b) = a b$ for some~$a \in A$. Observe then that 
$$
b = \phi^2(b) = \phi(ab) = \phi(a) \phi(b) = a(ab) = a^2 b.
$$ 
Thus $a^2 = 1$. Since $A$ is a cyclic group of odd order $a = 1$, and $\phi(b) = b$.
We conclude that $\phi$ acts trivially on~$P$.

Finally suppose that  $\phi$ acts on $\overline P$ as~$x \mapsto x^{-1}$ and let $c \in P$ be a generator.
Thus~$\phi(c) = c^{-1} a_0$ for some $a_0 \in A$. Let $B$ be a complement of $A$ and observe that~$P$
is conjugate to a Sylow subgroup of $B$. So replacing $B$ with a complement of $B$ if necessary, 
we can assume $P$ is a subgroup of $B$. Note that $\phi(c) \ne c$ since~$c^{-1}$ and~$c$ have 
distinct images in $\overline P$ (and $\overline P$ has odd order greater than 1). Thus $c$ is not in the center $Z(G)$ since
$Z(G) \subseteq \mu(G)$.
Since $c$ centralizes $B$,  it cannot centralize~$A$. Let $a \in A$ be such that $c a c^{-1} \ne a$.
Note that $c a c^{-1} \in A$ so
$$
c a c^{-1} = \phi(c a c^{-1} ) = \phi(c) a \phi(c)^{-1} = c^{-1} a_0 a a_0^{-1} c = c^{-1} a c.
$$
So $c^2 a c^{-2} = a$. However $c$ has odd order, so this implies that $c a c^{-1} = a$, a contradiction.
So $\phi$ cannot act on $\overline P$ as~$x \mapsto x^{-1}$.
\end{proof}

\begin{proposition} \label{aut_quotient_prop}
Let $G$ be a Sylow-cyclic group of odd order. Let
$\mu(G)$ be the~MCC subgroup of $G$, and let
 $O(\Aut(G))$ be the maximal odd normal subgroup of the automorphism group $\Aut(G)$.
Then the quotient
$$
\Aut(G) / O(\Aut(G)).
$$
is isomorphic to a $2$-group inside $\Aut(\mu(G))$. In particular, this quotient is an Abelian $2$-group.
\end{proposition}

\begin{proof}
Observe that $\Aut(\mu(G))$ is Abelian since $\mu(G)$ is a cyclic group. Recall that Abelian groups
are the products of their Sylow subgroups, and so 
$$
\Aut(\mu(G)) = A_1 A_2 = A_1 \times A_2
$$
where $A_1$ is the subgroup of $\Aut(\mu(G))$ consisting of elements of odd order, and where~$A_2$ is the $2$-Sylow subgroup of $\Aut(\mu(G))$.
Since $\mu(G)$ is characteristic in~$G$ we have a homomorphism $\Aut(G) \to \Aut(\mu(G))$.
We also have the projection homomorphism $\Aut(\mu(G)) = A_1 A_2  \to A_2$.
Let $K$ be the kernel of the composition
$$
\Aut(G) \to \Aut(\mu(G)) \to  A_2.
$$
Observe that $K$ contains $O(\Aut(G))$.

Claim: $K$ contains only elements of odd order. Suppose otherwise that $\psi \in K$ has order $2k$,
and let $\phi = \psi^k$. Then $\phi$ has order $2$ and is in $K$.
The image of $\phi$ in~$\Aut(\mu(G))$ is just the restriction $\phi |_{\mu(G)}$ and 
it is in the kernel of of the projection~$\Aut(\mu(G)) = A_1 A_2  \to A_2$. In other words, $\phi |_{\mu(G)} \in A_1$
and so has odd order. Since $\phi$ has order $2$, we conclude that $\phi |_{\mu(G)}$ has order 1.
By the previous lemma $\phi$ is the identity, a contradiction.

So the claim has been established. This means $K = O(\Aut(G))$ and we have an injection
$$
\Aut(G)  / O(\Aut(G)) \hookrightarrow  A_2.
$$
The result follows.
\end{proof}

\chapter{Applications to Division Rings}\label{Wedderburn}

The classification of Sylow-cyclic fields can be used to prove Wedderburn's theorem.
Along the way we will see an argument that every finite subgroup of a field is cyclic. 
This section is independent of Example~\ref{divisionring_example} where we used (1)  Wedderburn's theorem
and (2) the fact that $F^\times$ is cyclic for any finite field. 
We start with the following.

\begin{lemma} \label{Wedderburn_lemmaA}
Let $D$ be a division ring and let $F$ be its prime subfield. Then every finite subgroup $G$ of $D^\times$ is freely representable over $F$.
\end{lemma}

\begin{proof}
We let $G$ act on $V=D$ by left multiplication. Note that $V$ is an $F$-vector space. This action
is a free linear action since $D$ has no zero divisors.
\end{proof}

\begin{remark}
In particular, if $F$ has characteristic zero then $G$ is freely representable (and so is Sylow-cycloidal).
This  result is the starting point for Amitsur's classification of finite subgroups of $D^\times$ where $D$ is a division ring
(1955~\cite{amitsur1955}). Amitsur used  class field theory to complete the classification.
\end{remark}

\begin{lemma}\label{pp_lemma}
Let $G$ be a group of order $p^2$ where $p$ is a prime. Suppose $F$ is a field of characteristic not equal to $p$.
If $G$ is freely representable over $F$ then $G$ is cyclic.
\end{lemma}

\begin{proof}
Suppose $G$ is freely representable but 
 not cyclic. This means that every nonidentity element of $G$ is in a unique cyclic group of order $p$.
 Observe that there are $k = (p^2 - 1)/(p-1) = p+1$ such cyclic groups.
 Let $\mathcal C$
be the collection of  cyclic subgroups of~$G$ of order $p$. Then 
$$
\sum_{C \in \mathcal C} \N C = (k-1) \mathbf{1} + \N G = p \mathbf{1} + \N G.
$$
Since $p$ is nonzero in $F$, 
this contradicts Theorem~\ref{biasse2020norm_thm}.
\end{proof}

Suppose $D$ is a division ring and that $G$ is a finite subgroup of $D^\times$.  If $F$ is the prime subfield of $D$, then let $F(G)$
be the $F$-span of $G$ in $D$. (Warning:  $F(G)$ is analogous to the group ring $F[G]$, but they are not the same since $G$ might
not be linearly independent.)

\begin{lemma}\label{wedderburn_FG_lemma}
Let $D, G, F, F(G)$ be as above, and
suppose $F$ is $\bF_p$ for some prime~$p$.
Then $F(G)$  is  a finite division ring of order a power of $p$.
\end{lemma}

\begin{proof}
Let $k$ be the size of a basis of $F(G)$ for scalar field $F$, and observe that~$F(G)$ has finite size~$p^k$.
Observe that $F(G)$ is closed under multiplication, so we  conclude that~$F(G)$ is a subring of $D$.
Next suppose $a \in F(G)$ is nonzero. Then the map~$x \mapsto a x$ is an injective map $F(G) \to F(G)$
since $F(G)$ is contained in a division ring.  Since~$F(G)$ is finite,
this map is surjective, and $a b = 1$ for some~$b \in F(G)$.~This implies that~$F(G)$ is a division ring.
\end{proof}

\begin{lemma}
Let $D$ be a division ring whose prime field $F$ has prime characteristic~$p$. Then~$D^\times$ has no elements of order $p$.
\end{lemma}

\begin{proof}
Suppose $g \in D^\times$ has order $p$, and let $G$ be the group generated by $g$.
Observe that $G$ is a subgroup of~$F(G)^\times$,
and $F(G)^\times$ has order $p^k - 1$ for some $k\ge 1$. So the order of $g \in F(G)^\times$ fails
to divide the order of $F(G)^\times$, a contradiction.
\end{proof}

\begin{corollary}\label{pp_cor}
Let $G$ be a subgroup of $D^\times$ where $D$ is a division ring.
Then every subgroup of $G$ of order $p^2$, where $p$ is a prime, is cyclic.
\end{corollary}

\begin{proof}
Let $H$ be a subgroup of $G$ of order $p^2$ where $p$ is a prime.
By the above lemma, we can assume that $p$ is not the characteristic of the prime field $F$ of $D$.
By Lemma~\ref{Wedderburn_lemmaA}, $H$ is freely representable over $F$.
So $H$ is cyclic by Lemma~\ref{pp_lemma}.
\end{proof}

\begin{corollary} \label{field_cyclic_cor}
Suppose $G$ is a finite subgroup of $F^\times$ where $F$ is a field.
Then $G$ is cyclic.
\end{corollary}

\begin{proof}
By the previous corollary, every subgroup of $G$ of order $p^2$ is cyclic, for any prime $p$.
Since $G$ is Abelian, $G$ is cyclic by the finite structure theorem of Abelian groups (or you can
use the more elementary argument given in the remark after Corollary~\ref{pq_cor}).
\end{proof}

\begin{proposition}
Let $D$ be a division algebra and let $G$ be a finite subgroup of~$D^\times$.
Then $G$  is a Sylow-Cycloidal group
\end{proposition}

\begin{proof}
If $p$ is odd, then any $q$-Sylow subgroup of $G$ is cyclic by Theorem~\ref{pgroup_thm}.
If~$p = 2$ then any $q$-Sylow subgroup is either cyclic or a generalized quaternion group by Corollary~\ref{unique_2_cor}.

Now suppose $F$ has prime characteristic $p$. In this case $G$ is a subgroup of~$F(G)^\times$,
and $F(G)^\times$ has order $p^k - 1$ for some $k\ge 1$. So $p$ cannot divide the order of $G$, so there are no $p$-Sylow subgroups
of $G$ that we need to worry about.

We conclude that all Sylow-subgroups of $G$ have the desired form, and that
$G$ is a Sylow-cycloidal group.
\end{proof}

\begin{lemma}
Let $D$ be a division algebra of prime characteristic $p$,
and let $G$ be a finite subgroup of~$D^\times$.
Then $G$  is a Sylow-cyclic group
\end{lemma}

\begin{proof}
Let $q$ be any odd prime dividing the order of $G$.
By Corollary~\ref{pp_cor} and Theorem~\ref{pgroup_thm}, the $q$-Sylow subgroups of $G$ are cyclic.

Suppose the 2-Sylow subgroup of $G$ are not also cyclic. 
By Corollary~\ref{pp_cor} and Proposition~\ref{unique_2_cor} the $2$-Sylow subgroups of $G$ must be generalized
quaternion groups.
By Proposition~\ref{2_quaternion_subgroup_prop}, $G$ must then contain a subgroup $Q$  that we can identify with
the quaternion group of size~8.
By Lemma~\ref{wedderburn_FG_lemma} we have the division algebra $\bF_p(Q)$.
The idea of the remainder of the proof is to argue that there cannot be a ``quaternion ring'' over~$\bF_p$.

To proceed we solve $1+ x^2 + y^2 = 0$ over $\bF_p$. If $p=2$ then $x = 1, y=0$ is a solution.
Otherwise, observe that there are $(p+1)/2$ squares in $\bF_p$. So as $x$ varies in $\bF_p$, the expression $-1 - x^2$ takes on~$(p+1)/2$
distinct values.
At least one of these values must be a square since there are only $(p-1)/2$ nonsquares in $\bF_p$. So we choose $x$ so that $-1-x^2$
is a square, and we choose $y$ so that $y^2$ is $-1-x^2$. We can exchange $x$ and $y$ if necessary, and assume $y \ne 0$.
Then in $\bF_p(Q)$ we have 
$$
(1 + x \mathbf{i} + y \mathbf{j}) (1 - x \mathbf{i} - y \mathbf{j}) = 1 + x^2 + y^2 = 0.
$$
Since $\bF_p(Q)$ has no zero divisors, we have $1 + x \mathbf{i} + y \mathbf{j} = 0$ or $1 - x \mathbf{i} - y \mathbf{j} = 0$.
Since~$y \ne 0$ this means that~$\mathbf{j} \in \bF_{p} (\left< \mathbf{ i} \right>)$. But $ \bF_{p} (\left< \mathbf{ i} \right>)$
is a field, and so $\mathbf{i}$ and $\mathbf{j}$ commute, a contradiction.
\end{proof}

The following is a result of Herstein. He proved it as a corollary of Wedderburn's theorem, but we will prove Wedderburn's theorem as a corollary
of this result.

\begin{theorem} [Herstein]
Let $D$ be a division ring of prime characteristic $p$,
and let $G$ be a finite subgroup of~$D^\times$.
Then $G$  is cyclic.
\end{theorem}

\begin{proof}
First we consider the case where $D$ is finite and $G = D^\times$.
By the above lemma $G$ is a Sylow-cyclic group. 
Let $C$ be the maximum cyclic characteristic (MCC) subgroup of $G$.
Such a subgroup $C$ exists by Corollary~\ref{subgroups_cg_corollary} and,
by Corollary~\ref{MCC_maximal_cor}, $C$ is maximal among cyclic subgroups of $G$.
Observe that $\bF_p(C)$ is a field,
so $\bF_p(C)^\times$ is cyclic. By the maximality of $C$, this means that $C = \bF_p(C)^\times$.

Let $q$ be the number of elements of $\bF_p(C)$, and let $q^k$ be the number of elements of $D$.
(Here we use the fact that $D$ is a vector space over any subfield). If $k> 1$ then
$$
|G/C| = \frac{q^k - 1}{q - 1} = q^{k-1} + \ldots + q + 1 \ge q + 1 > q - 1 = |C|
$$
which contradicts Corollary~\ref{mck_size_cor}. Thus $k=1$ and so $G = C$, and $G$ is cyclic.

In general, we consider $\bF_p(G)$. Since $\bF_p(G)^\times$ is cyclic, as we have just shown, and since $G$ is a subgroup
of $\bF_p(G)^\times$, we conclude that $G$ is cyclic as well.
\end{proof}

\begin{corollary} [Wedderburn]
Every finite division  ring $D$ is a field.
\end{corollary}

\begin{proof}
By the above theorem $D^\times$ is cyclic. So $D$ must be a commutative ring.
\end{proof}


\chapter{Sylow-Cycloidal Groups: The Solvable Case} \label{scq1_section}

Suppose $G$ is a solvable Sylow-cycloidal group and  $O(G)$ is the maximal normal subgroup of $G$ of odd order.
Then our first important result will be to describe the possible quotients $G/O(G)$. We will show that $G/O(G)$ is
isomorphic to either a cyclic~$2$-group, a generalized quaternion group, the
binary tetrahedral group $2 T$ or the binary octahedral group~$2 O$. 
In particular, $G/O(G)$ is isomorphic to a solvable subgroup of $\bH^\times$.
This result divides  Sylow-Cycloidal groups into four mutually exclusive types. We then focus on each type individually.\footnote{I learned the technique of classifying freely representable groups $G$ by their quotients~$G/O(G)$
from a recent paper by Daniel Allcock~\cite{allcock2018}. My approach generalizes the scope of Allcock a bit.  
Although my proof is different and more elementary that that in~\cite{wolf2011}, I also lean on Wolf~\cite{wolf2011}, and thus indirectly on Zassenhaus (1936), for guidance.
Zassenhaus adopts a more general scope than mine by  only restricting the $2$-Sylow subgroup (see Lemma 6.1.9 of Wolf~\cite{wolf2011} attributed to Zassenhaus). }

\section{The Quotient $G/O(G)$}

We start more generally than  with Sylow-cycloidal groups. We say a finite group~$G$ \emph{satisfies the $(2, 3)$ condition} if every $2$-Sylow subgroup
of $G$ is cyclic or is the quotient of a generalized quaternion group\footnote{This means that the $2$-Sylow subgroups of $G$ are cyclic, dihedral, or generalized quaternion, but we will not need this fact.} and every $3$-Sylow subgroup is cyclic.

\begin{lemma}
Let $G$ be a group that satisfies the $(2, 3)$ condition. Then every subgroup and quotient of $G$ satisifies the $(2,3)$ condition.
\end{lemma}

\begin{proof}
Observe that every $p$-Sylow subgroup of a quotient $G/N$ is a quotient of a~$p$-Sylow subgroup of $G$. Also every~$p$-Sylow subgroup
of a subgroup $H$ of $G$ is a~$p$-group and so is a subgroup of a $p$-Sylow subgroup of $G$.

The class of cyclic $p$-groups is closed under the processes of quotient and subgroup.
Since the class of cyclic and generalized quaternion $2$-groups is closed under subgroup,
the class of quotients of such groups is closed under quotient and subgroups.
\end{proof}

\begin{definition}
Let $G$ be a finite solvable group and
let $G, G', G'', \ldots, G^{(k)} = \{1\}$ be the derived series of commutator subgroups.
Then the \emph{characteristic Abelian subgroup of $G$},
which we denote as~$\mathcal A (G)$,
is defined to be the first Abelian term of the series.  Observe that~$\mathcal A (G)$ is Abelian and characteristic, and if $G$ is nontrivial
then $\mathcal A (G)$ is also nontrivial.
\end{definition}

\begin{lemma}
Let $G$ be a  solvable group satisfying the $(2, 3)$ condition. Then either~$G$ has order of the form $2^m 3^n$ 
or there is a prime $p \ge 5$ such that there 
is a
 $p$-subgroup~$K$ of $G$  that is a characteristic subgroup of $G$.
\end{lemma}

\begin{proof}
Let $q$ be a prime dividing $\mathcal A (G)$ (if no such $q$ exists then $G$ has order $2^0 3^0$ and we are done).
If $q\ge 5$  then we can choose $p=q$ and choose $K$ to be the~$p$-Sylow subgroup of~$\mathcal A (G)$, and we are done.
Otherwise we define an elementary characteristic subgroup $N$ as the solutions of $x^q = 1$ in~$\mathcal A (G)$.

Since $G$ satisfies the $(2, 3)$ condition, the same is true of $N$ and $G/N$. If $q=3$ this means that $N$ is cyclic of order $3$.
If $q=2$ then $N$ is generated by one or two elements, so is either cyclic of order $2$ or is the Klein four group. So any automorphism
of $N$ has order $1, 2,$ or $3$. 

If $G/N$ has order of the form $2^m 3^n$ we are done, so we will now
assume that~$G/N$ is not of that form.
By induction we can assume that $G/N$ has a $p$-subgroup $L/N$ where $p\ge 5$ is prime
and where $L$ is a subgroup of $G$ containing $N$ such that $L/N$ is a characteristic subgroup of $G/N$.
Since $N$ is a characteristic subgroup, this implies that $L$ must also be a characteristic subgroup of $G$.
Let $K$ be a $p$-Sylow subgroup of $L$ and observe that $K$ is a complement of $N$ in $L$, so $L \cong N \rtimes K$.
Note that $K$ acts trivially on $N$ since all automorphism
of~$N$ have order prime to $p$. This means that  $L$ is isomorphic to $N\times K$,
and so $K$ is the unique $p$-Sylow subgroup of $L$. Since $L$ is a characteristic subgroup of $G$, its
unique $p$-Sylow subgroup $K$ is also characteristic subgroup of $G$. 
\end{proof}

\begin{corollary}
Let $G$ be a  solvable group satisfying the $(2, 3)$ condition. Let $O_6(G)$ be the maximal normal subgroup of $G$ of order relatively prime to $6$.
Then $G/O_6(G)$ has order of the form $2^m 3^n$.
\end{corollary}

\begin{corollary}\label{corollary96}
Let $G$ be a  solvable group satisfying the $(2, 3)$ condition. Let $O(G)$ be the maximal normal subgroup of $G$ of odd order.
Then $G/O(G)$ has order of the form $2^m 3^n$, and every nontrivial normal subgroup of $G/O(G)$ has even order.
\end{corollary}

Motivated by the above corollary, we focus on Sylow-Cycloidal groups whose order is of the form $2^m 3^n$.
Of course we know what such group are when  $m=0$ or when~$n=0$,
so we focus on the case where $m$ and $n$ are positive.

\begin{lemma} \label{lemma97}
Let $G$ be a solvable Sylow-cycloidal group of order $2^m 3^n$ with $m$ and~$n$ positive. 
Assume also that every normal subgroup of $G$
is of even order.
Then the following hold
\begin{itemize}
\item
$G$ contains a normal subgroup $N$ isomorphic to the quaternion group $Q_8$.
\item
Every element of $G$ of order $3$ acts nontrivially on $N$.
\item
$G$ contains four subgroups of order $3$, and the action of $G$ on these
subgroups gives a homomorphism $G \to S_4$ with kernel $Z(G)$.
\item
The center $Z(G)$ has two elements.
\item
$G/Z(G)$ is isomorphic to either~$A_4$ or $S_4$, and so
 $G$ has order 24 or 48.
\end{itemize}
\end{lemma}

\begin{proof}
Let $N$ be a maximal normal $2$-subgroup (i.e., the intersection of the $2$-Sylow subgroups of $G$).
Observe that characteristic Abelian subgroup $\mathcal A (G/N)$ of $G/N$ must be a nontrivial $3$-group
since if it had order divisible by~$2$ then we could violate the maximality of $N$ via the~$2$-Sylow subgroup
of $\mathcal A (G/N)$. Let $H/N$ be the unique cyclic group of order $3$
in~$\mathcal A (G/N)$. Here~$H$ is a subgroup of $G$ containing~$N$, and since $H/N$ is characteristic in $\mathcal A (G/N)$, we
see that $H/N$ is characteristic in~$G/N$. Since $N$ is characteristic in $G$ we then see that~$H$ is characteristic in $G$.
Let $C_3$ be a cyclic subgroup of $H$ of order $3$. Then we have that $H= NC_3$ is a semidirect product~$N \rtimes C_3$.
If $C_3$ acts trivially on $N$ then $H$ would be isomorphic to~$N \times C_3$ and would contain a characteristic subgroup of order 3.
This violates our assumption on normal subgroups of~$G$. So $C_3$ acts nontrivially on $N$.

Up to isomorphism the only $2$-Cycloidal group with an automorphism of order~3 is the quaternion group $Q_8$, so $N$ is isomorphic to $Q_8$.
Thus $H$ has $24$ elements.
Denote by $-1$ the unique element of $N$ of order $2$, which must be the unique element of $H$ of order 2.
Note that $-1$ is fixed by the action of $C_3$, and so must be in the center of $H = N C_3$.
In fact, the action of $C_3$ on the quaternions only fixes $\pm 1$ so the center $Z(H)$ is $\{ \pm 1 \}$.

Let $L$ be the normalizer of~$C_3$ in $H$.
Observe that $L \cap N$ and $C_3$ are normal in~$L$, and  so $L$ is isomorphic $(L \cap N) \times C_3$.
This means $C_3$ is characteristic in $L$.  So~$L$ cannot be normal in $H$, otherwise $C_3$ would be normal in $H$, and hence in~$G$,
a contradiction. So $L$ must have index at least $4$ in $H$. However, $L$ contains $C_3$ and~$Z(H)$. Thus $L$ has order $6$. Since $L$ is isomorphic to $(L \cap N) \times C_3 = Z(H) \times C_3$, it is cyclic of order~$6$.
By the orbit-stabilizer theorem (and the fact that $p$-Sylow groups are conjugate) we get that $H$ has exactly four $3$-Sylow subgroups.
This action gives a homomorphism $H \to S_4$. The kernel of this action is the intersection of the normalizers. Our description of $L$
applies to all the normalizers, and the intersection is seen to be $Z(H)$. So $Z(H)$ is the kernel of the action  $H \to S_4$.
Observe that the image of $L$ under $H \to S_4$ is isomorphic to $L/Z(H)$, so is cyclic of order $3$, and it fixes the element corresponding
to $C_3$. This observation applies not just to $C_3$ but to all 3-Sylow subgroups of $H$. So the image of $H \to S_4$ contains all 
subgroup of order $3$. This means that the image contains all of $A_4$, and must in fact be $A_4$ since~$H/\{\pm 1\}$ is a group of size $12$.

If $G = H$, then we are done.
So for the remainder of the proof we assume~$H$ is not all of $G$. We claim that $G$ has  
has order $2^m 3$ where $m > 3$. Otherwise every element $h$ of order $3$ in $H$ is of the form $g^3$
for some element of $g$ of order $9$.
Since such an~$g$ acts on $N$ as an automorphism of order 1 or 3, this forces $h$ to act trivially on $N$, a contradiction
to a previous conclusion.

The $2$-Sylow subgroups of $G$ cannot be normal in $G$ by maximality of $N$.
So there are three $2$-Sylow subgroups~$S_1, S_2, S_3$ since the number of such groups must divide $|G|/2^m$. 
Since $G$ acts transitively on $\{S_1, S_2, S_3\}$ by conjugation, we get  a homomorphism $G \to S_3$ whose image
is $S_3$ or $A_3$.
The kernel
is a normal $2$-subgroup of $G$ of index 3 or 6 in $G$. By maximality of $N$, this kernel is contained in $N$.
So the index of $N$ in $G$ is a divisor of $6$, and since it has size at least $6$ it is equal to $6$.
Hence (1) $G$ has size $48$, (2) each $S_i$ is a generalized quaternion group of size 16, (3) $H$ is the kernel of $G\to S_3$,
(4) the homomorphism $G\to S_3$ is surjective with each $S_i$ mapping to a different subgroup of order $2$ in $S_3$, and (5)~$H$ has index~2 in $G$ and its image under $G\to S_3$ is $A_3$.

Observe that every $3$-Sylow subgroup of $G$ is contained in $H$ since $[G: H] = 2$. Thus the set of $3$-Sylow subgroups of $G$ has size 4,
giving a homomorphism $G \to S_4$ extending the earlier surjection $H \to A_4$.
This means that the kernel $K$ of $G \to S_4$ has order $2$ or $4$.
The image of $K$ under $G \to S_3$ is a normal $2$-subgroup  of $S_3$, so~$K$ has trivial image in $S_3$. This means that $K\subseteq H$.
But we already know that the kernel of $H \to S_4$ is $Z(H)$.
Hence $K = Z(H)$. This implies that $G/Z(H)$ is isomorphic to $S_4$. 

We conclude by showing that $Z(H) = Z(G)$.
First observe that under the map~$G\to S_4$ the center $Z(G)$ maps to the center of~$S_4$
which is trivial. So~$Z(G)$ is contained in the kernel $Z(H)$. Conversely, since $S_1 \cap S_2 \cap S_3$ is the maximal normal
2-subgroup of $G$, this intersection is $H$. Thus $Z(H) \subseteq S_i$, 
and so  $Z(H)$ is the unique subgroup of $S_i$ of size 2.
So $Z(H)$ is contained in $Z(S_i)$ for each $i$. Since~$S_1, S_2, S_3$ generates $G$ we have that $Z(H) \subseteq Z(G)$.
\end{proof}
\begin{remark}
By a theorem of Burnside, all groups of order $p^m q^n$ are solvable for distinct primes $p$ and $q$.
So accepting this result allows us to drop the solvability assumption for groups of order $2^m 3^n$.
I will keep this assumption just to make the proofs a bit more accessible.
\end{remark}

There are two cases in the above lemma. We now link these cases to specific subgroups of $\bH^\times$.

\begin{proposition} \label{2T_iso_prop}
If $G$ is a solvable Sylow-cycloidal group of order 24 with no normal subgroup of order 3, then $G$ is isomorphic to the binary tetrahedral group~$2 T$.
\end{proposition}

\begin{proof}
By the above lemma, we can identify $Q_8$ with a normal subgroup of $G$.
Also, $G$ contains four subgroups of order $3$, each
which acts nontrivially on $Q_8$, and the center of $G$ has two elements.
Note that we have an embedding of $G/Z(G)$ into~$\Aut(Q_8)$.
This implies that if $\alpha$ and $\beta$ are distinct elements of order 3, their actions on $Q_8$
are distinct (otherwise $\alpha = \epsilon \beta$ where $\epsilon \in Z(G)$, and $\epsilon$
must be $1$ since~$\alpha$ has order $3$).

Note that any automorphism $\phi$ of order $3$ of $Q_8$ must permute its three subgroups of order 4,
and cannot fix any such subgroup (otherwise it would have to fix all three, and then $\phi^2$ would fix all of $Q_8$).
So there are only 4 possible values of $\phi(\mathbf{i})$, and they are contained in $\{ \pm \mathbf{j}, \pm \mathbf{k} \}$.
Once $\phi(\mathbf{i})$ is known, there are only two possibilities for $\phi(\mathbf{j})$.
Since $\mathbf{i}$ and $\mathbf{j}$ generate $Q_8$, this gives $8$ possibilities. All of these possibilities 
must occur as the automorphism associated with elements of order 3, since there are 8 such elements.

Thus we can find an element $g\in G$ of order 3 such that it acts on $Q_8$
by sending~$\mathbf{i}$ to $\mathbf{j}$ and sending $\mathbf{j}$ to $\mathbf{k}$.
Let $C_3$ be the subgroup generated by $g$. Then 
$$
G \cong Q_8 \rtimes C_3
$$
where $C_3$ acts on $G$ as specified. 
Thus $G$ is unique up to isomorphism. 

Since $2 T$ is freely representable, it is a Sylow-cycloidal group. Also~$2T/\{\pm 1\}$
is isomorphic to $A_4$. Since $A_4$ is solvable and has no normal subgroups of order~3,
the same is true of $2 T$. So $2T$ satisfies our assumptions for $G$, and so must be isomorphic to any such $G$.
\end{proof}

Here is a variant of the above:

\begin{corollary} \label{2T_iso_cor}
Let $G$ be a finite group of order 24 that has a unique element of order $2$
but does not have a unique subgroup of order $3$.
Then $G$ is isomorphic to the binary tetrahedral group~$2T$. 
\end{corollary}

\begin{proof}
Let $P_2$ be a $2$-Sylow subgroup of $G$. Then $P_2$ has a unique element of order~$2$, so must be isomorphic
to $Q_8$ or a cyclic group of order $8$. In particular, $G$ is a Sylow-cycloidal group. 
Note that the subgroups of order $3$ of $G$ are conjugate since they are Sylow subgroups, so
they cannot be normal since there is more than one. The result now follows from Proposition~98 if we grant that $G$ is solvable.

It is well-known that all subgroups of order 24 are solvable, but to see this directly for $G$ here let $C_2$ be the unique
subgroup of order $2$ and let $C_3$ be any subgroup of order $3$. Then 
the subgroup $C_2 C_3 =C_2 \rtimes C_3$
must be cyclic since~$C_2$ has no automorphism. So the normalizer of $C_3$ has at least 6 elements.
This means that there are at most~$4$ subgroups of $G$ of order $3$. Since the number of such Sylow
subgroups is congruent to 1 modulo $3$, we see there are exactly $4$ subgroups of order~$3$
and~$C_2 C_3$ is the normalizer of $C_3$. 
The action of $G$ on the set of subgroups of order~$4$ gives a homomorphism~$G \to S_4$.
The kernel of the action is the intersection of the normalizers for subgroups of order~3. Our description
of the normalizer shows this intersection is $C_2$.
Since $S_4$ is solvable, this implies that~$G$ is solvable.
\end{proof}

\begin{proposition} \label{binary_octahedral_prop}
Let $G$ be a solvable Sylow-cycloidal group of order 48 such that~$G/Z(G)$ is isomorphic to $S_4$.
Then $G$ is isomorphic to the binary octahedral group~$2 O$.
Furthermore $G$ has $8$ elements of order 3, and $4$ subgroups of order~3, and these elements
generate the unique subgroup of $G$ of index $2$,  and this index 2 subgroup is a binary tetrahedral group.
\end{proposition}

\begin{proof}
We start by checking that $G = 2O$ satisfies the hypothesis.
Recall that $2 O$ has $48$ elements.
We note that $2 O$ has center containing the elements~$\pm 1$, and since~$2O / \{\pm 1\} \cong O \cong S_4$
the center is exactly $\{ \pm 1 \}$ (since $S_4$ has trivial center). So $2 O / Z(2 O)$ is isomorphic to~$S_4$.

Now let $G$ be any subgroup of size 48 such that $G/Z(G) \cong S_4$.
Observe that~$Z(G)$ has order $2$ since~$G$ has order 48 and $S_4$ has order $24$.
So the generator of $Z(G)$ is the unique element of $G$ of order 2 (Lemma~\ref{order2_lemma}).
For now, we fix a particular isomorphism~$
G/Z(G) \to S_4.
$
Since $Z(G)$ has order $2$, the order of the image of $g \in G$ in~$S_4$ is either the 
same as the order of $g$ or has half the order of~$g$.
For example if~$g$ maps to an element of order 2 then $g$ must have order 2 or~4.
But in this case~$g$ is not the  unique element of order 2 since that maps to the identity element, so $g$ has order $4$
and $g^2$ is the unique element of order $2$.

In contrast if $g$ has order 3, then~$g$ must map to an element of order $3$
since~$3$ is odd. In particular if $g$  has order 3 then $g$ maps to a 3 cycle in $S_4$.
So if~$\mathcal S_3$ is the set of elements of order~$3$ in~$G$ then we have a  
map $\mathcal S_3 \to \mathcal T_3$ where~$\mathcal T_3$ is the set of three cycles.
Now let~$t \in \mathcal T_3$ be given. The subgroup $\left< t \right>$ of $S_4$ corresponds
to a subgroup $H_t$ of $G$ of order $6$ containing $Z(G)$ as a normal subgroup. 
In fact,~$H_t = Z(G) C = Z(G) \rtimes C$ where~$C$ is a 3-Sylow subgroup of $H_t$.
Since $Z(G)$ has no nontrivial automorphisms,~$H_t$ is just $Z(G) \times C$, so $C$ is the unique
subgroup of $G$ order 3 whose image in $S_3$ is $\left< t \right>$. Thus there is exactly one
element of $G$ that maps to $t$. This means that $\mathcal S_3 \to \mathcal T_3$ is a bijection.
In particular, there are $8$ elements of order~3 in~$G$, and $4$ subgroups of order 3 in $G$.

Let $H$ be the subgroup of $G$ generated by the set $\mathcal S_3$
of elements of order 3. The image in $S_4$ is $A_4$ (since $\mathcal T_3$ generates $A_4$).
This implies that $H$ is also of even order and so must contain the unique subgroup $Z(G)$
of $G$ of order $2$. Thus~$H/Z(G) \cong A_4$, and so $H$ has order $24$.
 Also if $H$ has a normal subgroup
of size 3, then its image in $A_4$ would also have a normal subgroup of size 3. But this is not the case.
So we have that $H$ is isomorphic to the binary tetrahedral group $2T$ (see previous proposition).
We also note that $H$ is the only subgroup of $G$ of order 24. So see this note that any other such group $H'$
would have to contain the unique element of~$G$ order $2$ and so would contain $Z(G)$. Thus its image
in $S_4$ would have size 12. But $A_4$ is the only subgroup of $S_4$ of index two\footnote{A subgroup $N$ of $S_4$
of index~2
is normal, and contains all three cycles since $G/N$ has order~2.
Since three cycles generate $A_4$ any such $N$ would have to be $A_4$.} and so $H'$ and $H$ would have the same image in $S_4$ and so would be equal.

This gives us enough information to describe $G$ in terms of relations. The final specification of $G$
will not depend on a particular isomorphism 
$
G/Z(G) \cong S_4
$
but only on the fact that such an isomorphism exists. 
Note that the subgroup $H$ (the unique subgroup of $G$ of index 2) and
the subset $\mathcal S_3$ do not depend on the map.
Choose an element $g_1 \in G$ of order~$3$. It can be any such element. 
Next choose $g_2$ to be any element of order $3$ such that the product $g_1 g_2$ has order 4.
To show this can be done it is useful to choose a map 
$
G/Z(G) \cong S_4
$
(by permuting the number if necessary from a given map)
so that $g_1$ corresponds to $(1 \, 2 \, 3)$.
Then the element~$g_2$ corresponding to $(1 \, 2 \, 4)$ will 
work since  $g_1 g_2$ maps to $(1 \, 3)(2 \, 4)$ of order~$2$ in $S_4$ and so~$g_1 g_2$ has order $4$
in $G$. In fact, once we have chosen a suitable $g_1$ and $g_2$, 
we can permute the numbering of the four elements permuted by $S_4$ so that
$g_1$ corresponds to $(1 \, 2 \, 3)$ and~$g_2$ corresponds to $(1 \, 2 \, 4)$.
Note that~$(3 4) (1 \, 2 \, 3) (3 4) = (1 \, 2 \, 4)$ so if we choose $\tau \in G - H$ mapping to $(3 4)$ then 
$\tau g_1 \tau^{-1}$ is an element of order 3 corresponding to $(1\, 2\, 4)$. Since $\mathcal S_3 \to \mathcal T_3$
is a bijection, this means that 
$$
\tau g_1 \tau^{-1}  = g_2.
$$
We also have 
$$
\tau g_1^{-1} \tau^{-1} = g_2^{-1}, \quad \tau g_2 \tau^{-1} = g_1, \quad \tau g_1^{-1} \tau^{-1} = g^{-1}_1.
$$
Here we made use of $ \tau^{-1} g \tau = \tau g \tau^{-1}$ for all $g\in G$ since $\tau^2$
has order~2 in $Z(G)$. By conjugating other 3-cycles in $S_4$ by $(3\, 4)$ we can see that
$$
\tau g \tau^{-1} = g^{-1}
$$
for the other four elements $g\in \mathcal S_3$.
This gives us eight relations, one for each element of $\mathcal S_3$. We verified they held by using a particular 
map $G/Z(G) \cong S_4$, but the actual relations themselves do not depend on the map.
In addition we have a ninth relation $\tau^2 = -1$ where $-1$ is the unique element of order $2$ in $H$.

Now consider the free product $H * C_4$ where $C_4$ is an abstract cyclic group of order $4$ with generator 
called $\tau$. Let $K$ be the normal subgroup generated by the~9 relations discussed above.
Note that every element of $H*C_4/K$ can be written as~$a$ or $a \overline \tau$
where $a$ is in the image of $H$ and $\overline \tau$ is the image of $\tau$.
This follows from the fact that $\mathcal S_3$ generates $H$. Thus the group $H*C_4/K$ 
has at most $48$ elements. However, $G$ satisfies these relations so there is a 
homomorphism $H*C_4/K \to G$, and this is a surjection since $G$ is generated by $\tau$ and $H$.
Thus $G$ is isomorphic to~$H*C_4/K$.

So if $G_1$ and $G_2$ are solvable Sylow-cycloidal groups of order~48 that contain
a common subgroup $H$ of size 24, and if   $G_1/Z(G_1)$ and $G_2/Z(G_2)$ are both isomorphic to~$S_4$, then $G_1$
must be isomorphic to $G_2$ since both are isomorphic to $H * C_4/K$ (and $K$ does not depend on $G_i$ but
only on $H$).
More generally, if $G_1$ and $G_2$ are are solvable Sylow-cycloidal groups of order 48 such
that $G_1/Z(G_1)$ and $G_2/Z(G_2)$ are isomorphic to $S_4$, then as observed above 
each~$G_i$ has a subgroup isomorphic to~$2 T$. By identifying these subgroups
we reduce to the situation where $G_1$ and $G_2$ share a subgroup of order $24$. We conclude
that $G_1$ and $G_2$ are isomorphic under these conditions.
\end{proof}

Here are some useful observations linking the $2$-Sylow subgroup of $G$ to the quotient $G/O(G)$.

\begin{proposition} \label{same_sylow_prop}
Let $G$ be finite group and let $O(G)$ be its maximal normal subgroup of odd order.
Then $G$ and $G/O(G)$ have isomorphic $2$-Sylow subgroups.
\end{proposition}

\begin{corollary}
Let $G$ be a solvable Sylow-cycloidal group whose $2$-Sylow subgroup~$S$ is not quaternionic of order 8 or 16.
Then $G/O(G)$ is isomorphic to $S$.
\end{corollary}

\begin{corollary}\label{sylow_cyclic_type_cor}
Let $G$ be a Sylow-cycloidal group. Then $G$ is a Sylow-cyclic group if and only if $G/O(G)$ is a cyclic $2$-group.
\end{corollary}

So we divide the solvable Sylow-cycloidal groups $G$ into four mutually exclusive types:
\begin{enumerate}
\item
Sylow-cyclic groups. These are the Sylow-cycloidal groups where $G/O(G)$ is cyclic.
\item
Quaternion type. These are defined to encompass the Sylow-cycloidal groups where $G/O(G)$ is a generalized quaternion group.
\item
Binary tetrahedral type. These are defined to encompass the Sylow-cycloidal groups where $G/O(G) \cong 2 T$.
\item
Binary octahedral type. These are defined to encompass the Sylow-cycloidal groups where $G/O(G) \cong 2 O$.
\end{enumerate}

In addition, there are non-solvable Sylow-cycloidal groups that we will consider later. These are Sylow-cycloidal
groups that contain a perfect Sylow-cycloidal subgroup.

\section{Type 1: Sylow-cyclic groups}

As noted out above, a Sylow-cycloidal group $G$ is of this type if and only if $G/O(G)$ is a cyclic $2$-group.
Such groups were treated in Section~\ref{sylow_cyclic_chapter}.
In particular, this type of group $G$ is freely representable if and only if it has a unique subgroup of order $p$ for each
prime $p$ dividing the order of $G$. 

In order to compare with later results it is convenient to break out the odd part from the even part:

\begin{proposition} \label{sylow_cyclic_odd_prop}
Let $G$ be a Sylow-cyclic group of order $2^k n$ where $n$ is odd and where $k\ge 1$.
Then $G$ has a normal Sylow-cyclic subgroup $M$ of order $n$. For such~$M$,
the group $G$ is freely representable if and only if (1) $M$ is freely representable
and (2) $G$ has a unique element of order $2$.
\end{proposition}

\begin{proof}
Let $M = O(G)$, so $G/M$ is a cyclic $2$-group (Corollary~\ref{sylow_cyclic_type_cor}), and $O(G)$ has
odd order by definition, so $M$ has order $n$.

If $G$ is freely representable then (1) and (2) hold by earlier results. So assume~(1) and (2).
To show $G$ is freely representable, it is enough to show that $G$ has a unique subgroup of order $p$
for each $p$ dividing $2^k n$. For $p=2$ we are covered by assumption~(2).
For an odd prime $p$ we note that every subgroup of  $G$ of order $p$ is a subgroup of $M$
since $G/M$ is a $2$-group. By (1) there is a unique subgroup of~$M$ (and hence of~$G$) of order $p$.
\end{proof}

\section{Case 2:  Quaternion Type}

As noted in Proposition~\ref{same_sylow_prop}, these groups have $2$-Sylow subgroups that are generalized
quaternion groups, and if conversely if $G$ is a solvable Sylow-cycloidal group with $2$-Sylow subgroups that are generalized
quaternion groups of order $32$ or more then $G$ must be of this type. (If $G$ is a solvable Sylow-cycloidal
group with  $2$-Sylow subgroups isomorphic to the quaternion group $Q_8$ of order $8$, it can either
be of this type or of binary tetrahedral type. If $G$ is a solvable Sylow-cycloidal
group with~$2$-Sylow subgroups isomorphic to the generalized quaternion group $Q_{16}$ of order $16$, it can either
be of this type or of binary octahedral type.)

Observe that for Sylow-cycloidal groups $G$ of quaternion type, any $2$-Sylow subgroup $Q$ of $G$
functions as a complement for the normal subgroup $O(G)$ of $G$.
Thus
$$
G = O(G) Q = O(G) \rtimes Q.
$$
In particular, up to isomorphism $G$ is determined by $O(G)$ and the action of $Q$ on~$O(G)$.

\begin{proposition}\label{quaternion_2_prop}
Let $G$ be a Sylow-cycloidal group of quaternion type. Then~$G$ has a unique element of order 2.
This element is in the center of $G$.
\end{proposition}

\begin{proof}
Let $Q$ be a $2$-Sylow subgroup of $G$, which is a generalized quaternion group (Proposition~\ref{same_sylow_prop}).
We start by considering the action of $Q$ on $G$.
By  Proposition~\ref{aut_quotient_prop}, we have that
$
A = \Aut(O(G))  / O(\Aut(O(G)))
$
is an Abelian $2$-group. The action of~$Q$ gives a map into $A$:
$$
Q \to \Aut(O(G))  \to \Aut(O(G))  / O(\Aut(O(G))) = A.
$$
Since $A$ is Abelian, this map has a nontrivial kernel. Thus if $C_2$ is the unique subgroup
of $Q$ of order 2, then $C_2$ is in this Kernel. In particular, the image of~$C_2$
in $\Aut(O(G))$ must land in $O(\Aut(O(G)))$. But $O(\Aut(O(G)))$ has odd order, so the
image of $C_2$ in $\Aut(O(G))$ is trivial. Thus $C_2$ acts trivially on $O(G)$.
Since~$G = O(G) Q$ we have that $C_2$ is in the center of $G$.

Observe that the subgroup $O(G) C_2$ of $G$ corresponds to the unique subgroup of $G/O(G)$ of order two.
If $g \in G$ has order 2, then its image in $G/O(G)$ is the unique element of order 2, 
and so $g \in O(G) C_2$. 
Since $C_2$ is in the center of $G$, we have $O(G) C_2 = O(G) \times C_2$, and so every element of 
$O(G) C_2$ of order 2 must be in~$C_2$. We conclude that every element of order $2$ is in $C_2$. In other words,
there is a unique element of order $2$ in $G$.
\end{proof}

Let $G$ be a Sylow-cycloidal group of quaternion type and let
 $C_2$ be its unique subgroup of order 2. Let $M = O(G) C_2$.
Observe that every element of odd prime order must be in $O(G)$, so every element of $G$ of prime order
is in $M$. By Corollary~\ref{sufficient_cor}, $G$ is freely representable if and only if $M$ is freely representable.
Since~$M \cong O(G) \times C_2$ we have the $M$ is freely representable if and only
if $O(G)$ and $C_2$ are freely representable (Corollary~\ref{rel_prime_cor}), but of course $C_2$ is freely representable.
Thus we get the following:

\begin{theorem} \label{quaternion_type_thm}
Let $G$ be a Sylow-cycloidal group of quaternion type. 
Then the following are equivalent:
\begin{enumerate}
\item
$G$ is freely representable.
\item
$O(G)$ is freely representable.
\item
For each prime $p$ dividing the order of $G$, there is exactly one subgroup of $G$ of order $p$.
\end{enumerate}
\end{theorem}

\begin{proof}
The equivalence $(1) \iff (2)$ was addressed in the discussion proceeding the statement of the theorem.
The implication $(3) \implies (1)$ follows from Proposition~\ref{unique_freely_rep_prop}.

So we just need to verify that $(2)$ implies $(3)$. We know that $(3)$ holds for $p=2$
by the previous proposition (independent of whether $(2)$ is true or not).
So suppose that~$(2)$ holds and that $p$ is an odd prime dividing the order of $G$.
Observe that every subgroup of order $p$ of $G$ must actually be a subgroup of $O(G)$
since $G/O(G)$ has even order.
Finally, $(2)$ implies $O(G)$ has exactly one subgroup of order $p$ by Theorem~\ref{sc_fr_thm}.
\end{proof}

Now we consider other characterizations of this type of group.

\begin{proposition}\label{quaternion_m_condition_prop}
Let $G$ be a finite group. Then $G$ is a Sylow-cycloidal group of quaternion type
if and only if it is a semidirect product $M \rtimes Q$ where $M$ is a Sylow-cyclic group of odd order
and $Q$ is a generalized quaternion group.

If $G$ is a semidirect product $M \rtimes Q$ where $M$ is a Sylow-cyclic group of odd order
and $Q$ is a generalized quaternion group, then $G$ is freely representable if and only if $M$ is freely representable.
\end{proposition}

\begin{proof}
We mentioned earlier that if $G$ is a Sylow-cycloidal group of quaternion type then $G = O(G) \rtimes Q$
where $Q$ is any $2$-Sylow subgroup of $G$ and where $Q$ is a generalized quaternion group.
Conversely suppose $G$ is of the form $M \rtimes Q$. Then~$M$ is isomorphic to a normal subgroup of $G$
of odd order, and the $2$-group $Q$ is isomorphic to the corresponding quotient of $G$. Thus $O(G)$
must be isomorphic to $M$, and $G/O(G)$ is isomorphic to $Q$.
So by the above theorem, $G$ is freely representable if and only if $M$ is freely representable.
\end{proof}

\begin{proposition}
Let $G$ be a finite group. Then $G$ is a Sylow-cycloidal group of quaternion type
if and only \emph{(1)} the $2$-Sylow subgroups of $G$ are generalized quaternion groups, and 
\emph{(2)} $G$ has a Sylow-cyclic subgroup $M$
of index $2$. In this case~$G$ is freely representable if and only if $M$ is freely representable.
\end{proposition}

\begin{proof}
Suppose is a Sylow-cycloidal group of quaternion type, so
$G/O(G)$ is a generalized quaternion group. This quotient contains a  cyclic
subgroup of index~2 which we can write as $M/O(G)$ where $M$ is a  subgroup of $G$ containing~$O(G)$.
Note that $M$ has index $2$ in $G$. 
Since~$O(G)$ has odd order, the $2$-Sylow subgroups of $M$ are isomorphic to the cyclic group $M/O(G)$.
Thus $M$ is a Sylow-cyclic subgroup of~$G$.
We also note that $G/O(G)$ is isomorphic to the 2-Sylow subgroups of $G$, so these~$2$-Sylow subgroups of $G$ are  generalized quaternion groups.

Conversely, suppose (1) and (2) hold. Since $M$ is a Sylow-cyclic group, the quotient $M/O(M)$ is a cyclic
$2$-group. Since $G/M$ has order 2, the subgroup $O(G)$ must be a subgroup of $M$, so $O(G) \subseteq O(M)$
by the maximality of $O(M)$. Since~$M$ is normal in $G$ and since $O(M)$ is characteristic in $M$ it follows
that $O(M)$ is normal in $G$. Thus $O(G) = O(M)$. Since $M/O(M) = M/O(G)$ is a $2$-group and since $G/M$ is a $2$-group, it follows that $G/O(G)$ is a $2$-group. Any Sylow $2$-group of $G$ is thus isomorphic to $G/O(G)$, and
so $G/O(G)$ is a generalized quaternion group. Also note that since $G/M$ has order 2, all odd order Sylow-subgroups of $G$ are contained in $M$, and so are cyclic. Thus $G$ is a Sylow-cycloidal group of quaternion type.

If $G$ is freely representable, then so is the subgroup $M$. Conversely, if $M$ is freely representable
then so is $O(G)$ since, as noted above, $O(G)$ is a subgroup of~$M$. Thus $G$ is freely representable
by Theorem~\ref{quaternion_type_thm}.
\end{proof}

\begin{proposition}
Let $G$ be a finite group and let $S$ be a $2$-Sylow subgroup of~$G$.
Then $G$ is a Sylow-cyclic group  if and only if $S$ is cyclic and $G$ has a normal
Sylow-cyclic
subgroup $M$ of index $|S|$.
Similarly, $G$ is a Sylow-cycloidal group of quaternion type if and only if $S$ is a generalized quaternion group and 
$G$ has a normal
Sylow-cyclic
subgroup $M$ of index $|S|$.
\end{proposition}

\begin{proof}
If $G$ is Sylow-cyclic group or a Sylow-cycloidal group of quaternion type, then let $M = O(G)$ which is normal.
Conversely, suppose there is a normal Sylow-cyclic subgroup~$M$ of index $|S|$. This group $M$ has odd order, and is maximal
with this property, so~$M = O(G)$. It follows that $G/O(G) \cong S$. Note also that every odd ordered Sylow subgroup of $G$
is actually in $M$ since $G/M$ has even order. Thus every odd ordered Sylow subgroup of $G$ is cyclic.
\end{proof}


\section{Case 3:  Binary Tetrahedral Type}

Let $G$ be a Sylow-cycloidal group of binary tetrahedral type and let $O(G)$ be its maximal normal subgroup of odd order.
Then $G/O(G)$ is isomorphic to the binary tetrahedral group $2T$, so
 the $2$-Sylow subgroups of $G$ are 
all isomorphic to the quaternion group of order 8 (Proposition~\ref{same_sylow_prop}).
What is interesting about this case is that the $2$-Sylow subgroup of $G$ is unique, and so is characteristic:

\begin{proposition} \label{binary_tetrahedral_normal_prop}
Let $G$ be a Sylow-cycloidal group of binary tetrahedral type.
Then $G$ has a unique $2$-Sylow subgroup of order $8$, and this $2$-Sylow  subgroup is isomorphic to the quaternion group with 8 elements. Furthermore, the $2$-Sylow subgroup of $G$ centralizes~$O(G)$.
\end{proposition}

\begin{proof}
We start with the fact that $2T$ has a cyclic quotient $C_3$ of order 3.
Let $H$ be the kernel of the composition $G \to G/O(G) \cong 2 T \to C_3$.
Note that every $2$-Sylow subgroup of $G$ must be in the kernel $H$. So we just need to show that $H$
has a unique $2$-Sylow subgroup.
Let $Q$ be a $2$-Sylow subgroup of $G$, and observe that $Q$ is a complement for $O(G)$ in $H$.
So 
$$H = O(G) Q = O(G) \rtimes Q.$$
To prove the uniqueness result for $H$, and hence for $G$, it is enough to show that~$Q$ acts trivially on $O(G)$.

By  Proposition~\ref{aut_quotient_prop}, 
$
A = \Aut(O(G))  / O(\Aut(O(G)))
$
is an Abelian~$2$-group. Let~$G$ act on $O(G)$ by conjugation.
Then we have the composition 
$$
G \to \Aut(O(G))  \to \Aut(O(G))  / O(\Aut(O(G))) = A.
$$
Since the codomain is a $2$-group, we have $O(G)$ is in the Kernel.
So we get a map
$$
2 T \cong G/O(G) \to A.
$$
The kernel must contain the commutator subgroup of $2T$.  The commutator subgroup of $2T$
contains the commutator subgroup of $Q_8$, so contains the unique subgroup $C_2$ of $2T$
of order $2$. But $2T/C_2$ is isomorphic to $T = A_4$, whose commutator subgroup is the normal subgroup of order $4$.
Thus the commutator subgroup $(2T)'$ of $2T$ has index $3$ in $2T$.
Thus we get a homomorphism
$$
(2T)/(2T)'  \to A.
$$
Since $(2T)/(2T)'$ has order 3, and $A$ is a $2$-group, we get that the image of $G$ in $A$ is trivial.
In other words, the image of $G$ in $\Aut(O(G))$ has odd order.
In particular, the image of $Q$ in $\Aut(O(G))$ must be trivial. So $Q$ acts trivially on $O(G)$.
\end{proof}

\begin{corollary}\label{binary_tetrahedral_2_cor}
Let $G$ be a Sylow-cycloidal group of binary tetrahedral type.
Then~$G$ has a unique element of order $2$.
\end{corollary}

Let $M/O(G)$ be the subgroup of $G/O(G)$ of order $8$. Here $M$ is a subgroup of~$G$ containing $O(G)$,
and since $M/O(G)$ is normal in $G/O(G)$ we have that $M$ is a normal subgroup of~$G$.
Let $Q_8$ be the $2$-Sylow subgroup of $G$. Observe that $Q_8$ is in $M$ since $G/M$ has order 3.
In fact $Q_8$ is a complement for $O(G)$ in $M$, so 
$$
M = O(G) Q_8 = O(G) \times Q_8
$$
since $Q_8$ acts trivially on $O(G)$.
We have that $Q_8$ is freely representable, so $M$ is freely representable if and only if $O(G)$ is freely
representable (Corollary~\ref{rel_prime_cor}).
The further analysis of freely representable depends on whether or not $9$ divides the order of $G$.

\begin{theorem}\label{quaternionic_type1_thm}
Let $G$ be a Sylow-cycloidal group of binary tetrahedral type.
If $9$ divides
the order of $G$ then the following are equivalent.
\begin{enumerate}
\item
$G$ is freely representable.
\item
$O(G)$ is freely representable.
\item
For each prime $p$ dividing the order of $G$, there is exactly one subgroup of $G$ of order $p$.
\end{enumerate}
\end{theorem}

\begin{proof}
The property of being freely representable is inherited by subgroups so~$(1) \implies (2)$.

Now suppose that $(2)$ holds. Note that   $(3)$ holds for $p=2$
(independent of whether $(2)$ is true or not) by the above Corollary.
Suppose that~$C$ is a subgroup of~$G$ of order $p$ where $p$ is an odd prime.
If $p \ne 3$ then $C$ is in $O(G)$ since $G/O(G)$ has order prime to $p$.
If $p = 3$, then $C$ is contained in a $3$-Sylow subgroup $P$ of order at least $9$.
The map $P$ to $G/O(G)$ has image of size $3$, so it has a nontrivial kernel.
Since $P$ is cyclic, all nontrivial subgroups of $P$ contain $C$. So $C$ is in the kernel of $P\to G/O(G)$.
In other words, $C$ is in $O(G)$.
Since all subgroups of odd prime order of $G$ are in $O(G)$, we have uniqueness for each prime order by Theorem~\ref{sc_fr_thm}. So $(3)$ holds.

Finally the implication $(3) \implies (1)$ follows from Proposition~\ref{unique_freely_rep_prop}.
\end{proof}

\begin{theorem} \label{quaternionic_type2_thm}
Let $G$ be a Sylow-cycloidal group of binary tetrahedral type.
If  $9$ does not divide
the order of $G$ then $G$ has a subgroup $H$ isomorphic to $2T$ and 
$$
G = O(G) H \cong O(G) \rtimes 2T.
$$
Furthermore, the following are equivalent.
\begin{enumerate}
\item
$G$ is freely representable.
\item
$O(G)$ is freely representable and $G \cong O(G) \times 2T$.
\item
For each prime $p \ne 3$ dividing the order of $G$ there is exactly one subgroup of~$G$ of order $p$,
and there are $4$ subgroups of order $3$ in $G$.
\end{enumerate}
\end{theorem}

\begin{proof}
Let $C_3$ be a $3$-Sylow subgroup of $G$, and let $Q_8$ be the unique $2$-Sylow subgroup of $G$.
Then $Q_8$ is normal in $G$, so $H = Q_8 C_3$ is a subgroup of $G$ of order~$24$. Note that $O(G)$ has order prime to $24$,
so $H$ maps isomorphically onto the quotient~$G/O(G)$. Hence $H$ is isomorphic to $2T$. Also
observe that $H$ is a complement to $O(G)$ so
$$
G = O(G) H \cong O(G) \rtimes 2T.
$$

Next observe that $O(G) C_3 = O(G) \rtimes C_3$ is a Sylow-cyclic subgroup of $G$.
So if~$G$ is freely representable, then same is true of $L = O(G) C_3$. 
If $L = O(G) C_3$ is freely representable then, by Theorem~\ref{sc_fr_thm}, $C_3$ is the only
subgroup of $L$ of order~3. Thus $C_3$ is normal in $L$, and so $L = O(G) \times C_3$.
In particular $C_3$ acts trivially on~$O(G)$. Since $H = Q_8 C_3$ and since $Q_8$ acts
trivially on $O(G)$ then $H$ acts trivially on $O(G)$. So $(1) \implies (2)$ holds.

Note that $(2) \implies (1)$ by  Corollary~\ref{rel_prime_cor}.

Next observe that if $G \cong O(G) \times 2T$ then every subgroup of prime order $p \ne 2, 3$ of $G$
must be in $O(G)$ and the subgroups of prime order $p=2$ or $p=3$ correspond to the
subgroup of $2 T$ of prime order. If $O(G)$ is freely representable
then there is one subgroup of order $p$ for each $p\ne 2, 3$ dividing the order of $G$ (and hence
the order of $O(G)$) (Theorem~\ref{sc_fr_thm}). Observe that $2 T$ has one subgroup of order 2
and four subgroups of order 3. So $(2) \implies (3)$.

Finally suppose $(3)$ holds. We can conclude that $O(G)$
is freely representable by Proposition~\ref{unique_freely_rep_prop}.
Since $H$ itself has four subgroups of order 3, we have at least four subgroups of order $3$.
Note the four subgroups of $H$ of order $3$ map to distinct subgroups of $G/O(G)$.
As before, let $L = O(G)C_3$ where $C_3$ is a cyclic subgroup of $H$. Any subgroup of order 3
in $L$ maps to the same subgroup of $G/O(G)$ as $C_3$. Since there are only 4 subgroups of order 3
in $G$ we conclude that $C_3$ is the unique subgroup of order $3$ in $L$. So $C_3$ is normal in $L$,
and $L = O(G) \times C_3$. Thus $C_3$ acts trivially on $O(G)$. Since $H = Q_8 C_3$
and since $Q_8$ acts trivially on $O(G)$, we conclude that $G = O(G) \times H$.
So $(3)$ implies $(2)$.
\end{proof}

For convenience we have divided into cases depending on whether or not $9$ divides the order of $G$.
Now we will see another more unified approach. As before let $Q_8$ be the unique $2$-Sylow
subgroup. Let $S_3$ be a $3$-Sylow subgroup of $G$. We note that $Q_8 S_3$ maps onto $G/O(G)$,
which implies that $S_3$ is not in the centralizer of $Q_8$. So $S_3$ must act on $Q_8$ by sending
a generator to an automorphism of $Q_8$ of order 3.

Recall that $O(G)$ acts trivially on $Q_8$. Consider $M = O(G) S_3 = O(G) \rtimes S_3$.
Note that the image of $M$ in $\Aut (Q_8)$ has order 3. Also note that 
$$
G = Q_8 M = Q_8 \rtimes M.
$$

\begin{proposition} \label{binary_tetrahedral_m_prop}
Let $G$ be a Sylow-cycloidal group of binary tetrahedral type, and let $Q$ be its
unique $2$-Sylow subgroup of of $G$, which is 
 isomorphic to the quaternion group with 8 elements.
Then there is a complement $M$ to $Q$ in $G$ which is a Sylow-cycloidal group of odd order:
$$
G = Q M = Q \rtimes M.
$$
Here $M \to \Aut Q$ has image of order $3$.
Moreover, $G$ is freely representable if and only if $M$ is freely representable.
\end{proposition}

\begin{proof}
The only thing left to show is that  if $M$ is freely representable, then $G$ is freely representable.
Since $M$ is freely representable, then the same is true of $O(G)$ since it is isomorphic to a subgroup of $M$.
If $9$ divides the order of $G$ then $G$ must be freely representable by Theorem~\ref{quaternionic_type1_thm}.
So from now on we assume that $9$ does not divide the order of $G$.

Since $M$ is freely representable, it has a unique subgroup $C_3$ of order $3$ (Theorem~\ref{sc_fr_thm}).
Since $O(G)$ and $C_3$ are normal in $M$ we have $M = O(G) \times C_3$ and $C_3$ acts
trivially on $O(G)$. Note that $Q$ acts trivially on $O(G)$, so $H = Q C_3$ acts trivially on $O(G)$.
This means that $G = O(G) \times H$. However, $H$ is isomorphic to $G/O(G)$.
Thus $G \cong O(G) \times 2T$. So $G$ is freely representable by Theorem~\ref{quaternionic_type2_thm}.
\end{proof}

We can strengthen the above:

\begin{proposition}
Let $G$ be a finite group and let $Q$ be a $2$-Sylow subgroup of~$G$.
Then $G$ is a Sylow-cycloidal group of binary tetrahedral type if and only 
if~$Q$ is a normal quaternionic subgroup of $G$  and there exists a non-normal Sylow-cyclic subgroup~$M$ of index $|S|$ in $G$.
In this case $G \cong Q \rtimes M$ where $Q$ is a quaternion group with 8 elements
and where the action map $M \to \Aut Q$ has image of size~3. Additionally, in this case
$G$ is freely representable if and only if $M$ is freely representable.
\end{proposition}

\begin{proof}
In light of the previous proposition we just need to show that 
$G$  is a Sylow-cycloidal group of binary tetrahedral type under the assumption that
$Q$ is a normal quaternionic group in $G$ and there exists a non-normal Sylow-cyclic subgroup $M$ of index $|S|$ in $G$.

Under these assumption,  every Sylow-subgroup of $M$ of odd order is actually a Sylow-subgroup of $G$ since $[G:M]$ is even,
and every Sylow-subgroup of $G$ is conjugate to a Sylow-subgroup of $M$ by the Sylow theorems.
Thus every Sylow-subgroup of $G$ of odd order is cyclic. 
Note also that $M$ is isomorphic to $G/Q$ since~$M$ is of odd order. Thus $G/Q$ and $Q$ are solvable,
and so $G$ is solvable.

Thus $G$ is a solvable Sylow-cycloidal group.
In addition $O(G)$ is a proper subgroup of $M$ since~$O(G)$ is normal. So $G/O(G)$ is not a $2$-group.
Thus $G$ is not Sylow-cyclic, and is not of quaternion type. So~$G$ is either of binary tetrahedral or binary octahedral type.
Note that the image of~$Q$ in $G/O(G)$ is a $2$-Sylow subgroup of $G/O(G)$ that is normal. This rules out
the binary octahedral type (since the existence of such a normal $2$-Sylow subgroup in $2O$ gives
a unique $2$-Sylow subgroup in its quotient $S_4$, contradicting the fact that two-cycles of $S_4$ generates $S_4$).
Thus~$G$  is a Sylow-cycloidal group of binary tetrahedral type.
\end{proof}

\section{Case 3:  Binary Octahedral Type}

We start with some basic observations about  key subgroups of this type of Sylow-cycloidal group:

\begin{proposition} \label{binary_octahedral_subgroups_prop}
Let $G$ be a Sylow-cycloidal group of binary octahedral type. 
Then every $2$-Sylow subgroup of $G$ is a generalized quaternion group of order 16,
and these $2$-Sylow subgroups are not
normal in $G$.
In addition $G$ contains a unique
subgroup~$H$ of index~$2$, and this subgroup is of binary tetrahedral type.
Moreover, $G$ contains a unique quaternion subgroup $Q$ of order 8, and this group $Q$ is contained in $H$.
Finally, $G$ contains exactly four subgroups of index $16$; these subgroups are conjugate subgroups of $H$
hence are conjugate in $G$; these subgroups are not normal in $H$ hence are not normal in $G$;
these subgroups each contain $O(G)$ as a subgroup of index 3; these subgroups are Sylow-cyclic groups of odd order;
and these subgroups are maximal among subgroups of odd order.
\end{proposition}

\begin{proof}
Every $2$-Sylow subgroup $S$ of $G$ is isomorphic to a $2$-Sylow subgroup of~$G/O(G) \cong 2O$, so is a
generalized quaternion group of order 16 (Proposition~\ref{same_sylow_prop}).
The image $\overline S$ in $G/O(G)$ of a $2$-Sylow subgroup $S$ contains the unique element of~$G/O(G)$ of order 2, so these Sylow subgroups $\overline S$
correspond to Sylow subgroups of the quotient $O \cong S_4$. But Sylow subgroups of $S_4$ are not normal (a normal~$2$-Sylow subgroup $N$ in $S_4$ would have to contain all 
two cycles since $S_4/N$ has order~3, but the collection of two cycles generate $S_4$).
Thus such an $\overline S$ is not normal in $G/O(G)$, and so $S$ cannot be normal in $G$.

Let $H$ be the subgroup of $G$ containing $O(G)$ such that $H/O(G)$ corresponds to the binary tetrahedral subgroup
of $G/O(G)$. Clearly $O(G) \subseteq O(H)$, but since~$H$ is normal in $G$ and since $O(H)$ is characteristic in $H$,
it follows that $O(H)$ is normal in $G$, and so $O(G) = O(H)$. Hence $H/O(H)$ is 
a binary tetrahedral group.

Suppose $L$ is any subgroup of $G$ of index 2. Thus $L$ is normal in $G$. Since 
the quotient~$G/L$ has two elements and $O(G)$ has odd order, $O(G)$
is contained in $L$. Note that $L/O(G)$ has index 2 in $G/O(G)$, so $L /O(G) = H/O(G)$ (Proposition~\ref{binary_octahedral_prop}).
This gives us $L = H$, and so $H$ is the unique subgroup of index 2 in $G$.

By Proposition~\ref{binary_tetrahedral_normal_prop},  $H$ has a unique subgroup $Q_H$ of order $8$
and this group is a quaternion group.
 This group is characteristic in $H$ and so is normal in $G$. 
By the Sylow theorems, $Q_H$ is contained in some $2$-Sylow subgroup of $G$, and hence in all since
$Q_H$ is normal (and all $2$-Sylow subgroups are conjugate). Every quaternion subgroup $Q$ of $G$ of order $8$ is contained in some $2$-Sylow subgroup $S$ of~$G$
by the Sylow theorems, so both $Q_H$ and $Q$ are subgroups of $S$. Thus $Q= Q_H$ since every
general quaternion group contains a unique subgroup isomorphic to the quaternion group with 8 elements.
So $Q_H$ is the unique subgroup $Q_H$ isomorphic to the quaternion group with 8 elements.

Each subgroup of order $3$ in $G/O(G)$ is uniquely of the form $M/O(G)$ where~$M$
is a subgroup of $G$ containing $O(G)$. Each such $M$ is of index $16$ in $G$ and contains~$O(G)$
as a subgroup of index 3. Since each such $M$ is of odd order,  $M$ is a Sylow-cyclic group.
Having index $16$ in $G$, each such $M$ must be maximal among subgroups of odd order in $M$.
Since $M \ne O(G)$ this means $M$ cannot be normal in $G$ since $O(G)$ is the maximal subgroup of odd order.

Next we argue that each  subgroup of index 16 in $G$ arises in this way.
Suppose~$M$ has index $16$  in $G$. Then $O(G) M / O(G)$ is isomorphic to $M/(M\cap O(G))$
which has odd order. Thus $O(G) M$ has odd order. 
But since $M$ has index 16 in $G$, there is no strictly larger subgroup of odd order. So
 $M = O(G)M$, and so $O(G)$ is contained in $M$. In particular, such a group corresponds to a 
subgroup $M/O(G)$ of $G/O(G)$ of index 16 and order 3.

By Proposition~\ref{binary_octahedral_prop}, there are $4$ subgroups of $G/O(G)$ of order 3,
and they are all contained in $H/O(G)$. They constitute the $3$-Sylow
subgroups of $H/O(G)$ hence are conjugate in $H/O(G)$ by the Sylow theorems, and thus cannot be normal
in~$H/O(G)$. This means that there are 4 subgroups of $G$ of index 16,
they are all contained in $H$, they are conjugate in $H$, and cannot be normal in $H$.
\end{proof}

In particular, there is a unique subgroup of order $2$ in any group of this type:

\begin{proposition} \label{binary_octahedral_2_prop}
Let $G$ be a Sylow-cycloidal group of binary octahedral type. 
Then $G$ has a unique element of order $2$.
\end{proposition}

\begin{proof}
By the above proposition, 
we have a normal subgroup of order 8 in $G$, and this group has a unique subgroup of order 2.
Thus we have a normal subgroup of order 2 in G. The result follows from Lemma~\ref{order2_lemma}.
\end{proof}

The main theorem about freely representable groups of this type is as follows:

\begin{theorem}\label{binary_octahedral_thm}
Let $G$ be a Sylow-cycloidal group of binary octahedral type. Let~$H$ be the unique
subgroup of $G$ of index $2$, which is of binary tetrahedral type. Let $M$
be any of the four subgroups of $G$ of index 16, which is Sylow-cyclic subgroup of $H$ of odd order.
Then the following are equivalent
\begin{enumerate}
\item
$G$ is freely representable.
\item
$H$ is freely representable.
\item
$M$ is freely representable.
\end{enumerate}
Furthermore, if $9$  divides the order of $G$ then if $O(G)$ is freely representable, then so is $G$.
\end{theorem}

\begin{proof}
The implication $(1) \implies (2) \implies (3)$ is clear since a subgroup of a freely representable group 
is freely representable (Proposition~\ref{freely_representable_subgroup_prop}). So we just need to show that $(3) \implies (1)$. In fact we will proceed by showing
$(3) \implies (2) \implies (1)$.

Suppose that $M$ is freely representable. By Proposition~\ref{binary_tetrahedral_m_prop}, $H$ is also 
freely representable.
Since $G/H$ has order 2, any subgroup of  $G$ odd prime order $p$ is a subgroup of $H$.
Also, by the previous proposition,  $G$ has a unique subgroup of order $2$ and this is a subgroup of $H$ since $H$ has even order.
By Corollary~\ref{sufficient_cor} we conclude that $G$ is freely representable.

Now suppose $9$ divides the order of $G$ and that $O(G)$ is freely representable.
Note that $O(H)$ is characteristic is $H$, and $H$ is normal in $G$, so $O(H)$ is normal in~$G$.
This implies that $O(H) \subseteq O(G)$ so that $O(H)$ is freely representable.
Thus~$H$ is freely representable by Theorem \ref{quaternionic_type1_thm},
which as we have seen implies that $G$ is freely representable.
\end{proof}

\section{Final Observations for the Solvable Case}

We make some observations about solvable Sylow-cycloidal groups in general.

\begin{proposition} \label{solvable2_prop}
Let $G$ be a solvable Sylow-cycloidal group.
If $G$ is not a Sylow-cyclic group, then $G$ always has a unique element of order 2.
If $G$ is a Sylow-cyclic group then $G$ has a unique element of order 2 if and only if
$G$ has a normal subgroup of order 2.
\end{proposition}

\begin{proof}
This was proved for each type individually. See Lemma~\ref{order2_lemma},
Proposition~\ref{quaternion_2_prop},
Corollary~\ref{binary_tetrahedral_2_cor},
and Proposition~\ref{binary_octahedral_2_prop}.
\end{proof}

\begin{proposition} \label{general_m_prop}
Let $G$ be a solvable Sylow-cycloidal group of order $2^k n$ where~$n$ is odd and where $k\ge 1$.
Then $G$ has a Sylow-cyclic subgroup $M$ of order $n$ with the following property:
$G$ is freely representable if and only if (1) $M$ is freely representable
and (2) $G$ has a unique element of order $2$.
In particular, if $G$ is not itself Sylow-cyclic then $G$ is freely representable if and only if $M$ is freely representable.
\end{proposition}

\begin{proof}
This was proved for each type individually. See Proposition~\ref{sylow_cyclic_odd_prop},
Proposition~\ref{quaternion_m_condition_prop}, 
Proposition~\ref{binary_tetrahedral_m_prop}, and Theorem~\ref{binary_octahedral_thm}.
\end{proof}

\chapter{The Non-Solvable Case: the Group $\SL_2(\bF_p)$.}

A very important family of Sylow-cycloidal groups is $\SL_2(\bF_p)$. The first goal here is to show that $\SL_2(\bF_p)$ is Sylow-cycloidal
for all odd primes $p$. In some sense these, together with the Sylow-cycloidal groups we have considered up to now, are all that are needed to form the most general 
Sylow-cycloidal groups. Along the way we see some very striking results about the cyclic (or equivalently the Abelian) subgroups of $\SL_2(\bF_p)$. 
These results leads naturally
to the classification of normal subgroups of $\SL_2(\bF_p)$ and nonsolvability results. We will also consider an amusing necessary condition for 
these groups to be freely-representable: $p$ is a Fermat prime. It turns out that $p = 3$ or $5$ is a necessary and sufficient condition, but we will not
prove this here (at least not in this version of the document). Interestingly,~$p=3$ and~$p=5$ are the two cases that occur as subgroups of $\bH^\times$.

A few of the initial results can be proved for any finite field $F$, but we will need to specialize to $F$ of odd prime order if we want Sylow-cycloidal groups.

\begin{proposition} \label{SL_order_prop}
Let $F$ be a finite field of order $q$. Then $\mathrm{SL}_2 (F)$ is a group 
of order $(q-1)q(q+1)$. If $q$ is odd then there is a unique element of order $2$
in $\mathrm{SL}_2 (F)$.
\end{proposition}

\begin{proof}
The first row of an invertible $2$-by-$2$ matrix is nonzero, so we can limit our attention to~$q^2-1$
candidates for the top row. For each of these candidates, we form an invertible matrix
if and only if we choose the second row not to  be in the span of the first row. This gives $q^2-q$ choices
for the second row for each given top row. So there are 
$$(q^2 -1)(q^2 -q) = (q-1)^2 q (q+1)$$ elements of $\mathrm{GL}_2 (F)$.
The determinant homomorphism $\mathrm{GL}_2 (F) \to F^\times$ is  surjective since even the invertible
diagonal matrices map onto $F^\times$.
So the kernel of this map, which is $\mathrm{SL}_2 (F)$, has order $(q-1) q (q+1)$.

Now we assume $q$ is odd.
Every element $\alpha$ of order 2 in~$\mathrm{SL}_2 (\bF_p)$ has 
has an eigenvalue~$1$ or $-1$ since $(\alpha - I)(\alpha + I) = 0$. Any element of $\mathrm{SL}_2 (\bF_p)$ with
eigenvalue~$\pm 1$ has the following form (with respect to some basis):
$$
\alpha =  \pm \begin{pmatrix}
1 & a \\
0 & 1
\end{pmatrix}, \qquad \alpha^2 = \begin{pmatrix}
1 & 2 a \\
0 & 1
\end{pmatrix}.
$$
We conclude that the only element of order 2 in $\mathrm{SL}_2(\bF_p)$ is $-I$.
\end{proof}

Next we classify elements of order greater than 2 by the number of eigenvalues in $F$, starting with two and working down to zero eigenvalues.

\begin{lemma} \label{SL_2eigen_lemma}
Let $F$ be a finite field of order $q$ and let $\alpha \in \mathrm{SL}_2 (F)$ be an element with two distinct
eigenvalues in $F$. Then $\alpha$ is contained in a 
cyclic subgroup~$C$ of~$\mathrm{SL}_2 (F)$ of order $q-1$ where $C$ has the property that
that there is a basis such that every element of $C$ is a diagonal matrix.
Furthermore, the existence of such an $\alpha$ implies that $q > 3$.
\end{lemma}

\begin{proof}
Fix a basis of eigenvectors for $A$. Then $A$ is contained in the subgroup $C$ consisting of matrices
of the form
$$
\begin{pmatrix}
a & 0\\
0 & a^{-1}
\end{pmatrix}
$$
with respect to the chosen basis, where $a \in F^\times$.
Since $F^\times$ is cyclic, this group is a cyclic group of order $q-1$.

Since the product of the eigenvalues of $\alpha$ is $1$ and since they are distinct, we must have~$q>3$. 
\end{proof}

\begin{lemma} \label{SL_1eigen_lemma}
Let $F$ be a finite field of order $q$ and characteristic~$p$. 
Let $v \in F^2$ be a nonzero vector. Then the set of elements $D_v$ of $\mathrm{SL}_2 (F)$
with exactly one eigenvalue and with eigenvector $v$ is a subgroup of $\mathrm{SL}_2 (F)$
isomorphic to $\{\pm 1\} \times F$ where $F$ is the additive group of $F$ (of size $q$).
In particular $D_v$ is Abelian and each element of $F$ has order divisible by $2p$ if $q$ is odd, and
divisible by $p=2$ if $q$ is even.

Let~$\alpha \in \mathrm{SL}_2 (F)$
be an element with exactly one eigenvalue in $F$.
If $q$ is odd then~$\alpha$ is a contained in
a cyclic subgroup $C$ of~$\mathrm{SL}_2 (F)$ of order $2p$.
If $q$ is even then $\alpha$ is contained in a cyclic subgroup $C$ 
of~$\mathrm{SL}_2 (F)$ of order $p=2$.\end{lemma}

\begin{proof}
Choose a basis whose first element is $v$.
Then $D_v$ consists of the matrices whose representation with respect to this basis is of the form
$$
\pm \begin{pmatrix}
1 & a\\
0 & 1
\end{pmatrix}
$$
for some $a \in F$. 
Observe that 
$$
\begin{pmatrix}
1 & a\\
0 & 1
\end{pmatrix} \begin{pmatrix}
1 & b\\
0 & 1
\end{pmatrix} = \begin{pmatrix}
1 & a+ b\\
0 & 1
\end{pmatrix},
$$
and in fact that $D_v$ is isomorphic to $\{\pm 1\} \times F$.
In particular, every element of $D_v$ is contained in a cyclic subgroup of order $2p$ if $q$ is odd, and a subgroup of order~$p=2$ if $q$ is even.
\end{proof}

\begin{lemma}  \label{SL_0eigen_lemma}
Let $F$ be a finite field of order $q$ and let $\alpha \in \mathrm{SL}_2 (F)$
be an element with no eigenvalues in $F$. Then $\alpha$ is contained in a cyclic subgroup of  $\mathrm{SL}_2 (F)$
of order $q+1$. More specifically, let
$M_2(F)$ be the ring of $2$-by-$2$ matrices with entries in~$F$, where we view $F$ as a subring via the diagonal embedding.
Then the  subring~$E = F[\alpha]$ of $M_2(F)$ generated by $F$ and $\alpha$ is a field.
This field $E$ has the following properties:
\begin{itemize}
\item
$E$ has size $q^2$. 
\item
The group  $K  = \SL_2(F) \cap E^\times$ is a cyclic group of order $q+1$ that contains $\alpha$.
\item
The determinant homomorphism $E^\times \to F^\times$ is the  map~$x \mapsto x^{q+1}$,
and the kernel of this map is $K  = \SL_2(F) \cap E^\times$.
\item
The Galois group of $E$ over $F$ has two elements. Its nontrivial element $\sigma$ is the automorphism $x \mapsto x^q$.
For all $x\in K  = \SL_2(F) \cap E^\times$ we have $\sigma x = x^{-1}$.
\end{itemize}
\end{lemma}

\begin{proof}
Using the ring homomorphism $F[X] \to M_2(F)$ sending $X$ to $\alpha$, 
we see that~$F[\alpha]$ is isomorphic to $F[X]/\left< f\right>$
where $F[X]$ is the polynomial ring in one-variable and $f$ is the minimal polynomial of $\alpha$ in $F[X]$.
By the Cayley-Hamilton theorem,~$f$ divides the characteristic polynomial of $\alpha$, so in this case~$f$ must be an irreducible quadratic polynomial since the characteristic polynomial of $\alpha$ has no roots in $F$. This implies that $F[X]/\left< f\right>$
is a quadratic field extension of~$F$. In particular, $E = F[\alpha]$ is a field of size $q^2$.
Note that $E^\times$ is cyclic of order $q^2 - 1$ (Corollary~\ref{field_cyclic_cor}).

Because $F^\times$ is the subgroup of $E^\times$ of size $q-1$, we have that
that~$\beta\in E^\times$ is in~$F^\times$ if and only if~$\beta^{q-1} = 1$. This implies that $\beta \in E$ is in~$F$ if and only if~$\beta^q = \beta$. 

Now let $\beta$ be a generator of the cyclic group $E^\times$.
In $E[X]$ we have the polynomial
$$
(X-\beta)(X-\beta^q) = X^2 - (\beta+\beta^q) X + \beta^{q+1}
$$
Observe that $$(\beta+\beta^q)^q = \beta^q + \beta^{q^2} = \beta^q + \beta$$ and that $$(\beta^{q+1})^q = \beta^{q^2} \beta^q = \beta^{1+q}.$$
These follow since $\beta^{q^2-1} = 1$, so $\beta^{q^2} = \beta$. Also $q$ is a power of the characteristic~$p$ of $E$ 
and in fields of characteristic $p$ we have the identity
 $(a+b)^p = a^p + b^p$,
 so~$(a+b)^q = a^q + b^q$ for all $a, b \in E$.
We conclude that $X^2 - (\beta+\beta^q) X + \beta^{q+1}$ lies in $F[X]$ and so must be 
the minimal polynomial of $\beta$ in $F[X]$ (since $\beta$ is not in $F$).
By the Cauchy-Hamilton theorem, 
it is the characteristic polynomial of~$\beta$.
In particular the determinant of $\beta$ is $\beta^{q+1}$.
Note also that we have established that~$x \mapsto x^q$ is an automorphism $\sigma$ of the field $E$, and 
that it fixes $F$. 

Consider the determinant homomorphism $E^\times \to F^\times$. This sends 
the generator~$\beta$ to $\beta^{q+1}$,
so sends any element in $E^\times$ to its $q+1$ power.
In particular the kernel~$K$ must be the cyclic subgroup of $E^\times$ of size $q+1$ since $q+1$ divides $q^2-1$.
Note also that~$K = E^\times \cap \SL_2(F)$ since it is in the kernel of the determinant map.
So $\alpha \in K$.

Also note that any automorphism of $E$ fixing $F$ is determined by its action on the generator
$\beta$ of $E^\times$, and that $\beta$ must map to a root of $(X-\beta)(X-\beta^q)$.
Thus there are only two elements of the Galois group of $E$ over $F$ (i.e., the automorphisms of $E$ fixing $F$): the identity 
$x \mapsto x$ and $x \mapsto x^q$. Let $\sigma$ be the map $x\mapsto x^q$. Note that if $x \in K$
then $x^{q+1} = 1$ so $x^q = x^{-1}$. Thus the restriction of $\sigma$ to $K$ is the map $x \mapsto x^{-1}$.
\end{proof}
%

We can combine these three lemmas:

\begin{lemma}  \label{SL_eigen_lemma}
Let $F$ be a finite field of  order $q$ and characteristic $p$. Let~$\alpha \in \mathrm{SL}_2 (F)$
be an element not equal to $1$ or $-1$. Then the following hold:
\begin{itemize}
\item
The element $\alpha$ has two distinct eigenvalues in $F$ if and only if $\alpha$ has order dividing $q-1$.
\item
The element $\alpha$ has exactly one  eigenvalue in $F$ if and only if $\alpha$ has order dividing  $2 p$.
\item
The element $\alpha$ has no  eigenvalues in $F$ if and only if $\alpha$ has order dividing~$q+1$.
\end{itemize}
Furthermore $\alpha$ can only have two distinct eigenvalues if $q>3$.
\end{lemma}

\begin{proof}
One direction of each implication follows directly from Lemmas~\ref{SL_2eigen_lemma}, \ref{SL_1eigen_lemma}, and~\ref{SL_0eigen_lemma}.
To see the converse, observe that the GCD of any two distinct elements of $\{q-1, 2p, q+1\}$ is $2$
is $q$ is odd, and is $1$ is $q$ is even.
Now if $q$ is odd then $-1$ is the only element of order $2$ (Proposition~\ref{SL_order_prop}). Since $\alpha$ is not $1$ or $-1$, we conclude that the order of $\alpha$
cannot divide two of $\{q-1, 2p, q+1\}$.
From this the converse follows.
\end{proof}

A \emph{maximal} cyclic subgroup of a finite group  $G$ is defined to be a cyclic subgroup of $G$ that is not  contained in a cyclic group of $G$
of larger order. Of course every cyclic subgroup is contained in an maximal cyclic group (since $G$ is finite), but a given cyclic subgroup
might be contained in several maximal cyclic groups in a general group $G$. For $G = \SL_2(F)$ we can establish uniqueness.

\begin{corollary}\label{SL_maximalorders_cor}
Let $F$ be a finite field of  order $q$ and characteristic $p$ and let $C$ be a cyclic subgroup of~$\SL_2(F)$.
\begin{itemize}
\item
If $q> 3$ is odd then  $C$ is a maximal cyclic subgroup if and only if it has order~$q-1$, $2p$, or $q+1$.
\item
If $q=3$ then  $C$   is a maximal cyclic subgroup if and only if it has order $2p = 6$ or $q+1 = 4$.
\item
If $q > 2$ is even then $C$ is a maximal cyclic subgroup if and only if it has order~$q-1$, $2$, or $q+1$.
\item
If $q=2$  then $C$ is a maximal cyclic subgroup if and only if it has order~$2$ or~$3$.
\end{itemize}
\end{corollary}

\begin{proof}
Let $\alpha$ be a generator of $C$. 
Since $C$ is a maximal cyclic group, it has the listed orders by Lemmas~\ref{SL_2eigen_lemma},~\ref{SL_1eigen_lemma}, and~\ref{SL_0eigen_lemma}
depending on the number of eigenvalues of $\alpha$ in $F$.

Conversely, suppose $q>3$ is odd and suppose $C$ has order exactly $q-1$,~$2p$, or $q+1$. Let $D$ be a maximal cyclic subgroup containing $C$.
Then $D$ also has \text{order~$q-1$,~$2p$,} or $q+1$ by the above.
Note that (1) the order of $C$ divides the order of $D$, (2) the order of $C$ is at least $3$ (since $q >3$), and (3)  the GCD of any two of $\{q-1, 2p, q+1\}$ is $2$.
So it is impossible for $C$ and $D$ to have different orders. So~$C = D$.
A similar but modified argument works for  $q=3$, or $q>2$ even, or~$q=2$.
\end{proof}

We are now ready for the first main theorem.

\begin{theorem} \label{SL_sylow_cycloidal_thm}
Let $p$ be an odd prime. Then $\SL_2 (\bF_p)$
is a Sylow-cycloidal group whose $2$-Sylow subgroups are not cyclic.
\end{theorem}

\begin{proof}
Let $q$ be an odd prime dividing the order $(p-1) p (p+1)$ of the group~$\mathrm{SL}_2 (\bF_p)$
and let $q^k$ be the maximal power that $q$ divides $(p-1) p (p+1)$.
By Cauchy's theorem, there is an element $\alpha$ of order $q$.
Let $C$ be a maximal cyclic subgroup~$\mathrm{SL}_2 (\bF_p)$ containing $\alpha$.
By the above corollary,  $C$ has order $p-1$ or $2p$ or $p+1$.
Since~$q \ge 3$,  we have that $q$ and hence $q^k$ divides exactly one of $p-1$ or $p$ or $p+1$.
Since $q$ divides $|C|$ this means that if $|C| = p-1$ or $p+1$ then $q^k$ divides $|C|$.
If $|C|=2p$ then $q = p$ (since $q$ is odd) and so $q^k$ divides $|C|$ as well (and $k=1$ in this case).
In any case, $C$ has a cyclic subgroup of order $q^k$.
Since all $q$-Sylow subgroups are conjugate, we have established that all $q$-Sylow subgroups are cyclic.

Since $\SL_2 (\bF_p)$ has a unique element of order $2$ (Proposition~\ref{SL_order_prop}), each $2$-Sylow subgroup of $\SL_2 (\bF_p)$
has a unique element of order $2$ (by Cauchy's theorem). Thus each $2$-Sylow subgroup $S$ of $\mathrm{SL}_2 (\bF_p)$
is either cyclic or quaternionic (Corollary~\ref{unique_2_cor}). 
So $\SL_2 (\bF_p)$ is a Sylow-cycloidal group.

Suppose $S$ is a cyclic $2$-subgroup of $\SL_2(\bF_p)$, and let $C$ be a maximal cyclic subgroup containing $S$. By the above corollary,
$S$ has order dividing $|C|$ which is either $p-1$ or $2p$ or $p-1$. Since $2$ divides both $p-1$ and $p+1$,
the largest power of $2$ diving $|C|$ is less than the largest power of $2$ dividing the order $(p-1)p(p+1)$ of $\SL_2(\bF_p)$.
Thus $S$ is not a 2-Sylow subgroup.
\end{proof}

Here is a partial converse.

\begin{proposition}
Suppose that $F$ is a finite field.
If $\SL_2 (F)$ is a Sylow-cycloidal group then $F = \bF_p$ for some prime $p$.
\end{proposition}

\begin{proof}
By Lemma~\ref{SL_1eigen_lemma} there is a subgroup $D$ of $\SL_2 (F)$ isomorphic to the additive group $F$.
If $\SL_2 (F)$ is a Sylow-cycloidal group then the Abelian group $D\cong F$ must be cyclic.
This can only happen if $F = \bF_p$ where $p$ is the characteristic of $F$.
\end{proof}

\begin{remark}
The case $\SL_2 (\bF_2)$ is special. It has order $6$, and so is Sylow-cyclic since it is of prime free order. As we have seen, all its cyclic subgroups are of
order $2$ or~$3$, so it is not cyclic and so must be dihedral. Note that  $\SL_2 (\bF_2)$ fails to have a unique element of order $2$.
\end{remark}

The next major result is that all cyclic subgroups of $\SL_2(\bF_p)$ of the same order are conjugate when $p$ is a prime.
It is a bit easier to show that such groups are conjugate in $\mathrm{GL}_p (\bF_p)$, but the following two lemmas
will give us tools to achieve conjugacy in $\SL_2(\bF_p)$ instead of~$\mathrm{GL}_p (\bF_p)$.

\begin{lemma}\label{SL_centralize_lemma}
Let $\alpha \in \mathrm{GL}_2 (F)$ where $F$ is a finite field of order $q$ and let $d \in F^\times$.
If $\alpha$ has two distinct eigenvalues in $F$ or no eigenvalues in $F$ then 
there is an element~$\beta \in \mathrm{GL}_2(F)$ such that $\beta \alpha \beta^{-1} = \alpha$
and $\det \beta = d$.
\end{lemma}

\begin{proof}
Suppose $\alpha$ has two eigenvalues in $F$ and let 
$v_1, v_2 \in F^2$ be a basis of eigenvectors. Then let $\beta$ be such that its associated
linear tranformation maps $v_1 \mapsto d v_1$ and $v_2 \mapsto v_2$.
Then $\beta$ clearly works.

Suppose that $\alpha$ has no eigenvalues in $F$. By Lemma~\ref{SL_0eigen_lemma}, $\alpha \in E^\times$
where $E^\times$ is a cyclic subgroup of $\mathrm{GL}_2(F)$ of order $q^2-1$ (in fact $E^\times$ is the multiplicative group of a field);
furthermore, the determinant map on $E^\times \to F^\times$ is given by $x\mapsto x^{q+1}$ and so is surjective since $F^\times$ has order $q-1$.
Now just let $\beta \in E^\times$ be an element of determinant $d$.
\end{proof}

\begin{lemma} \label{SL_normalize_lemma}
Let $p$ be a prime and let
$C$ be a cyclic subgroup of $\SL(\bF_p)$. \text{Let~$d \in \bF_p^\times$.}
Then there is an 
there is an element~$\beta \in \mathrm{GL}_2(F)$ such that $\beta C \beta^{-1} = C$
\text{and~$\det \beta = d$.}
\end{lemma}

\begin{proof}
Let $\alpha$ be a generator of $C$. If $\alpha$ has two distinct eigenvalues or no eigenvalues then the result follows from the
previous lemma. So suppose that $\alpha$ has exactly one eigenvalue in $F$ and let $v_1$ be an eigenvector. Let $v_2$ be such
that $v_1, v_2$ form a basis for $F^2$. The representation for $\alpha$ in this basis is of the form
$$
\begin{pmatrix} e & a \\ 0 & e \end{pmatrix}
$$
where $e$ is $1$ or $-1$. So let $\beta$ be an element with the following matrix representation (for this same basis $v_1, v_2$):
$$
\begin{pmatrix} d & 0 \\ 0 & 1 \end{pmatrix}
$$
Note that
$$
\begin{pmatrix} d & 0 \\ 0 & 1 \end{pmatrix} \begin{pmatrix} e & a \\ 0 & e \end{pmatrix}
\begin{pmatrix} d^{-1} & 0 \\ 0 & 1 \end{pmatrix} = 
\begin{pmatrix} e &d  a \\ 0 & e \end{pmatrix} = \begin{pmatrix} e &  a \\ 0 & e \end{pmatrix}^m
$$
where $m$ is an odd positive integer such that $m\equiv d \pmod p$.
\end{proof}

As mentioned above, our next goal is to establish that cyclic subgroups (equivalently, Abelian subgroups) of the same order of $\SL_2(\bF_p)$ are conjugate.
We start with cyclic groups of order $p-1$, followed by order $p$, then order $p+1$, ending with general order.

\begin{lemma}\label{SL_2eigen_lemma2}
Let $F$ be a finite field of order $q$.
The vector space $F^2$ has $q+1$ distinct one-dimensional subspaces. 
Let $L_1$ and $L_2$ be two distinct one-dimensional subspaces of $F^2$.
Then the set elements of $\SL_2(F)$ with a basis of eigenvectors in~$L_1 \cup L_2$
forms a cyclic subgroup of $\SL_2(F)$ of order $q-1$. 

Now assume $q>3$. Then all cyclic subgroups of 
$\SL_2(F)$ of order~$q-1$ arise in this way. 
There are $\frac{1}{2}q(q+1)$ cyclic subgroups of order $q-1$ and they are conjugate.
Two distinct cyclic subgroups of order $q-1$ have intersection $\{ \pm 1\}$.
\end{lemma}

\begin{proof}
Counting the number of one-dimensional subspaces is straightforward.

Let $v_1, v_2$ be a basis for $F^2$ such that  $v_1 \in L_1$ and $v_2 \in L_2$. Then $\alpha \in \SL_2(F)$ has a basis of eigenvectors in $L_1 \cup L_2$
if and only if it has form
$$
\alpha =  \begin{pmatrix}
a & 0 \\
0 & a^{-1}
\end{pmatrix}
$$
with respect to this basis, where $a \in F^\times$. So the set of such matrices is a cyclic subgroup of order $q-1$.

Assume $q>3$ for the remainder of the proof.

Let $C$ be a cyclic subgroup of $\SL_2(F)$ of order $q-1$, and let $\alpha$ be a generator of~$C$. 
So $\alpha$ is not $1$ or $-1$ since $q-1>2$.
By Lemma~\ref{SL_eigen_lemma},~$\alpha$ has two distinct eigenvalues, and its eigenvectors determine two
distinct one-dimensional subspaces~$L_1$ and $L_2$ of $F^2$. So $\alpha$ is in a cyclic group of order $q-1$ of the form described above.
Hence~$C$ has the desired form.

Suppose $C$ and $C'$ are cyclic subgroup of $\SL_2(F)$ of order $q-1$. Suppose~$C$ is defined using $L_1, L_2$
and $C'$ is define using $L_1', L_2'$.  Suppose $g \in C \cap C'$ is not~$\pm 1$. Then $g$ has distinct eigenvalues (Lemma~\ref{SL_eigen_lemma}).
Let $v_1$ be an eigenvector of $g$. Then~$v_1 \in L_i$ and~$v_1 \in L'_j$ for some $i, j$, so $L_i = L'_j$ since these are one-dimensional.
After renumbering we can assume~$L_1 = L'_1$. Let $v_2$ be an eigenvector of $g$ not in the span of $v_1$. Then $v_2$
is in $L_2$ and $L_2'$ so $L_2 = L'_2$. Thus~$C = C'$. In other words,  distinct cyclic subgroup of $\SL_2(F)$ of order $q-1$
intersect in $\{\pm 1\}$. (Note that $-1 \in C$ and $-1 \in C'$ by Proposition~\ref{SL_order_prop} if $q$ is odd, and trivially if $q$ is even).

So different choices of $\{L_1, L_2\}$ will produce difference cyclic subgroups of order~$q-1$ since $q-1>2$. So there are $\frac{1}{2} q(q+1)$ such cyclic subgroups.

Suppose $C$ and $C'$ are cyclic subgroup of $\SL_2(F)$ of order $q-1$. Suppose $C$ is defined using $L_1, L_2$
and $C'$ is define using $L_1', L_2'$.  
Let $\alpha$ be a generator for $C$.
Let~$\beta_1 \in \mathrm{GL}_2(F)$ give a linear transformation mapping $L_1$ to $L_1'$ and $L_2$ to $L_2'$.
Let~$\beta_2  \in \mathrm{GL}_2(F)$ be such that $\beta_2 \alpha \beta^{-1}_2 = \alpha$ and $\det (\beta_1 \beta_2) = 1$
(see Lemma~\ref{SL_centralize_lemma}). Let $\beta = \beta_1 \beta_2$.
Observe that $C' = \beta C \beta^{-1}$. So~$C$ and $C'$ are conjugate.
\end{proof}

\begin{lemma}\label{SL_1eigen_lemma2}
Let $p$ be an odd prime.
The vector space $\bF_p^2$ has $p+1$ distinct one-dimensional subspaces. 
Let $L$ be a one-dimensional subspaces of $\bF_p^2$.
Then the set elements of $\SL_2(\bF_p)$ with exactly one eigenvalue, and with an eigenvector in $L$,
forms a cyclic subgroup of $\SL_2(\bF_p)$  of order $2p$. 

All cyclic subgroups of 
$\SL_2(\bF_p)$  of order $2p$ arise in this way. 
There are $p+1$ cyclic subgroups of order $2p$ and they are conjugate.
Two distinct cyclic subgroups of order $2p$ have intersection $\{ \pm I\}$.
\end{lemma}

\begin{proof}
Counting the number of one-dimensional subspaces is straightforward.

Let $v_1, v_2$ be a basis for $\bF_p^2$ such that  $v_1 \in L$. Then $\alpha \in \SL_2(\bF_p)$ has  exactly one eigenvalue and has an eigenvector in $L$ if
and only if can be written as
$$
\alpha =  \pm  \begin{pmatrix}
1 & a \\
0 & 1
\end{pmatrix}
$$
with respect to this basis, where $a \in \bF_p$. Observe that the set of such matrices forms a cyclic group isomorphic to
$\{\pm 1\} \times \bF_p$ where $\bF_p$ here is the additive group. 

Let $C$ be a cyclic subgroup of $\SL_2(\bF_p)$ of order $2p$, and let $\alpha$ be a generator of~$C$. 
Then $\alpha$ has exactly one eigenvalue by Lemma~\ref{SL_eigen_lemma}. 
If $L$ is an eigenspace for~$\alpha$, then as above $\alpha$ is contained in a group of order $2p$ of the given form.
So $C$ is a group of the desired form.

Suppose $C$ and $C'$ are cyclic subgroup of $\SL_2(\bF_p)$ of order $2p$ where $C$ is defined using $L$
and  $C'$ is define using $L'$.  Suppose $g \in C \cap C'$ is not~$\pm 1$. Then $g$ has a unique eigenvalue, and the
eigenspace associated to that eigenvalue has dimension one (since $g \ne \pm 1$). Thus $L=L'$ and so $C = C'$.
In other words,  distinct cyclic subgroup of $\SL_2(\bF_p)$ of order $2p$
intersect in $\{\pm 1\}$.

So different choices of $L$ will produce difference cyclic subgroups of order~$2p$. So there are $p+1$ such cyclic subgroups.

Suppose $C$ and $C'$ are cyclic subgroup of $\SL_2(\bF_p)$ of order $2p$. Suppose $C$ is defined using $L$
and $C'$ is define using $L'$. Let $\beta_1 \in \mathrm{GL}_2 (\bF_p)$ represent a linear transformation sending~$L$ to $L'$.
Let~$\beta_2  \in \mathrm{GL}_2(F)$ be such that $\beta_2 C \beta^{-1}_2 = C$ and so that~$\det (\beta_1 \beta_2) = 1$
(see Lemma~\ref{SL_normalize_lemma}). Let $\beta = \beta_1 \beta_2$.
Observe that~$C' = \beta C \beta^{-1}$. So~$C$ and $C'$ are conjugate.
\end{proof}

\begin{lemma}\label{SL_0eigen_lemma2}
Let $F$ be a finite field of order $q$
and let  $M_2(F)$ be the ring of $2$-by-$2$ matrices with entries in $F$.
Then the following hold:
\begin{itemize}
\item
If $C$ is a cyclic subgroup of $\SL_2(F)$ of order $q+1$ then there is a unique quadratic field extension $E$ of $F$ in
$M_2(F)$ containing $C$, and~$C$ is the unique subgroup of $E^\times$
of order $q + 1$. 

\item
Two distinct cyclic subgroups  of $\SL_2(F)$ of order~$q+1$ are conjugate in $\SL(F)$ and have intersection $\{ \pm 1\}$.

\item
If $C$ is a cyclic subgroup of $\SL_2(F)$ of order $q+1$ then
there is a $\gamma\in \SL_2(F)$ such that $x \mapsto \gamma x \gamma^{-1}$ is an automorphism of
the group $C$
sending any $x\in C$ to $x^{-1}$.

\end{itemize}
\end{lemma}

\begin{proof}
Let $\alpha$ be a generator of $C$.
By Lemma~\ref{SL_eigen_lemma}, $\alpha$ cannot have eigenvalues in~$F$.
So by Lemma~\ref{SL_0eigen_lemma}, the subring $E = F[\alpha]$ of of $M_2(F)$ is a field of order $q^2$,
and $\alpha$ is contained in $K  = \SL_2(F) \cap E^\times$, which is the unique subgroup of $E^\times$ of order~$q+1$.
So in fact $C = K$.
Since $E = F[\alpha]$ any quadratic field extension $E$ of~$F$ in $M_2(F)$ containing~$C$ would contain $\alpha$ and so $F[\alpha]$, 
and hence be equal to $E$. So $E$ is the unique such field.

Suppose $C$ and $C'$ are two subgroups of $\SL_2(F)$ size $q+1$ and that $E$ and~$E'$ are the respective quadratic extensions of $F$ containing $C$ and $C'$.
Suppose $\beta \in C\cap C'$ is not~$\pm 1$. Then $\beta$ cannot be of the form $c 1$ for some $c \in F$ since the determinant of~$\beta$ is $1$.
Thus $\beta$ generates a proper extension of $F$ contained in the intersection~$E \cap E'$. So~$E = E'=F[\beta]$
since $[E: F] = [E: F] = 2$ and $[E':F] \ge 2$. We conclude that~$C = C'$.
In other words, if~$C$ and $C'$ are distinct subgroups of $\SL_2(F)$ of size~$q+1$ then their intersection is $\{\pm 1\}$. 

Next fix a generator $g$ of $F^\times$. In other words, $g$ has order $q-1$.
Suppose $C$ is a cyclic subgroup of $\SL_2(F)$ of order~$q+1$ and let $E$ be the field extension of~$F$ inside~$M_2(F)$
of size $q^2$ containing $C$.
Since~$E^\times$ has order $q^2-1$, and since $q+1$ is even, there is a subgroup of $E^\times$ of order $2(q-1)$.
This implies that $E^\times$ has an element $\beta$ such that $\beta^2=g$.
Fix a nonzero vector $v_1 \in F^2$ and let $v_2 = \beta v$ (viewing~$v$ as a column vector).
By the Cayley-Hamilton theorem, the characteristic polynomial of $\beta$ is $X^2 - g$, so $\beta$ has no eigenvalues in $F$.
Thus $v_1, v_2$ must be a basis of $F^2$. Note also that $\beta v_2 = g v_1$.

Suppose $C'$ is another cyclic subgroup of $\SL_2(F)$ of order~$q+1$ and let $E'$ be the field extension of~$F$ inside~$M_2(F)$
of size $q^2$ containing $C'$. Let $\beta' \in (E')^\times$ be such that $v_1' = v_1$ and $v'_2 = \beta' v'_1$ forms
a basis with $\beta' v_2' = g v_1'$.

Let $\gamma_1 \in \mathrm{GL}_2(F)$ be chosen so that $\gamma_1$ maps $v_1 = v_1'$ to itself, and maps $v_2$ to~$v_2'$.
Note that
$
\gamma_1 \beta \gamma_1^{-1} = \beta'
$
since both sides of this equation map $v'_1 \mapsto v'_2$ and~$v'_2 \mapsto g v'_1$.
Since $F[\beta] = E$ and $F[\beta'] = E'$ this means that the map
$$
x \mapsto \gamma_1 x \gamma_1^{-1}
$$
is an isomorphism $E \to E'$ between fields. Hence it sends the unique subgroup $C$ of $E^\times$ of order $q+1$
to the unique subgroup $C'$ of $(E')^\times$ of order $q+1$.

In other words, $\gamma_1 C \gamma_1^{-1} = C'$.
Let $\gamma_2 \in  \mathrm{GL}_2(F)$ be chosen so that $\gamma_1 \gamma_2$ has determinant 1
and so that $\gamma_2 C \gamma_2^{-1} =C$ (Lemma~\ref{SL_centralize_lemma} applied to a generator $\alpha$ of~$C$
and to $d = \det \gamma_1^{-1}$).
If $\gamma = \gamma_1 \gamma_2$ then $\gamma C \gamma^{-1} = C'$ and $\gamma\in\SL_2(F)$.
So $C$ and $C'$ are conjugate in $\SL_2(F)$.

A special case of this construction is where $C = C'$ and where $\beta' = -\beta$
so that~$\gamma_1 \beta \gamma_1^{-1} = - \beta$.
Observe that~$x \mapsto \gamma_1 x \gamma_1^{-1}$ must then be a field automorphism fixing $F$ (thought
of as diagonal matrices). 
Since $F^\times$ has $C$ for its unique subgroup of order $q+1$, this automorphism acts on $C$.
Lemma~\ref{SL_0eigen_lemma} implies that it sends~$x \in C$ to $x^{-1}$.
So if $\gamma = \gamma_1 \gamma_2$ then $\gamma x \gamma^{-1} = x^{-1}$ for all $x \in C$.
\end{proof}

These lemmas combine to yield the following:

\begin{lemma} \label{combine_lemma}
Let $p$ be an odd prime, and 
let $C$ and $C'$ be two distinct maximal cyclic subgroups of~$\SL_2(\bF_p)$.
Then $C \cap C' = \{ \pm 1 \}$.
Furthermore, 
if $C$ and $C'$ have the same order then
$C$ and $C'$ are conjugate subgroups of~$\SL_2(\bF_p)$.
\end{lemma}

\begin{proof}
By Corollary~\ref{SL_maximalorders_cor},
the orders of $C$ and $C'$ are in the set~$\{p-1, 2p, p+1\}$.
If $C$ and $C'$ happen to have different orders then $C\cap C'$ has order dividing $2$
since the GCD of any two of $\{p-1, 2p, p+1\}$ is $2$.
Since $C$ and $C'$ are of even order they both contained $-1$ (since $-1$ is the unique element of order 2).
So $C \cap C' = \{ \pm 1 \}$.

Now we consider the case where
$C$ and $C'$ have the same order in~$\{p-1, 2p, p+1\}$. Then $C$ and $C'$ are conjugate
and $C \cap C' = \{ \pm 1 \}$ by
 Lemma~\ref{SL_2eigen_lemma2}, Lemma~\ref{SL_1eigen_lemma2}, or Lemma~\ref{SL_1eigen_lemma2} (depending
on the order of $C$).
\end{proof}

The above lemma yields
 the conjugacy results we want but only for \emph{maximal} cyclic subgroups of~$\SL_2(\bF_p)$.
The following lemma and proposition make it possible to establish the result for cyclic subgroups more generally.

\begin{lemma}
Let $C$ be a cyclic subgroup of $\SL_2(\bF_p)$ where $p$ is an odd prime.
If the order of $C$ is at least $3$ then
 there is a unique maximal cyclic subgroup of $\SL_2 (\bF_p)$ containing $C$.
\end{lemma}

\begin{proof}
Suppose $D_1$ and $D_2$ are distinct maximal cyclic subgroups of $\SL_2(\bF_p)$ containing $C$.
By the above lemma $D_1 \cap D_2 = \{ \pm 1 \}$, contradicting the assumption on the size of $C$.
\end{proof}

\begin{proposition} \label{SL_zc_c}
Let $C$ be a cyclic subgroup of $\SL_2(\bF_p)$ where $p$ is a odd prime.
If $C$ has order at least $3$ then
the centralizer $Z(C)$ is the unique maximal cyclic subgroup of $\SL_2(\bF_p)$ containing $C$.
So if $C$ is a maximal cyclic subgroup of $\SL_2(\bF_p)$ then $C = Z(C)$.
\end{proposition}

\begin{proof}
Let $D$ be the unique maximal cyclic subgroup of $\SL_2(\bF_p)$ containing $C$ (see previous lemma).
Suppose that $h \in Z(C)$ and let $H$ be the subgroup generated by~$h$.
Observe that $H C$ is an Abelian subgroup of $\SL_2(\bF_p)$.
Since $\SL_2(\bF_p)$ is Sylow-cycloidal (Theorem~\ref{SL_sylow_cycloidal_thm}) the group $HC$ must be cyclic and so is contained
in a maximal cyclic subgroup $D'$. By the previous lemma $D = D'$ since both contain~$C$. Thus $h \in D$.
We have established that $Z(C) \subseteq D$. The other inclusion is clear.
\end{proof}

We are ready for the second main theorem.

\begin{theorem} \label{second_thm}
Let $p$ be an odd prime, and 
let $C$ and $C'$ be two  cyclic subgroups of~$\SL_2(\bF_p)$ of the same order.
Then $C$ and $C'$ are conjugate subgroups of~$\SL_2(\bF_p)$.
If, in addition, $C \ne C'$ then
 $C \cap C' \subseteq \{ \pm 1 \}$.
\end{theorem}

\begin{proof}
If $C$ and $C'$ have order $1$ or $2$ then the result is clear since $-1$ is the unique element of $\SL_2(\bF_p)$
of order $2$. So from now on assume that $C$ and $C'$ have equal order at least $3$.
By the above proposition, $Z(C)$ and $Z(C')$ are maximal cyclic subgroups.
By Corollary~\ref{SL_maximalorders_cor},
the orders of $Z(C)$ and $Z(C')$ are restricted to the set~$\{p-1, 2p, p+1\}$.
So the order of $C$ and $C'$ must divide an element of the set~$\{p-1, 2p, p+1\}$,
and in fact divides a unique element of $\{p-1, 2p, p+1\}$ since the GCD of any two distinct elements
is $2$. This means that $Z(C)$ and $Z(C')$ have the same order, namely the unique element
of $\{p-1, 2p, p+1\}$ that is a multiple of the order of $C$.

Thus $Z(C)$ and $Z(C')$ are conjugate subgroups by Lemma~\ref{combine_lemma}.
Since $C$ is the unique subgroup of $Z(C)$ of its order, and the same is true of $C'$ in $Z(C')$,
we conclude that $C$ and $C'$ are conjugate as well.

If, in addition, $C \ne C'$ then $Z(C)$ cannot equal $Z(C')$ since 
 $C$ is the unique subgroup of $Z(C)$ of its order, and the same is true of $C'$ in $Z(C')$.
So the intersection~$Z(C) \cap Z(C')$ is $\{ \pm 1 \}$ by Lemma~\ref{combine_lemma}.
Thus  $C \cap C' \subseteq \{ \pm 1 \}$.
\end{proof}

Next we wish to count the number of cyclic subgroups of $\SL_2(\bF_p)$ of each order.
We have already counted such groups when the order is $p-1$ or $2p$, so we start with 
the order $p+1$ case.

\begin{lemma} \label{SL_0eigen_lemma3}
Let $p$ be an odd prime.
The number of cyclic subgroup of $\SL_2(\bF_p)$ of order~$p+1$ is equal to $\frac{1}{2} (p-1) p$.
\end{lemma}

\begin{proof}
We consider two approaches. One is to note that each element $g$ of  $\SL_2(\bF_p)$ 
outside of $\{ \pm 1 \}$ is contained in a unique maximal cyclic subgroup of $\SL_2(\bF_p)$,
namely the centralizer $Z(g)$ (see Lemma~\ref{SL_zc_c}). Also a cyclic subgroup of $\SL_2(\bF_p)$ is maximal
if and only if its order is $p-1$ (when $p\ne 3$), $2p$, or $p+1$. 
So we can partition the elements of $\SL_2(\bF_p)$ outside of $\{ \pm 1 \}$: 
if $c_n$ is the number
of cyclic subgroups of order $n$ we have
$$
|\SL_2(\bF_p)| - 2 = c_{p-1} ((p-1)-2) + c_{2p} (2p - 2) + c_{p+1} ((p+1) - 2)
$$
(this works even for $p=3$ since $(p-1)-2 = 0$).
Based on what we know so far
$$
(p-1)p(p+1) - 2 = \frac{1}{2} p (p+1) (p-3) + (p+1) (2p - 2) + c_{p+1} (p-1).
$$
Now we solve for $c_{p+1} $.

This first approach is legitmate, but there is also a more group theoretical / Galois theoretical approach.
Let $C$ be any cyclic group of order $p+1$, and let~$E$ be the field~$F[C]$ inside of $M_2(\bF_p)$
generated by the elements of $C$. Then the normalizer~$N(C)$ acts via conjugation not just on $C$
but on all of $E=F[C]$. In fact we get a homomorphism from~$N(C)$ to the Galois group $G$ of $E$ over $F$,
and kernel of this homomorphism is just $Z(C)$. 
Since $E$ is a quadratic extension of~$\bF_p$, the Galois group $G$ has order 2 (see Lemma~\ref{SL_0eigen_lemma}).
So $Z(C)$ has index $1$ or $2$ in~$N(C)$.
Lemma~\ref{SL_0eigen_lemma2} shows that~$N(C)$ is not $Z(C)$ so the index is actually 2,  and Proposition~\ref{SL_zc_c}
show that $Z(C) = C$. Thus $[N(C): C] = 2$ and so $N(C)$ has order~$2(p+1)$.
By the orbit-stabilizer theorem we have that the number of conjugates of $C$ is
$$
\frac{(p-1) p (p+1)}{2(p+1)} = \frac{1}{2} (p-1) p
$$
which counts the number of cyclic subgroups of order $p+1$ since all such groups are conjugate (Theorem~\ref{second_thm}).
\end{proof}

Our calculations culminate with a full census of the cyclic subgroups of $\mathrm{SL}_2 (\bF_p)$. Since
$\mathrm{SL}_2 (\bF_p)$ is Sylow-cyloidal, this gives a full census of Abelian subgroups as well.

\begin{proposition}\label{SL_orders_prop}
Let $p$ be an odd prime and
let $c_m$ be the number of cyclic subgroups of $\mathrm{SL}_2 (\bF_p)$ of order $m$.
Then
$$
c_1 = c_2 = 1.
$$
If $m > 2$ divides $p-1$ then 
$$c_m = \frac{1}{2} p(p+1).
$$
If $m>2$ divides $2q$ then
$$c_m =p + 1.
$$
If $m>2$ divides $p+1$ then
$$
c_m = \frac{1}{2} p(p-1).
$$
Otherwise $c_m = 0$.
\end{proposition}

\begin{proof}
Since there is a unique element of order 2 (Proposition~\ref{SL_order_prop}) and order 1, we have $c_1 = c_2 = 2$.
So from now on we assume $m > 2$. 

Suppose that $m$ divides~$p-1$ (so $p>3$ since $m>2$). Each cyclic subgroup~$C$ of order $p-1$
has $Z(C)$ for the maximum cyclic subgroup containing $C$ (Proposition~\ref{SL_zc_c}). By Corollary~\ref{SL_maximalorders_cor}, 
$Z(C)$ has order in the set $\{ p-1, 2p, p+1 \}$, and $m$ divides the order of~$Z(C)$ since $C$ is a subgroup of $Z(C)$.
Since $m$ divides $p-1$, it cannot divide $2p$ or $p+1$ as well (since the GCD of $p-1$ with each is $2$ and~$m>2$).
Thus $Z(C)$ is a cyclic group of order $p-1$. Conversely each cyclic group of order~$p-1$ has a unique subgroup of order $m$.
Thus there is a one-to-one correspondences between cyclic subgroups of order $m$ and subgroups of order $p-1$.
By  Lemma~\ref{SL_2eigen_lemma2} there are $\frac{1}{2}p(p+1)$ cyclic subgroups of order $p-1$.
Thus $c_m = \frac{1}{2}p(p+1)$.

The case where $m$ divides $2p$ is similar:  there is a one-to-one correspondence~$C \mapsto Z(C)$ between
cyclic subgroups of order $m$ and cyclic subgroups of order $2p$. We can use Lemma~\ref{SL_1eigen_lemma2}
to conclude that there are $c_m= p+1$ cyclic subgroups of order $m$.

The case where $m$ divides $p+1$ is also similar:  there is a one-to-one correspondence~$C \mapsto Z(C)$ between
cyclic subgroups of order $m$ and cyclic subgroups of order~$p+1$. We can use Lemma~\ref{SL_0eigen_lemma3}
to conclude that there are $c_m= \frac{1}{2} (p-1) p$ cyclic subgroups of order $m$.
\end{proof}

\begin{corollary}
Suppose that $p$ is an odd prime.
The only normal cyclic subgroups of  $\mathrm{SL}_2 (\bF_p)$ are $\{1\}$ and $\{\pm 1\}$.
\end{corollary}

\begin{proof}
Suppose that $C$ is a normal subgroup of order $m$. 
By Theorem~\ref{second_thm} all cyclic subgroups of order $m$ are conjugate. Since $C$
is normal, this means that there is only one subgroup of order $C$. According to the previous proposition,
this can only happen if $m =1,2$. The result follows now from the fact that $-1$ is the only element of order two (Proposition~\ref{SL_order_prop}).
\end{proof}

\begin{corollary}
Suppose that $p$ is an odd prime.
Then the only normal  Sylow-cyclic subgroups
of  $\mathrm{SL}_2 (\bF_p)$ are $\{1\}$ and $\{\pm 1\}$.
\end{corollary}

\begin{proof}
Suppose that $N$ is a normal  Sylow-cyclic subgroup of $\mathrm{SL}_2 (\bF_p)$.
By Theorem~\ref{semidirect_thm} the commutator subgroup $N'$ of $N$ is cyclic of odd order.
Since $N'$ is characteristic in $N$ it must be normal in $\mathrm{SL}_2 (\bF_p)$.
By the above corollary, $N'$ is the trivial group. So $N$ is Abelian, hence cyclic since $N$ is Sylow-cyclic.
So by the above corollary $N$ is $\{1\}$ or $\{\pm 1\}$.
\end{proof}

\begin{lemma}
Suppose that $p$ is an odd prime
and that $N$ is a normal subgroup of~$\mathrm{SL}_2 (\bF_p)$ 
not equal to $\{1\}$ or $\{\pm 1\}$.
Then $N$ has odd index in~$\mathrm{SL}_2 (\bF_p)$.
\end{lemma}

\begin{proof}
Let $N_2$ be a $2$-Sylow subgroup of $N$. By the above corollary, $N_2$ cannot be cyclic,
so $N_2$ must be a general quaternion group. In particular, $N$ contains a cyclic subgroup of order $4$.
All cyclic subgroups of order $4$ are conjugate in $\mathrm{SL}_2 (\bF_p)$
by Theorem~\ref{second_thm}. Since $N$ is normal, it must contain all cyclic subgroups of order $4$.
Let $P_2$ be a $2$-Sylow subgroup of  $\mathrm{SL}_2 (\bF_p)$. Note that $P_2$ is generated
by its cyclic subgroups of order $4$. So $P_2$ must be contained in $N$. The result follows.
\end{proof}

\begin{lemma}
Let  $p$ be an odd prime.
Suppose~$\mathrm{SL}_2 (\bF_p)$ 
has a normal subgroup not equal to $\{1\}$ or $\{\pm 1\}$ or $\mathrm{SL}_2 (\bF_p)$.
Then $\mathrm{SL}_2 (\bF_p)$ has a  normal subgroup of odd prime index.
\end{lemma}

\begin{proof}
Let $N$ be a normal subgroup not equal to $\{1\}$ or $\{\pm 1\}$ or $\mathrm{SL}_2 (\bF_p)$.
Then $N$ has odd index in $\mathrm{SL}_2 (\bF_p)$ by the previous lemma.
Observe that $G = \mathrm{SL}_2 (\bF_p) / N$ is a Sylow-cyclic group of odd order.
By Corollary~\ref{cyclic_quotient_cor}, $G$ has a nontrivial cyclic quotient.
Hence $G$ has a quotient of odd prime order. This means
that $\mathrm{SL}_2 (\bF_p)$ has a quotient of odd prime order as well, and  the result follows.
\end{proof}

Now we are ready for the third main result:

\begin{theorem}\label{SL_normal_thm}
Let $p\ge 5$ be an odd prime. Then the normal subgroups of ~$\mathrm{SL}_2 (\bF_p)$ 
are $\{1\}$,  $\{\pm 1\}$, and  $\mathrm{SL}_2 (\bF_p)$ itself.
\end{theorem}

\begin{proof}
Suppose there are normal subgroups that differ from $\{1\}$,  $\{\pm 1\}$, and  $\mathrm{SL}_2 (\bF_p)$ itself.
By the previous lemma, there is a normal subgroup $N$ such that $N$ has prime index $q$ in $\mathrm{SL}_2 (\bF_p)$.
Since the order of $\mathrm{SL}_2 (\bF_p)$ is $(p-1) p (p+1)$ we have that $q$ divides an element of the set $\{ p-1, p, p+1 \}$. It 
cannot divide two elements of this set since $q \ge 5$. So there are three cases based on which element is divisible by~$q$.

The next step is to partition $\mathrm{SL}_2 (\bF_p)$ as follows:
\begin{itemize}
\item
Let $\Gamma_0 = \{\pm 1 \}$.
\item 
Let $\Gamma_1$ be all elements $g$ of order $m\ge 3$ such that the centralizer~$Z(\left<g\right>)$ of~$\left<g\right>$ has order $p-1$.
\item
Let $\Gamma_2$ be all elements $g$ of order $m\ge 3$ such that the centralizer~$Z(\left<g\right>)$ of~$\left<g\right>$ has order $2p$.
\item
Let $\Gamma_3$ be all elements $g$ of order $m \ge 3$ such that the centralizer~$Z(\left<g\right>)$ of~$\left<g\right>$ has order $p+1$.
\end{itemize}
These sets are clearly disjoint. They form
 partition of~$\mathrm{SL}_2 (\bF_p)$ since every element outside $\Gamma_0$ has order at least $3$ (Proposition~\ref{SL_order_prop}), 
 and~$Z(\left<g\right>)$ has order in the set~$\{ p-1, 2 p , p+1\}$ (see Proposition~\ref{SL_zc_c} and Corollary~\ref{SL_maximalorders_cor}).
 
Next we further partition each $\Gamma_i$. 
We start with $\Gamma_1$.
Every element $g \in \Gamma_1$ is in exactly one 
cyclic subgroup of order $p-1$, namely $Z(\left< g \right>)$ (it cannot be in a second cyclic group of order $p-1$ by Theorem~\ref{second_thm}).
So $\Gamma_1$ can be partitioned into sets of the form $C \cap \Gamma_1$ where $C$ is a cyclic subgroup of order $p-1$.
What is the size of $C \cap \Gamma_1$? By
Corollary~\ref{SL_maximalorders_cor}, if $C$ is a cyclic subgroup of order $p-1$ then it is a maximal cyclic group, and so $C$ equal 
to $Z(\left< g \right>)$ for all $g \in C$ of order $m\ge 3$ (Proposition~\ref{SL_zc_c}). The only elements of $C$ not in $C \cap \Gamma_1$
are the elements of order~$1$ or $2$, namely the elements $\pm 1$. So each partition $C \cap \Gamma_1$ has size $(p-1)-2 = p-3$.
From Proposition~\ref{SL_orders_prop} we know that the number of partitions is $c_{p-1} = \frac{1}{2} p (p+1)$. So $\Gamma_1$
has $c_{p-1} (p-3) = \frac{1}{2} p (p+1) (p-3)$ elements.
A similar calculation can be made for $\Gamma_2$ and $\Gamma_3$.

In fact, we can list the proportion of elements of~$\mathrm{SL}_2 (\bF_p)$ in $\Gamma_1, \Gamma_2, \Gamma_3$
as follows:
\begin{itemize}
\item
The proportion of elements of $\mathrm{SL}_2 (\bF_p)$ in $\Gamma_1$ is 
$$
\frac{\frac{1}{2} p (p+1) (p -1 -2)}{(p-1)p (p+1) } = \frac{1}{2} \frac {p-3} {p-1} < \frac{1}{2}.
$$
\item
The proportion of elements of $\mathrm{SL}_2 (\bF_p)$ in $\Gamma_2$ is 
$$
\frac{(p+1) (2 p -2)}{(p-1)p (p+1) } =\frac {2} {p}.
$$
\item
The proportion of elements of $\mathrm{SL}_2 (\bF_p)$ in $\Gamma_3$ is 
$$
\frac{\frac{1}{2} (p-1) p (p + 1 -2)}{(p-1)p (p+1) } = \frac{1}{2} \frac {p-1} {p+1} < \frac{1}{2}.
$$
\end{itemize}

Claim: \emph{the complement of $N$ is contained in $\Gamma_i$ for some $i \in \{1, 2, 3\}$.}
For instance, suppose $q$ divides $p-1$. Then clearly $\Gamma_0 \subseteq N$.
Note that $q$ cannot divide the order of any $g \in \Gamma_2$ since $q$ does not divide $2p$.
So each $g \in \Gamma_2$ must have trivial image under $\mathrm{SL}_2 (\bF_p) \to \mathrm{SL}_2 (\bF_p) / N$.
In other words, $\Gamma_2 \subseteq N$. 
Similarly $q$ cannot divide the order of any $g \in \Gamma_3$ since $q$ does not divide $p+1$.
So each $g \in \Gamma_3$ must have trivial image under $\mathrm{SL}_2 (\bF_p) \to \mathrm{SL}_2 (\bF_p) / N$.
In other words, $\Gamma_3 \subseteq N$. 
So the complement of $N$ in $\mathrm{SL}_2 (\bF_p)$ must be contained in $\Gamma_1$.
The other cases are similar. In fact, if $N^c$ is the complement $\mathrm{SL}_2 (\bF_p) - N$ then
\begin{itemize}
\item
If $q$ divides $p-1$ then $N^c \subseteq \Gamma_1$.
\item
If $q$ divides $p$ then $N^c \subseteq \Gamma_2$.
\item
If $q$ divides $p+1$ then $N^c \subseteq \Gamma_3$.
\end{itemize}

This leads to a contradiction when $p\ge 5$. Since $N$ is a proper subgroup of~$\mathrm{SL}_2 (\bF_p)$,
the complement $N^c$ must contain at least one-half of the elements of~$\SL_2 (\bF_p)$. But the formulas for 
proportion of the elements of $\Gamma_1, \Gamma_2, \Gamma_3$ in $\SL_2 (\bF_p)$
shows that these proportions are each strictly smaller than one-half. So no such normal subgroup~$N$ of index $q$ exists.
\end{proof}

\begin{corollary}\label{linear_non_solvable_cor}
Let $p\ge 5$ be an odd prime. Then~$\SL_2 (\bF_p)$ is a non-solvable Sylow-cycloidal group,
and the quotient
$$
\mathrm{PSL}_2 (\bF_p) \; \defeq \; \SL_2 (\bF_p) / \{ \pm 1\}
$$
is a simple group.
\end{corollary}

\begin{remark}
The simple groups $\mathrm{PSL}_2 (\bF_p)$
were studied even by Galois, and were the earliest known non-Abelian simple groups outside of the  alternating groups.
\end{remark}

The above takes care of the case $p\ge 5$. The case $p=3$ is of interest as well:

\begin{proposition}
The group $\SL_2 (\bF_3)$ is a solvable Sylow-cycloidal group isomorphic to the binary tetrahedral group $2T$.
\end{proposition}

\begin{proof}
The order of $\SL_2 (\bF_3)$ is $2 \cdot 3 \cdot 4 = 24$ (Proposition~\ref{SL_order_prop}).
By Theorem~\ref{SL_sylow_cycloidal_thm},~$\SL_2 (\bF_3)$ is a Sylow-cycloidal group whose
$2$-Sylow subgroups are not cyclic. Thus the $2$-Sylow subgroups of $\SL_2 (\bF_3)$ must
be isomorphic to the quaternion group~$Q_8$ with $8$ elements.

According to Proposition~\ref{SL_orders_prop}, there are $3$ cyclic subgroups of $\SL_2 (\bF_3)$ of order~$4$ which is exactly
the number of such cyclic subgroups of $Q_8$. Thus there is a unique~$2$-Sylow subgroup of $\SL_2 (\bF_3)$, and so it is normal.
This explains why $\SL_2 (\bF_3)$ is solvable.
Proposition~\ref{SL_orders_prop} also asserts that there are $4$  cyclic subgroups of $\SL_2 (\bF_3)$ of order~$3$,
and these subgroups are conjugate (Theorem~\ref{second_thm}). Thus $\SL_2 (\bF_3)$ has no normal subgroup
of order 3. 
By Proposition~\ref{2T_iso_prop}, $\SL_2 (\bF_3)$ is isomorphic to the binary tetrahedral group $2T$.
\end{proof}

We have considered cyclic subgroups and normal subgroups of $\SL_2(\bF_p)$.
Now we will consider subgroups of odd order. We will do so using the following result on normalizers of cyclic subgroups.

\begin{lemma}
Let $C$ be a cyclic subgroup of $\SL_2 (\bF_p)$ where $p$ is an odd prime.
Then the normalizer $N(C)$ of $C$ is the normalizer $N(Z(C))$ of the maximal cyclic subgroup $Z(C)$ containing $C$.
Furthermore,
\begin{itemize}
\item
If $Z(C)$ has order $p-1$ then $N(Z(C))$ is a non-Abelian subgroup of $\SL_2 (\bF_p)$ of order~$2(p-1)$.
\item
If $Z(C)$ has order $2p$ then $N(Z(C))$ is a non-Abelian subgroup of $\SL_2 (\bF_p)$ of order~$(p-1)p$.
\item
If $Z(C)$ has order $p+1$ then $N(Z(C))$ is a non-Abelian subgroup of $\SL_2 (\bF_p)$ of order~$2 (p+1))$.
\end{itemize}
\end{lemma}

\begin{proof}
Since $Z(C)$ is cyclic, $C$ is a characteristic subgroup of $Z(C)$. Thus $C$ is normal in $N(Z(C))$ since $Z(C)$ is normal in $N(Z(C))$.
Hence
$$
N(Z(C)) \subseteq N(C).
$$
We show equality by showing that both subgroups in this inclusion have the same order.

We focus on the case where $Z(C)$ has order $p-1$. The other cases are similar.
By Theorem~\ref{second_thm} the group $\SL_2 (\bF_p)$ acts transitively on the collection
of cyclic subgroups of order~$|C|$. 
By Proposition~\ref{SL_orders_prop} the orbit of $C$ under this action has size~$\frac{1}{2}p(p+1)$.
So by the orbit-stabilizer theorem, the stabilizer $N(C)$ under conjugation has the following size:
$$
|N(C)|  = \frac{(p-1)p(p+1)} {\frac{1}{2}p(p+1)} = 2(p-1).
$$
This calculation is valid for $N(Z(C))$ as well since $Z(C)$ is cyclic of size $p+1$. So~$N(Z(C))$ has
order $2(p-1)$ as well. Equality follows.
\end{proof}

\begin{corollary} \label{SL_odd_cor}
Suppose $H$ is a noncyclic subgroup of $\SL_2(\bF_p)$ of odd order, where~$p$ is an odd prime.
Then $H$ has order dividing $(p-1)p$. Furthermore $H$ contains a normal subgroup $C$ of order $p$, 
and $H$ is subgroup of the normalizer~$N(C)$ of $C$ in~$\SL_2(\bF_p)$.
\end{corollary}

\begin{proof}
Since $\SL_2(\bF_p)$ is a Sylow-cycloidal group, the subgroup $H$ of odd order is a Sylow-cyclic group.
By Theorem~\ref{semidirect_thm}, $H$ has a nontrivial normal cyclic subgroup $C$, and so~$H \subseteq N(C)$.
If $C$ has order dividing $p-1$ or $p-2$ then $N(C)/Z(C)$ is even, and the image of $H$ in $N(C)/Z(C)$ is trivial.
Thus $H \subseteq Z(C)$ and so $H$ is cyclic.
So we are left with the case where $C$ has order dividing $2p$. In other words, $C$ has order $p$.
By the above proposition $N(C)$ has order $(p-1)p$.
\end{proof}

Given the above corollary, the normalizer $N(C)$ of a cyclic subgroup $C$ of order~$p$ warrants our further attention.

\begin{lemma}
Let $C_p$ be a subgroup of $\SL_2(\bF_p)$ of order $p$, where~$p$ is an odd prime.
Then $C_p$ is conjugate to the subgroup
$$
\left\{ \begin{pmatrix} 1 & b \\ 0 & 1  \end{pmatrix}     \;   \middle| \; b \in \bF_p \right\} \cong \bF_p.
$$
The normalizer $N(C_p)$ is conjugate to the subgroup
$$
\left\{ \begin{pmatrix} a & b \\ 0 & a^{-1}  \end{pmatrix}     \;   \middle| \; a \in \bF^\times_p, \; b \in \bF_p \right\}.
$$
Furthermore, $N(C_p)$ is a semidirect product $C_p \rtimes C_{p-1}$ where $C_{p-1}$ is the cyclic group
$$
\left\{ \begin{pmatrix} a & 0 \\ 0 & a^{-1}  \end{pmatrix}     \;   \middle| \; a \in \bF^\times_p \right\} \cong \bF_p^\times.
$$
Viewing $C_{p-1}$ as $\bF_p^\times$ and $C_p$ as $\bF_p$, the action of $a\in C_{p-1}$ on $C_p$ is given by
$$
b \mapsto a^2 b.
$$
\end{lemma}

\begin{proof}
Let $\alpha$ be a generator of $C_p$. By Lemma~\ref{SL_eigen_lemma}, $\alpha$ has exactly one eigenvalue in $\bF_p$. 
If we work in a basis $v_1, v_2$ where $v_1$ is an eigenvector of $\alpha$, then $\alpha$ has the form 
$$
\begin{pmatrix} 1 & b \\ 0 & 1  \end{pmatrix}
$$
for some nonzero $b \in  \bF_p$. It follows that $C_p$ has the desired form (up to conjugation).
From now on we assume $C_p$ is this particular subgroup of $\SL_2(\bF_p)$.

Next consider the subgroup
$$
H \; \defeq \; \left\{ \begin{pmatrix} a & b \\ 0 & a^{-1}  \end{pmatrix}     \;   \middle| \; a \in \bF^\times_p, \; b \in \bF_p \right\}.
$$
Observe that
$$
\begin{pmatrix} a & b \\ 0 & a^{-1}  \end{pmatrix}   \mapsto a
$$
is a surjective homomorphism $H \to \bF_{p}^\times$ with kernel $C_p$.
So $C_p$ is a normal subgroup of $H$, and so $H\subseteq N(C_p)$.
By the above lemma $N(C_p)$ has size $(p-1) p$, which is also the size of $H$.
So $H =  N(C_p)$.
Note that $C_{p-1}$ (as defined in the statement of the lemma) is a complement to $C_p$ in $H = N(C_p)$
and so $N(C_p)$ is a semidirect product $C_p \rtimes C_{p-1}$. The last assertion follows from the calculation:
$$
\begin{pmatrix} a & 0 \\ 0 & a^{-1}  \end{pmatrix}  \begin{pmatrix} 1 & b \\ 0 & 1  \end{pmatrix}  \begin{pmatrix} a & 0 \\ 0 & a^{-1}  \end{pmatrix}^{-1}  =
\begin{pmatrix} a & 0 \\ 0 & a^{-1}  \end{pmatrix}  \begin{pmatrix} 1 & b \\ 0 & 1  \end{pmatrix}  \begin{pmatrix} a^{-1} & 0 \\ 0 & a  \end{pmatrix}
=
 \begin{pmatrix} 1 & b \\ 0 & 1  \end{pmatrix} 
$$
\end{proof}

Here is the next major result: 

\begin{theorem}
Let $H$ be a noncyclic subgroup of $\SL_2(\bF_p)$ of odd order, where~$p$ is an odd prime.
Then $H$ contains a cyclic subgroup of order $p$ and $H$ is contained in the normalizer $N(C)$ of $C$ in $\SL_2 (\bF_p)$.
Each such $H$ is isomorphic to $\bF_p \rtimes  T$ where~$T$ is a nontrivial subgroup of $\bF_p^\times$ of odd order.
Here the action of $a \in T$ is given by $b \mapsto a^2 b$. 

Furthermore, given such a semidirect product, $\bF_p \rtimes  T$ there is a non-Abelian subgroup of~$\SL_2(\bF_p)$ isomorphic to  $\bF_p \rtimes  T$.
\end{theorem}

\begin{proof}
The existence of $C$ follows from Corollary~\ref{SL_odd_cor}.
By the above lemma, any such $H$ is isomorphic to a subgroup of $\bF_p \rtimes \bF_{p}^\times$ where $\bF_p$ corresponds to $C$.
Observe that $H$ is actually isomorphism to $\bF_p \rtimes T$ where
 $T$ is the image of $H$ in $\bF_p^\times$. Note that $T$ is nontrivial since $H$ is noncyclic.
 
 Conversely, given a nontrivial odd order subgroup $T$ of $\bF_p^\times$, we have a subgroup of $N(C)$ isomorphic $\bF_p \rtimes T$
 where here $C$ is any subgroup of $\SL_2(\bF_p)$ of order $p$. (This follows from the above lemma).
\end{proof}

This leads to some important corollaries:

\begin{corollary}  \label{SL_fermat_cor2}
Suppose $p$ is an odd prime. Then $\SL_2(\bF_p)$ contains a noncyclic subgroup of order the product of two primes
if and only if $p$ is not a Fermat prime.
\end{corollary}

\begin{proof}
First suppose $\SL_2(\bF_p)$ has a noncyclic subgroup $H$ of order $|H| = q_1 q_2$ where $q_1 \le  q_2$ are primes. 
Since $\SL_2(\bF_p)$ is a Sylow-cycloidal group, any subgroup of prime squared order is cyclic. Thus $q_1 \ne q_2$.

First suppose that $q_1 = 2$.  By Theorem~\ref{semidirect_thm} we have that $H$ is isomorphic to
a semidirect product~$C_{q} \rtimes C_2$
where $C_q$ and $C_2$ are cyclic subgroups of order $q = q_2$ and $2$ respectively.
The only nontrivial action of $C_2$ is the one where the nontrivial element of $C_2$ acts on $C_q$ by $x\mapsto x^{-1}$.
Thus $H$ is a dihedral group of order $2q$, and so contains $q$ elements of order 2. However,  $\SL_2(\bF_p)$ has a unique 
element of order~$2$. Thus we can assume that $q_1$ and $q_2$ are odd.

By the above theorem $H$ is of order $p r$ where $r$ is an odd prime dividing $p-1$. This $p$ is not a Fermat prime.

Conversely, suppose $p$ is not a Fermat prime, and let $r$ be an odd prime dividing~$p-1$. Let $T$ be a cyclic subgroup of $\bF_p^\times$
of order $r$.  Then $\SL_2(\bF_p)$ contains a non-Abelian subgroup of order $pr$  (isomorphic to a semidirect product $\bF_p \rtimes T$) by the above theorem.
\end{proof}

\begin{corollary} \label{SL_fermat_cor3}
Suppose $p$ is an odd prime. 
If $p$ is not a Fermat prime then $\SL_2(\bF_p)$ is not freely representable.
\end{corollary}

\begin{proof}
Use the previous corollary, Theorem~\ref{pq_thm}, and the fact that every subgroup of a freely representable group must be freely representable.
\end{proof}

There are similar limits of subgroups of $\SL_2 (\bF_p)$ even of even order, but we will not cover this here.\footnote{I hope to include
this in a sequel, or future version of this document. We will use previous results
on solvable Sylow-cycloidal groups. For the non-solvable case we will need Suzuki's theorem.}
%
%
%
%
%
%
%
%
%
%
%
%
%
%
%
%

\chapter{Classification of Non-Solvable Sylow-Cycloidal Groups}

We take as given a substantial result of Suzuki from 1955 \cite{suzuki1955}.

\begin{theorem} \label{suzuki_thm}
Suppose $G$ is a non-solvable Sylow-cycloidal Group. 
Then~$G$ has a subgroup isomorphic to $\mathrm{SL}_2 (\bF_p)$ for some prime $p \ge 5$.
In fact, $G$ has a subgroup~$H$ of index $1$ or $2$ such that $H$
is isomorphic to $ \mathrm{SL}_2 (\bF_p) \times M$
where $M$ is a Sylow-cyclic subgroup of $G$ of order prime to $(p-1)p(p+1)$.
\end{theorem}

\begin{remark}
Note that $2T \cong \mathrm{SL}_2 (\bF_3)$ played an important role in the above classification
of solvable Sylow-cycloidal groups, so it is interesting that~$\mathrm{SL}_2 (\bF_p)$ with $p\ge 5$
plays a central role in the non-solvable case.
\end{remark}

\begin{remark}
The group $\mathrm{SL}_2 (\bF_5)$ can be shown to be isomorphic to the binary icosahedral subgroup of $\bH^\times$, so is 
freely representable.\footnote{I hope to cover this in a sequel, or in a future edition of this document.}
\end{remark}

Suzuki's theorem allows us to focus our attention to $\mathrm{SL}_2 (\bF_p)$:

\begin{proposition} \label{nonsolvable2_prop}
Suppose $G$ is a non-solvable Sylow-cycloidal group. Let~$H$, $M$, and $\mathrm{SL}_2 (\bF_p)$ be as in
Suzuki's theorem. 
Then $G$ has a unique element of order~$2$. Moreover,
$G$ is freely representable if and only if (1) $M$ and $\mathrm{SL}_2 (\bF_p)$
are freely representable.
\end{proposition}

\begin{proof}
Since $M$ is of odd order and $\mathrm{SL}_2 (\bF_p)$
has a unique subgroup of order $2$ (Proposition~\ref{SL_order_prop}), the group $H \cong M \times \mathrm{SL}_2 (\bF_p)$ has a unique element of order $2$.
So if $G = H$ then we are done. In general, $H$ is normal in $G$, and $H$ has a characteristic subgroup $C_2$
of order $2$. Thus $C_2$ must be normal in $G$. By Lemma~\ref{order2_lemma}, $G$ has a unique element of order 2.

Next suppose $G$ is freely representable. Then $M$ and $\mathrm{SL}_2 (\bF_p)$ must be freely representable
since they are isomorphic to subgroups of $G$.

Finally suppose that $M$ and $\mathrm{SL}_2 (\bF_p)$ are freely representable.
Then the product~$H \cong M \times \mathrm{SL}_2 (\bF_p)$ is freely representable by Corollary~\ref{rel_prime_cor}.
Since $G/H$ has order~$1$ or $2$, any subgroup of odd prime order is in the kernel of $G \to G/H$
and so is a subgroup of $H$. Since $G$ and $H$ have a unique subgroup of order 2, we conclude
that every subgroup of $G$ of prime order is a subgroup of $H$.
 Since $H$ is freely representable, the same is true of~$G$ 
 by the technique of induced representations (Proposition~\ref{sufficient_prop}).
\end{proof}

\begin{definition}
A finite group $G$ is \emph{perfect} if its commutator subgroup $G'$ is all of~$G$. In other words,
if $G$ has no nontrivial Abelian quotient.
\end{definition}

\begin{theorem}
If $p \ge 5$ then $\SL_2 (\bF_p)$ is a perfect group, and is a Sylow-cycloidal group.
\end{theorem}

\begin{proof}
The fact that $G = \SL_2 (\bF_p)$ is a perfect group  follows from Theorem~\ref{SL_normal_thm} and its corollary.
In fact, by these earlier results,~$G' = \{1\}$ or~$G'=\{\pm 1\}$ or $G' = G$.
However, $\SL_2 (\bF_p) / \{\pm 1\}$ is of non-prime order and is simple, so cannot be Abelian.
This means that $G'$ cannot be  $\{\pm 1\}$ or $\{1\}$. Thus~$G = G'$.

The group $\SL_2 (\bF_p)$ is Sylow-cycloidal by~Theorem~\ref{SL_sylow_cycloidal_thm}.
\end{proof}

\begin{lemma}\label{lemma160}
The group $\mathrm{SL}_2 (\bF_5)$ is freely representable.
\end{lemma}

\begin{proof}
Recall that the binary icosahedral group ${2 I}$ is the preimage of the icosahedral
group under the standard two-to-one map $\bH_1 \to \mathrm{SO}(3)$. 
Since ${2 I}$ is a subgroup of~$\bH_1$, the group $2 I$ is freely representable.
Finally ${2 I}$, as mentioned above, $2 I$ is isomorphic to $\mathrm{SL}_2 (\bF_5)$.
\end{proof}

Here is a theorem that Wolf attributes to Zassenhaus (1935). See Wolf~\cite{wolf2011}, Section 6.2 for a proof.
The proof is long: running to 14 pages, and I have not studied it yet.
Recently shorter proofs have appeared (which I also need to study): see
Mazurov~\cite{mazurov2003} and Allcock~\cite{allcock2018}.

\begin{theorem} [Zassenhaus]
Suppose $G$ is a perfect freely representable group that is not trivial. 
Then $G$ is isomorphic to $\mathrm{SL}_2 (\bF_5)$.
\end{theorem}

\begin{remark}
This gives another argument that a simple group is freely representable if and only if it is cyclic.
\end{remark}

\begin{corollary}\label{linear_not_freely_representable_cor}
If $p > 5$ then $\mathrm{SL}_2 (\bF_p)$ is a non-solvable Sylow-cycloidal group that
is not freely representable.
\end{corollary}

If we combine Zassenhaus and Suzuki's results we get the following:

\begin{theorem}\label{zassenhaus_sukuki_thm}
Let $G$ be a non-solvable group with a unique element of order 2.
Then $G$ is freely representable if and only $G$ has a subgroup $H$ of index $1$ or $2$
and a freely representable subgroup~$M$ of order prime to $30$ such that $H$ is isomorphic to
$\mathrm{SL}_2 (\bF_5) \times M$.
\end{theorem}

\begin{proof}
One direction is straightforward based on Suzuki and Zassenhaus's results (because every freely representable
group is Sylow-cycloidal, see Section~\ref{classification_section}, and every subgroup of a freely representable group is freely representable). 

In the other direction 
we have that $H$ is freely representable by Lemma~\ref{lemma160} and Corollary~\ref{rel_prime_cor}.
Every element of prime order of $G$ must be in $H$: for odd primes just note that all such elements
must be in the kernel of $G\to G/H$, and for the prime~$2$ use the uniqueness.
Now use  the technique of induced representations (Proposition~\ref{sufficient_prop})
\end{proof}

If we combine Proposition~\ref{solvable2_prop} and  Proposition~\ref{nonsolvable2_prop} 
we get the following:

\begin{proposition}\label{unique2_prop}
Suppose $G$ is a Sylow-cycloidal group that is not a Sylow-cyclic group. Then~$G$ has a unique element of order two.
\end{proposition}


\chapter{Semiprime-Cyclic Groups}

An important necessary condition for $G$ to be freely representable is the following:
if $H$ is a subgroup of order $pq$ where $p$ and $q$ are primes then $H$ is cyclic.\footnote{See Corollary~\ref{pq_cor}
above.  Wolf~\cite{wolf2011} calls this the \emph{pq-condition}. Note that we allow $p = q$ in this condition.}
The class of such groups forms an interesting class of ``cycloidal'' groups, and is 
surprisingly close to being the same as the class
of freely representable groups. There are  differences only in the nonsolvable case.
We call such groups ``semiprime-cyclic'' groups:

\begin{definition}
A finite group $G$ is said to be  \emph{semiprime-cyclic} if every subgroup~$H\subseteq G$ 
whose order is a semiprime (the product of two primes) is cyclic.
\end{definition}

We restate Corollary~\ref{pq_cor} using this new terminology:

\begin{proposition} \label{semiprime_cyclic_fr_prop}
Every freely representable group is semiprime-cyclic.
\end{proposition}

Now we consider some basic properties of semiprime-cyclic groups:

\begin{proposition}
The subgroups of a semiprime-cyclic group are semiprime cyclic.
\end{proposition}

\begin{proof}
This follows from the definition.
\end{proof}

\begin{proposition}\label{semiprime_cyclic_product_prop}
Suppose that $A$ and $B$ are finite groups of relatively prime order.
Then $A\times B$ is semiprime-cyclic if and only if both $A$ and $B$ are.
\end{proposition}

\begin{proof}
One direction is clear since $A$ and $B$ can be identified with subgroups of
the product~$A\times B$. Conversely, suppose that $A$ and $B$ are both semiprime-cyclic.
Let~$D$ be a subgroup of $A \times B$ of order $pq$ where $p, q$ are prime. 
Suppose $pq$ also divides the order of $A$. Then the image of $D$ under $A \times B \to B$ is trivial.
Thus~$D \subseteq A$, and hence $D$ is cyclic. Similarly if $pq$ divides the order of $B$, then $D$ is cyclic.

So we can focus on the case where $p$ divides the order of $A$ and $q$ divides the order of $B$.
Then the image of $D$ under $A \times B \to B$ must have order $q$ and the kernel~$D \cap A$ must have
order $p$. Similarly, $D \cap B$ has order $q$. Thus both terms of the inclusion $(D\cap A)  (D \cap B) \subseteq D$ 
have order $pq$, and 
so this is an equality. This means that $D$ 
is isomorphic to  $(D\cap A) \times (D \cap B)$, and $D$ must
must be cyclic since it is isomorphic to the product of cyclic groups of relatively prime order.
\end{proof}

\begin{proposition} \label{semiprime_cyclic_primes_prop}
Let $G$ be a finite group with subgroup $H$. If $H$ is semiprime-cyclic and
if $H$ contains every 
element of $G$ of prime order, then $G$ is also semiprime-cyclic.
\end{proposition}

\begin{proof}
Let $D$ be a subgroup of $G$ of order $pq$ where $p$ and $q$ are prime.
Let $a, b \in D$ where $a$ has order $p$ and $b$ has order $q$.
By assumption $a, b \in H$, so $D$ is a subgroup of $H$. Thus $D$ is cyclic.
\end{proof}

\begin{proposition}
A finite Abelian group $A$ is semiprime-cyclic
if and only if it is cyclic.
\end{proposition}

\begin{proof}
If $A$ is cyclic then it is semiprime-cyclic since all subgroups of $A$ are cyclic.
If $A$ is semiprime-cyclic then it follows that $A$ is cyclic from the decomposition of 
Abelian groups into cyclic subgroups of prime power order. (See the remarks
after Corollary~\ref{pq_cor} for a more elementary argument).
\end{proof}

\begin{proposition}
Let $G$ be a $p$-group for some prime $p$. 
Then 
$G$ is semiprime-cyclic
if and only if 
$G$ is cyclic or (if $p=2$) is isomorphic to a general quaternion group.
\end{proposition}

\begin{proof}
A $p$-group $G$ is semiprime-cyclic if and only if all subgroups of $G$ of order~$p^2$ are cyclic.
So the result  follows from 
Theorem~\ref{pgroup_thm} and 
Corollary~\ref{unique_2_cor}.
\end{proof}

\begin{corollary} \label{sp_scq_cor}
If $G$ is semiprime-cyclic then $G$ is a Sylow-cycloidal group
\end{corollary}

Now we determine which Sylow-cyclic groups are semiprime-cyclic. We start with a special case:

\begin{lemma}
Let $C_p$ be cyclic of order $p$ and let~$C_{q^k}$ be cyclic of order~$q^k$
where $p$ and $q$ are distinct primes.
If a semidirect product $G = C_{q^k} \rtimes C_p$ is semiprime-cyclic then~$G$ is cyclic.
\end{lemma}

\begin{proof}
If~$k=1$ the result holds by definition. So assume $k\ge 2$, and we proceed by induction.  Consider the homomorphism $C_{q^k} \to C_{q^k}$
defined by $x\mapsto x^q$; let $A$ be its image and let $K$ be its kernel. Observe that $A$ is cyclic of order $q^{k-1}$
since a generator of $C_{q^k}$ maps to an element of order $q^{k-1}$. Thus
$K$ has order $q$.

The action of $C_p$ on $C_{q^k}$ restricts to an action on $A$ since $A$ is the only subgroup of $C_{q^k}$ of order $q^{k-1}$.
By induction we can assume $A \rtimes C_p$ is cyclic, so $C_p$ acts trivially on $A$. Let $a \in A$ be a generator, and
let $S$ be the elements of $C_{q^k}$ mapping to $a$ under $x \mapsto x^q$. 
Since $K$ has $q$ elements, $S$ also has $q$ elements.
Note that $C_p$ acts on $S$ with orbits of size $p$ or $1$. Since $S$ has $q$ elements, there is an orbit
of size~$1$. In other words, $C_p$ fixes
an element of $S$. But the elements of $S$ generate~$C_{q^k}$.
So~$C_p$ acts trivially on $C_{q^k}$. Thus  $C_{q^k} \rtimes C_p$ is Abelian. So  $C_{q^k} \rtimes C_p$ is cyclic
since it is semiprime-cyclic and Abelian.
\end{proof}

Next we extend the special case a bit:

\begin{lemma}
Let $A$ be a finite cyclic group, and let $C_p$ be cyclic of order $p$ where $p$ is a prime not dividing $|A|$.
If a semidirect product $G = A \rtimes C_p$ is semiprime-cyclic then~$G$ is cyclic.
\end{lemma}

\begin{proof}
Suppose $G = A \rtimes C_p$ is semiprime-cyclic but not cyclic.  Then it is non-Abelian (all Abelian semiprime-cyclic groups are cyclic),
and so $C_p$ acts nontrivially on $A$.
The action of $C_p$ on $A$ restricts to an action on any  subgroup of $A$ (since $A$ has at most one subgroup of any given order).
The groups $C_p$ cannot act trivially on all Sylow subgroups of $A$ since $A$ is generated by such groups.
So~$C_p$ acts nontrivially on some Sylow subgroup $P$
of $A$. However, by the previous lemma, the group $P \rtimes C_p$ is cyclic, a contradiction.
\end{proof}

\begin{proposition}\label{semiprime_cyclic_sylow_cyclic_prop}
Let $G$ be a Sylow-cyclic group. Then $G$ is semiprime-cyclic if and only if $G$ is freely representable.
\end{proposition}

\begin{proof}
One implication has been established (Proposition~\ref{semiprime_cyclic_fr_prop}), so we can assume 
that~$G$ is semiprime-cyclic.
Since $G$ is Sylow-cyclic, we can write $G$ as $A \rtimes B$ where $A$ and $B$ are cyclic groups of
relatively prime order.
Consider the kernel $K$ associated homomorphism~$B \to \mathrm{Aut}(A)$.
If $C$ is a cyclic subgroup of $B$ of prime order, then $A \rtimes C$ is cyclic by the above lemma.
So $C$ is contained in $K$. By Corollary~\ref{ns_cor} we conclude that $G$ is freely representable.
\end{proof}

We can strengthen the above to include Sylow-cyclic-quaternion groups:

\begin{theorem} \label{theorem_175}
Let $G$ be a solvable finite group. Then $G$ is semiprime-cyclic if and only if $G$ is freely representable.
\end{theorem}

\begin{proof}
One implication has been established (Proposition~\ref{semiprime_cyclic_fr_prop}), so we can assume 
that~$G$ is semiprime-cyclic.
If $G$ is Sylow-cyclic, we appeal to the previous proposition. 
If $G$ is not Sylow-cyclic then, by Proposition~\ref{general_m_prop}, there is a Sylow-cyclic subgroup $M$ of $G$ with the property that
$G$ is freely representable if and only if $M$ is freely representable.
Since $M$ is semiprime-cyclic,  $M$ is freely representable by the previous proposition.
Thus $G$ is freely representable.
\end{proof}

The above theorem is remarkable given that it does not hold for non-solvable groups:
if $p>5$ is a Fermat prime such as $p = 17$ then $\SL_2(\bF_p)$ is a non-solvable semiprime-cyclic group, but is not freely
representable (Corollary~\ref{linear_not_freely_representable_cor} and Corollary~\ref{SL_fermat_cor2}).

We have the following version of Suzuki's theorem (Theorem~\ref{suzuki_thm}):

\begin{theorem} \label{theorem_176}
Suppose $G$ is a finite group with a unique element of order $2$.
Then~$G$ is a non-solvable semiprime-cyclic group if and only if 
$G$ has a subgroup~$H$ of index $1$ or $2$ such that $H$
is isomorphic to $ \mathrm{SL}_2 (\bF_p) \times M$
where $p$ is a Fermat prime and
where $M$ is a freely representable group of order prime to $(p-1)p(p+1)$.
\end{theorem}

\begin{proof}
Suppose $G$ is a non-solvable semiprime-cyclic group.
By Corollary~\ref{sp_scq_cor}, $G$ is Sylow-cycloidal.
So by Suzuki's theorem (Theorem~\ref{suzuki_thm}), $G$ 
 has a subgroup~$H$ of index $1$ or $2$ such that $H$
is isomorphic to $\mathrm{SL}_2 (\bF_p) \times M$
where $p\ge 5$ is a prime and
where $M$ is a Sylow-cyclic subgroup of order prime to $(p-1)p(p+1)$.
Since~$\mathrm{SL}_2 (\bF_p)$ and $M$ are isomorphic to subgroups
of $G$ they are both semiprime-cyclic. Thus $p$ is a Fermat prime (Corollary~\ref{SL_fermat_cor2}).
By Proposition~\ref{semiprime_cyclic_sylow_cyclic_prop}, $M$ is freely representable.

Conversely,  assume the existence of such an $H$. 
By Corollary~\ref{SL_fermat_cor2}, $\mathrm{SL}_2 (\bF_p)$ is semiprime-cyclic.
Also $\mathrm{SL}_2 (\bF_p)$ is non-solvable (Corollary~\ref{linear_non_solvable_cor}) so the same is true of $G$.
Also $M$ is semiprime-cyclic (Proposition~\ref{semiprime_cyclic_sylow_cyclic_prop}).
Thus the product $H$ is semiprime-cyclic (Proposition~\ref{semiprime_cyclic_product_prop}).

Suppose $g$ is an element of $G$ of prime order. If $g$ has odd order, then its image under $G \to G/H$
is trivial so $g \in H$. If $g$ has order 2, then $g \in H$ simply
because $H$ has an element of order 2, and the element of order 2 in~$G$ is unique.
Since every element of prime order in $H$ is in $G$ then $G$ must also be semiprime-cyclic
(Proposition~\ref{semiprime_cyclic_primes_prop}).
\end{proof}

Note that semiprime-cyclic groups have unique elements of order 2:

\begin{proposition}
If $G$ is a semiprime-cyclic group then $G$ has a unique element of order $2$.
\end{proposition}

\begin{proof}
If $G$ is solvable, then it is freely-generated by Theorem~\ref{theorem_175}, so $G$ has a unique element of order~$2$.
If $G$ is non-solvable, then it is at least Sylow-cycloidal (Corollary~\ref{sp_scq_cor}), and so has a unique element of order $2$
by Proposition~\ref{unique2_prop}.
\end{proof}

\begin{remark}
The above proof seems like a very high-powered way to prove such a basic result.
Is there a direct proof?
\end{remark}

We can use the close connection between freely representability and the semiprime-cyclic conditions to 
characterize which groups are not freely representable.\footnote{This is the condition highlighted in~\cite{biasse2020norm}.
I think of this as the ``poison pill'' proposition.}

\begin{proposition}
Let $G$ be a finite group. Then $G$ is not freely representable if and only if $G$ has a subgroup $B$ such that
either (1) $B$ is a non-cyclic group of order the product of two primes, or (2) $B$ is
is isomorphic to $\mathrm{SL}_2 (\bF_p)$ for some Fermat prime $p \ge 17$.
\end{proposition}

\begin{proof}
Suppose $G$ is not freely representable. If $G$ is solvable then $G$ cannot be semiprime-cyclic
(semiprime-cyclic implies freely representable for solvable groups by Theorem~\ref{theorem_175}),
so $G$ has a subgroup $B$ satisfying condition (1). Next suppose $G$ is non-solvable
and does not contain a subgroup satisfying (1). 
Then $G$ is semiprime-cyclic and so has a subgroup $H$ and a Fermat prime $p$ describe by 
Theorem~\ref{theorem_176}.
Now $p\ne 5$ or else $G$ would be freely representable (Theorem \ref{zassenhaus_sukuki_thm}). 
Thus $H$, and hence $G$, has a subgroup $B$ isomorphic to 
$\mathrm{SL}_2 (\bF_p)$ for some Fermat prime $p \ge 17$.

Conversely, suppose $G$ has such a subgroup $B$. 
In either case, such a $B$ is not freely representable, so $G$ is not freely representable.
\end{proof}


\chapter*{Appendix: Sylow Theorems}

We take these as given.

\begin{theorem} [First Sylow Theorem] \label{sylow1_thm}
For every prime power $p^k$ dividing the order of $G$ there is a subgroup of $G$
of order $p^k$.
\end{theorem}

\begin{theorem} [Second Sylow Theorem] \label{sylow2_thm}
Let $G$ be a finite group of order divisible by a prime~$p$.
Every subgroup $H$ of $G$ that is  a $p$-group is contained in a $p$-Sylow subgroup $P$ of $G$.
\end{theorem}

\begin{theorem}  [Third Sylow Theorem]\label{sylow3_thm}
Let $G$ is a finite group of order $m p^k$ where~$m$ is not divisible by $p$, 
where $p$ is a prime, and where $k\ge 1$.
All $p$-Sylow subgroups of $G$ are conjugate. The number $t$ of $p$-Sylow subgroups 
of $G$ divides $m$
and
$$
t \equiv 1 \pmod p.
$$
\end{theorem}

\chapter*{Appendix: Some facts about $p$-groups.}

\begin{proposition} \label{center_prop}
Every nontrivial $p$-group has nontrivial center.
\end{proposition}

\begin{proof}
We let $G$ act on $G$ by conjugation. Each orbit has size a power of $p$, and an element is in the center $Z$
if and only if it is in an orbit of size $1$. Since $|G|$ is divisible by $p$, and each orbit involving a $g \not\in Z$
has size divisible by $p$, it follows that $|Z|$ is also divisible by $p$.
\end{proof}

\begin{proposition}\label{p^2_prop}
Let $G$ be a group of order $p^2$. Then $G$ is either cyclic or is isomorphic to the product of two cyclic group of order $p$.
\end{proposition}

\begin{proof}
Let $Z$ be the center of $G$. By the previous proposition $Z$ has order $p$ or $p^2$. 
Suppose $Z$ has order $p$, and let $g\in G$ be an element outside of $Z$. 
We let $G$ act on~$G$ by conjugation. The stabilizer of $g$ is a group containing $g$ and every element of $Z$.
Thus the stabilizer of $g$ is all of $G$. This means that $g$ is in the center, a contradiction.

Thus $Z$ has order $p^2$ and so $G =Z$ must be Abelian. If $G$ is cyclic, we are done, so assume $A$ is any
nontrivial cyclic subgroup. Let $B$ be the cyclic subgroup generated by any $b \not\in A$. Observe that $A \cap B = \{1\}$.
Thus the map $A \times B \to AB$ is an isomorphism. The result follows from the fact that $G = AB$.
\end{proof}

\begin{proposition} \label{indexp_prop}
Suppose $G$ is a $p$-group for prime $p$ and $H$ is a subgroup of index~$p$ in~$G$.
Then $H$ is a normal subgroup of $G$.
\end{proposition}

\begin{proof}
We let $G$ act on the collection of subgroups by conjugation. The stabilizer of $H$, which is just
the normalizer of $H$, must be $G$ or $H$. Suppose the normalizer is $H$, so the orbit $\mathcal H$ of $H$ has
size $p$.
Note that the normalizer of any $g^{-1} H g$ in $\mathcal H$ must necessarily be  $g^{-1} H g$.

 Now we restrict the action of $G$ on $\mathcal H$ to an action of $H$ on $\mathcal H$.
 Clearly $H\in \mathcal H$ is fixed under this action. If $g^{-1} H g$ is not $H$, then the stabilizer
 of $g^{-1} H g$ is the intersection of $H$ with $g^{-1} H g$, which is a proper subgroup of $H$.
So the orbit of such~$g^{-1} H g$ under the action of $H$ must have
size greater than $1$. The orbit of $H$ has size $1$ and the
 orbit of $g^{-1} H g$ has size at least $p$. However, 
both of these orbits (under $H$) are contained in $\mathcal H$, which has only $p$ elements. This is a contradiction.
\end{proof}

\begin{proof} [Second proof (induction)]
Suppose that $H$ does not contain the center $Z$ and observe that $H$ is normal in $G=ZH$.
If $H$ contains the center $Z$ then by the induction hypothesis~$H/Z$ is normal in $G/Z$ and so $H$ is normal in $G$.
\end{proof}

\begin{proof} [Third proof]
This is actually a consequence of the following.
\end{proof}

\begin{proposition} \label{cseries_prop}
Every finite $p$-group $G$ is solvable.  In fact
if $H$ is a  subgroup of~$G$ then there is a composition series
$$
\left\{e\right\} = G_0 \subsetneq G_1 \subsetneq G_2 \ldots \subsetneq G_k = G
$$
such that $H = G_i$ for some $i$. (Here $G_j$ is a normal subgroup of $G_{j+1}$ of index $p$.)
\end{proposition}

\begin{proof}
Every $p$-group has a nontrivial center (Proposition~\ref{center_prop}), and so has a normal subgroup of order $p$.
This is enough to show that any $p$-group is solvable.
In particular, if $H$ is a subgroup of $G$ then $H$ has a composition series.
Now we need to show that we can extend the series. In other words if we need
establish the claim that if~$G_j$ is a proper subgroup of $G$ then we can find a subgroup $G_{j+1}$ such that~$G_j$ is 
a normal subgroup of $G_{j+1}$ of index $p$. 

We prove this claim by induction on $k$.
There are two cases. If~$G_j$ contains the center~$Z$ of $G$
then consider $G_j/Z$ inside $G/Z$ and argue (by the induction hypotheis) that there is a subgroup $G_{j+1}$ containing $Z$
such that $G_j/Z$ is normal of index $p$ in $G_{j+1}/Z$. This implies $G_{j+1}$ is as desired.
If $G_j$ does not contain the center $Z$,
then observe that $G_j$ is a proper normal subgroup of $G_j Z$, so the desired
group $G_{j+1}$ corresponds to any subgroup of order $p$ in $G_j Z / G_j$.
\end{proof}


\section*{Appendix: Frobenius's Theorem}

\begin{theorem} [Frobenius] \label{frobenius_thm}
Suppose $n$ divides the order of a finite group $G$. Then the number of elements in $\{ x \in G \mid x^n = 1\}$
is a multiple of $n$.
\end{theorem}

Actually we will view Frobenius's theorem as a special case of Theorem~\ref{strong_frobenius},
which seems to be easier to prove than trying to prove Frobenius's theorem directly (see \cite{hall1959} and Zassenhaus).
We will use the following notation:

\begin{definition}
Let $G$ be a group and let $n$ be a positive integer. If $a \in G$ then $a^{1/n}$ is the set of $x\in G$ whose $n$th power is $n$.
We extend this notation as follows: if $C$ is a subset of $G$ then
$$
C^{1/n} \; \defeq \; \left\{ x \in G \mid x^n \in C \right\}
$$
\end{definition}

\begin{remark}
We will need the following easy identities: $C^1 = C$, 
$$(C_1 \cup C_2)^{1/n} = C_1^{1/n} \cup C_2^{1/n},$$ 
$$
(a^{1/n_1})^{1/n_2} = a^{n_1 n_2},
$$
and
$$
g^{-1} a^{1/n} g = (g^{-1} a g)^{1/n}.
$$
\end{remark}

We start with a few lemmas. 

\begin{lemma}
Let $G$ be a finite group and let $p$ be a prime dividing the order of~$G$. 
If $a \in G$ has order divisible by $p$ then $|a^{1/p}|$ is also divisible by $p$.
\end{lemma}

\begin{proof}
If $x \in a^{1/p}$ has order $m$ then~$x^p = a$ has order $m/\mathrm{gcd}(p, m) = n$. In particular, $p \mid m$ since $p \mid n$, so 
$x$ has order $m = p n$.
 We partition~$a^{1/p}$ using the equivalence relation that defines $x$ and $y$
to be equivalent if and only if $x$ and $y$ generate the same cyclic subgroup of~$G$. 
For any $x \in a^{1/p}$ consider the homomorphism $\left<x\right> \to \left<a\right>$ defined by $x\mapsto x^p$.
The kernel has $p$-elements, and so it is a $p$-to-one map. 
The equivalence class of $x$ is the elements mapping to $a$ under this map, so it has $p$ elements.
This shows that $|a^{1/p}|$ is $k p$ where $k$ is the number of such equivalence classes (and $k$ is the number
of cyclic subgroup of order $p n$ containing $a$).
\end{proof}

In the next lemma, recall that the centralizer $Z_a$ of $a \in G$ is the subgroup
of elements of $G$ that commute with $a$.

\begin{lemma}
Let $G$ be a finite group. If $p$ is a prime dividing the order of~$G$
and if~$a$ is in the center of $G$ then  $p$ also divides~$|a^{1/p}|$.
\end{lemma}

\begin{proof}
Our proof will be by induction on the order of $G$, the case  $|G| = 1$ being trivial.
So suppose $|G| > 1$ and let  $p$ be a prime dividing the order of $G$.

Let $A$ be the subgroup of the center of $G$ consisting of elements of order prime to $p$.
By the previous lemma, it is enough to show that $p$ divides $|a^{1/p}|$ for all $a \in A$.

Let $a_1, a_2 \in A$. Consider the map $A \to A$ defined by $x \mapsto x^p$.
Its kernel is trivial, so it is a bijection. Thus there is a $u\in A$ such that $u^p  = a_2 a_1^{-1}$.
Consider the map $a_1^{1/p} \to a_2^{1/p}$ defined by the rule  $x \mapsto u x$.
Observe that it is well-defined and injective, so is a bijection. Thus $|a^{1/p}|$ is the same for all $a \in A$.
In particular,~$|A^{1/p}| = |A| |a^{1/p}|$ for any $a \in A$.

Let $B$ be the set of elements $b$ of $G$ whose centralizer $Z_b$ is a proper subgroup of~$G$ of order
divisible by $p$. Observe that $Z_b$ contains $b^{1/p}$ for any $b\in B$ since if~$x^p = b$
then $x^{-1} b x = x^{-1} x^p x = x^p  = b$ (in other words $b^{1/p}$ relative
to $Z_b$ is the same as $b^{1/p}$ relative to $G$). Also observe that $b$ is in the center of~$Z_b$.
So by the inductive hypothesis applied to $Z_b$, the prime $p$ divides $|b^{1/p}|$ for each $b\in B$. Thus~$p$ divides~$|B^{1/p}|$.

Let $C$ be the set of elements of $G$ not in $A$ or $B$. In other words, $C$ is the elements whose 
centralizer $Z_b$ has order prime to $p$. 

For each $x \in G$ let $C_x$ be the set of conjugates of $x$ in $G$.
So $G$ acts on $C_x$ transitively with stabilizer $Z_x$. The orbit size $|C_x|$ is $|G|/|Z_x|$.  Also if $x \in C$ then for any $y \in C_x$ we have $Z_y$ and $Z_x$ have the same order.
Since $g^{-1} x^{1/p} g = (g^{-1} x g)^{1/p}$, we have that $|y^{1/p}| = |x^{1/p}|$ if $y \in C_x$.
In particular, 
$$
|C_x^{1/p}| = |C_x| |x^{1/p}|.
$$

In particular, if $x \in C$ then the orbit $C_x$ has size a multiple of $p$. For $x\in C$ and~$y\in C_x$, we
have that $Z_y$ has the same order as $Z_x$, so $Z_y$ has order prime to~$p$. 
In other words, if $x \in C$ then $C_x \subseteq C$.
We conclude that $|C_x^{1/p}| = |C_x||x^{1/p}|$ is divisible by $p$ for each $x \in C$, and
$|C^{1/x}|$ is divisible by $p$ since $C$ is the disjoint union of such $C_x$.

Since
$$
|G| = |(A \cup B \cup C)^{1/p}| = |A^{1/p}| + |B^{1/p}| + |C^{1/p}|
= |A| |a^{1/p}| + |B^{1/p}| + |C^{1/p}|
$$
and since $|G|, |B^{1/p}|,$ and $|C^{1/p}|$ are divisible by $p$, we 
see that $|A| |a^{1/p}|$ is divisible by $p$. However $|A|$ is not divisible by $p$
since $A$ is an Abelian group whose elements have their orders prime to $p$.
So $|a^{1/p} |$ is divisible by $p$ as desired.
\end{proof}

\begin{theorem} \label{strong_frobenius}
Let $G$ be a finite group and let $n$ be a positive integer.
If $C$ is a subset of~$G$ closed under conjugation then $|C^{1/n}|$ is a multiple of $\mathrm{gcd}(n|C|, |G|)$.

In particular, if $a \in G$ is in the center of $G$, and if $n$ divides $|G|$, then $|a^{1/n}|$ is a multiple of $n$.
\end{theorem}

\begin{proof}
We can restate the result as asserting that $|C^{1/n}|$ is a $\bZ$-linear combination of $n|C|$ and $|G|$.

Observe that the result holds trivially for $G = \{ 1 \}$ regardless of $n$. So we can inductively assume the result
is true for all subgroups of $G$ (regardless of $n$). Note also that this result holds in $G$ itself for $n=1$ 
so we can 
inductively assume the result holds for all divisors of $n$ (for our given $G$).

Suppose $C = \{a \}$ and $n = p$ is a prime dividing $|G|$. Then $a$ must be in the center of $G$, and the result
holds by the previous lemma. 
Also if $C = \{a \}$ and $n = p$ is a prime not dividing $|G|$, the result holds trivially.

Next consider the case $C = \{ a \}$ where $a \in G$ in the center of $G$ and where~$n$ is composite.
Write $n$ as $n_1 n_2$ where $n_1, n_2$ are proper divisors of $n$.
Since~$n_1 < n$ we have by induction that $|a^{1/n_1}| = u (n_1 \cdot 1) + v |G|$ for some $u, v \in \bZ$.
Similarly,
$$
\left| \left( a^{1/n_1}  \right)^{1/n_2} \right| = u' n_2 (u n_1 + v |G|) + v' |G|
$$
for some $u', v' \in \bZ$.
In particular,
$$
\left| \left( a^{1/n_1}  \right)^{1/n_2} \right| = (u u') n + v'' |G|
$$
for some $v'' \in \bZ$. Thus the result holds for $C=\{a \}$ and composite $n$.

We have established the result for $C = \{ a \}$ with $a$ in the center. In general, by the identity
$(C_1 \cup C_2)^{1/n} = C_1^{1/n} \cup C_2^{1/n}$ it is enough to prove the result for sets~$C$ which
are in a single orbit under the action of $G$ by conjugation.
We can also assume $|C| > 1$. For any $a \in C$ and $g \in G$ we have $g^{-1} a^{1/n} g = (g^{-1} a g)^{1/n}$.
So~$a^{1/n}$ has the same size for all $a \in C$. Thus for any particular $a \in C$,
$$
|C^{1/n}| = |C| |a^{1/n}|.
$$
Let $Z_a$ be the centralizer of a particular $a \in C$.
In other words, $Z_a$ is the subgroup stabilizing $a$ under the conjugation action. 
Note that~$|Z_a| |C| = |G|$ by the orbit-stabilizer principle. So $Z_a$ is a proper subgroup of $G$ since $|C|>1$. Also
$a^{1/n}$ is a subset of $Z_a$ since if $x^n = a$ then $x^{-1} a x = x^{-1} x^n x = x^n = a$, and $a$ is in the center
of $Z_a$. So by the induction hypothesis
applied to $Z_a$ and $\{ a\}$
$$
|a^{1/n}| = u (n\cdot 1) + v |Z_a| =  u n+ v |G|/|C|.
$$
for some $u, v \in \bZ$.
Thus
$$
|C^{1/n}| = |C| |a^{1/n}| = |C| (u n + v |G|/|C|) = u n |C| + v |G|
$$
as desired. This establishes the result for general $C$ closed under conjugation.

\end{proof}

\chapter*{Appendix: Group Representations}

We assume that the reader is familiar with the construction of the group ring $R[G]$ where $R$ is a commutative ring 
with unity and where
$G$ is a group (see Dummit and Foote~\cite{dummit_foote}, Section 7.2). Observe that~$R[G]$ is a commutative ring if and only if~$G$ is an Abelian group.

A \emph{linear representation} of $G$ on a vector space $V$ over a field $F$ is a homomorphism~$\Phi\colon G \to \mathrm{GL}(V)$
from $G$ to the group $\mathrm{GL}(V)$ of invertible linear transformations of $V$. If $V$ is $F^n$ then $\mathrm{GL}(V)$
can be identified with $\mathrm{GL}_n(F)$, the group of~$n$-by-$n$ invertible matrices with coefficients in $F$.
If we fix an ordered basis of $V$ of size $n$ then we can identify $V$ with $F^n$, and so identify $\mathrm{GL}(V)$
with $\mathrm{GL}_n(F)$.

If~$\Phi\colon G \to \mathrm{GL}(V)$ is a linear representation of $G$ then, for $g\in G$ and $v \in V$, we often
write $g v$ for $\Phi(g) v$ if there is no chance of confusion. This is the usual convention for group actions in general.

We can also think of a linear representation of $G$ as a way to give an $F$-vector space an $F[G]$-module structure.
In other words, each $F[G]$-module $V$ supplies, canonically, a linear representation of $G$ on the $F$-vector space $V$.

Suppose $V$ is a representation, thought of as an $F[G]$-module, then a \emph{subrepresentation} of $V$ is a submodule.
This can be thought of as a subspace $W \subseteq V$ invariant under the action of $G$. We will use the term \emph{nontrivial}
for representations where $V$ has at least dimension one.
A representation is \emph{irreducible}
if its only subrepresentations are itself and the zero subspace.

In representation theory it is often convenient to use the field $F = \bC$. In this case all irreducible representations
of finite Abelian groups $A$ are one-dimensional, and the image of $A$ under $A \to \mathrm{GL} (V) = \bC^\times$ 
is a cyclic group.

This report uses representation theory, but not in a deep way; certainly much less representation theory than mentioned
in Wolf~\cite{wolf2011}. Inducted representations and the tensor product of representations
are used. Such ideas are reviewed in the document as needed.

\chapter*{Appendix: Orthogonal Transformations}

Recall that $\bR^n$ has an inner product, and this inner product can be
used to define distances and angles. We write $\left< x, y\right>$ for
the inner product between vectors.
When we say ``orthogonal'' or ``orthonormal'' it will be with respect to this standard inner product.
We write $\be_1, \ldots, \be_n$ for the standard orthonormal basis of $\bR^n$.

\begin{definition}
An \emph{orthogonal matrix} is a square matrix with real entries such
that the columns are orthonormal with respect to the standard
inner product.
\end{definition}

\begin{proposition}
Let $L$ be a linear transformation from $\bR^n \to \bR^n$. Then the following are equivalent.
\begin{enumerate} [(i)]
\item
$L$ preserve the inner product. In other words, for all $x, y \in \bR^n$
$$
\left< L x, L y \right> =  \left< x, y \right>.
$$

\item
$L$ maps any orthonormal basis to an orthonormal basis.

\item
$L$ maps some orthonormal basis to an orthonormal basis.

\item
The matrix representation of $L$ with respect to any orthonormal basis of $\bR^n$
is an orthogonal matrix.

\item
The matrix representation of $L$ with respect to some orthonormal basis of $\bR^n$
is an orthogonal matrix.

\end{enumerate}
\end{proposition}

The above motivates the following definition:

\begin{definition}
An \emph{orthogonal transformation} $\bR^n \to \bR^n$ is a linear map such that
$$
\left< L x, L y \right> =  \left< x, y \right>
$$
for all $x, y \in \bR^n$.
Since distances and angles are defined in terms of this inner product,
orthogonal transformations preserve angles and lengths.

The set of orthogonal transformation of $\bR^n$ is seen to be a group.
We call this group the \emph{orthogonal group} and write it as $\OO(n)$.
If we fix an orthonormal basis, we will also identify $\OO(n)$ with the group of $n$-by-$n$
orthogonal matrices.
The subgroup of matrices of determinant $1$ will be written $\SO(n)$.
\end{definition}

\begin{proposition}
The group $O(n)$ consists of all isometries of $\bR^n$ fixing the origin.
\end{proposition}

\begin{example}
The rotation of $\bR^2$ counter-clockwise by $\theta$ 
sends $\be_1$ to $(\cos \theta, \sin \theta)$, and~$\be_2$
to $(\cos (\theta + \pi/2) , \sin  (\theta + \pi/2))$.
So it is represented by
$$
\begin{pmatrix}
\cos \theta & -\sin \theta\\
\sin \theta & \cos \theta
\end{pmatrix}.
$$
Note this is orthogonal with determinant $1$.
Clearly composition of such rotations corresponds to addition of angles (mod $2\pi$).
\end{example}

\begin{proposition}
Let $A$ be an $n$-by-$n$ real matrix.
The matrix $A$ is orthogonal if and only if 
$$
A^T A = I.
$$
\end{proposition}

\begin{corollary}
So every orthogonal matrix $A$
has inverse $A^T$.
Every orthogonal matrix has determinant $1$ or $-1$. 
\end{corollary}

\begin{corollary}
The transpose of an orthogonal matrix is orthogonal matrix with the same determinant.
An~$n$-by-$n$ real  matrix if orthogonal if and only if its rows are orthonormal.
\end{corollary}

\subsection{The group $\OO(2)$}

Now we investigate $\OO(2)$ in more detail.
Every unit vector of $\bR^2$ is of the form~$(\cos \theta, \sin \theta)$. So any orthogonal transformation maps
$\be_1$ to some such unit vector. Now $\be_2$ will have to map to an orthogonal vector
$$
(\cos (\theta \pm \pi/2), \sin(\theta \pm \pi/2))
= (\mp \sin \theta, \pm \cos \theta)
$$
(where we use the standard trig identities for angle addition).
So the associated matrix has two possibilities
$$
\begin{pmatrix}
\cos \theta & - \sin \theta\\
\sin \theta &  +\cos \theta
\end{pmatrix}
\qquad\text{or}
\quad
\begin{pmatrix}
\cos \theta &  +\sin \theta\\
\sin \theta & -\cos \theta
\end{pmatrix}
$$
If the matrix is of determinant 1 then we have
the rotation by $\theta$  we considered above:
$$
\begin{pmatrix}
\cos \theta & -\sin \theta\\
\sin \theta &  \cos \theta
\end{pmatrix}.
$$
We conclude that $\SO(2)$ consists of rotation vectors. Note that $\SO(2)$
acts freely on the set~$\bR^2- \{(0, 0)\}$. We see that $\SO(2)$ is isomorphic to the circle group,
which can be expressed in terms of the angle group $\bR/ 2\pi \bZ$.

\begin{theorem}
The group $\SO(2)$ is equal to the rotation group, and so is isomorphic
to the circle group $\bR/ 2\pi \bZ$. This group acts freely on $\bR^2- \{(0, 0)\}$.
\end{theorem}

Now 
matrices of $\OO(2)$ outside of $\SO(2)$ have form
$$
\begin{pmatrix}
\cos \theta & \sin \theta\\
\sin \theta &  -\cos \theta
\end{pmatrix}
$$
The trace is $0$ and the norm is $1$. This means that the characteristic polynomial must be
$$
X^2 - \mathrm{trace} (M) X + \mathrm{det} M  = X^2 - 1.
$$ 
So it has eigenvalues $1$ and $-1$. This means there is a basis of eigenvectors
(recall that every real eigenvalue of a real matrix has a real eigenvector);  such a basis is orthogonal:

\begin{proposition}
All real eigenvalues of an orthogonal matrix are $1$ or~$-1$. If $x$ is an eigenvector of eigenvalue $1$ and $y$ is an eigenvector of eigenvalue~$-1$, then~$x, y$ are orthogonal.
\end{proposition}

\begin{proof}
Suppose~$M x = \lambda x$ where $M$ is an orthogonal matrix and where $x \ne 0$.
Then
$$
\left< x, x \right> =  \left<M x, M x\right>  =  \left<\lambda x, \lambda x\right>  = \lambda^2 \left< x, x \right>.
$$
Thus $\lambda^2 = 1$. So $\lambda = \pm 1$.

Now suppose that $M x =  x$ and $M y = - y$.
Then
$$
\left< x, y \right> = \left<M x, M y\right> =  \left< x, - y\right>  = -  \left< x, y \right>.
$$
So $2 \left< x, y \right> = 0$, and so $\left< x, y \right> = 0$.
\end{proof}

So if you have a matrix of $\OO(2)$ outside of $\SO(2)$, then it has characteristic polynomial
$X^2-1$. So has a fixed unit vector $x$, and an orthogonal eigenvector~$y$ 
mapping $y$ to $-y$.
If fixes the span of $x$,
and send an orthogonal vector to its inverse, so represents a reflection. (And does not act freely on $\bR^2- \{(0, 0)\}$).

\subsection{The group $\OO(3)$}

We start with the following:

\begin{lemma}
If $M \in \OO(3)$ then $M$ has at least one eigenvalue in the set $\{ \pm 1 \}$.
\end{lemma}

\begin{proof}
The characteristic polynomial is cubic in $\bR[X]$ and so has a root.
\end{proof}

Let $L$ be an orthogonal transformation of $\bR^3$.
If we choose an orthonormal vector whose first term is an eigenvector then we have representation
$$
\begin{pmatrix}
\pm 1 & * & *\\
0 & * & *\\
0 &* & *
\end{pmatrix}
$$
Since this must be an orthogonal matrix, it really has form
$$
\begin{pmatrix}
\pm 1 & 0 & 0\\
0 & * & *\\
0 &* & *
\end{pmatrix}
$$

And the $2$-by-$2$  lower right submatrix is orthogonal.

\begin{lemma}
If $M \in \SO(3)$ then $M$ has eigenvalue $+1$. 
\end{lemma}

\begin{proof}
From the previous lemma, $M$ has an eigenvalue of $1$ or $-1$. Suppose it has eigenvalue $-1$,
and choose an orthonormal basis whose first term is an eigenvector with eigenvalue $-1$.
In terms of this basis, the orthogonal transformation becomes
$$
\begin{pmatrix}
- 1 & 0 & 0\\
0 & * & *\\
0 &* & *
\end{pmatrix}.
$$
The submatrix must be in $\OO(2)$ and must have determinant $-1$. This means it has eigenvalue $1$, which
gives us an eigenvector of the form $(0, *, *)$ in $\bR^3$ with eigenvalue~$1$.
\end{proof}

\begin{lemma}
If $M \in \OO(3)$ has determinant $-1$ then  $M$ has eigenvalue $-1$. 
\end{lemma}

\begin{proof}
Apply the above lemma to $-M$.
\end{proof}

\begin{proposition}
If $M \in \SO(3)$ there is an orthonormal basis for which the matrix~$M$ has the form
$$
\begin{pmatrix}
1 & 0 & 0\\
0 & \cos \theta & -\sin \theta \\
0 &\sin \theta & \cos \theta
\end{pmatrix}
$$
\end{proposition}

\begin{proof}
Choose the first term of the basis to have eigenvalue $1$. So we get
$$
\begin{pmatrix}
1 & 0 & 0\\
0 & * & *\\
0 &* & *
\end{pmatrix}.
$$
The lower right  submatrix has determinant $1$ and is in $\SO(2)$, so it has the desired form.
\end{proof}

\begin{remark}
This shows that if $M \in \SO(3)$ is not the identity then it is a rotation of space. It fixes a one-dimension subspace called
the ``axis of rotation''. 
\end{remark}

 \begin{proposition}
If $M \in \OO(3)$ has determinant $-1$ has there is an orthonormal basis for which the matrix $M$ has the form
$$
\begin{pmatrix}
-1 & 0 & 0\\
0 & \cos \theta & -\sin \theta \\
0 &\sin \theta & \cos \theta
\end{pmatrix}
=
\begin{pmatrix}
-1 & 0 & 0\\
0 &  1 & 0\\
0 &0& 1
\end{pmatrix}
\begin{pmatrix}
1 & 0 & 0\\
0 & \cos \theta & -\sin \theta \\
0 &\sin \theta & \cos \theta
\end{pmatrix}
$$
\end{proposition}

\begin{proof}
The proof is similar to that of the previous proposition. 
\end{proof}

\begin{remark}
This shows that any such transformation is a rotation followed by a reflection across the plane perpendicular to the axis. (Or is just
a reflection if $\theta = 0$).
\end{remark}

\begin{theorem}
The group $\SO(3)$ is the group of rotation matrices for $\bR^3$ (and the identity matrix).
\end{theorem}

Finally, we will need the following:

\begin{lemma}
Suppose $A \in \OO (3)$. 
If $\det A = 1$ then the fixed space of $A$ is dimension~1 (for a rotation) or dimension 3 (for the identity). If $\det A = -1$ then the fixed space of $A$ has dimension 2 (for a reflection of planes)
or dimension 0.
In particular $A$ is a rotation if and only if the fixed space is dimension 1. 
\end{lemma}

\begin{proof}
This is fairly clear from the geometric description of the associated operators. We will discuss the case of $\det A = -1$ in more detail. The above discussion shows that there is an orthonormal basis $x_1, x_2, x_3$ such that $A x_1 = - A x_1$, and that $A$ fixes the span of $x_2$ and $x_3$, and is a rotation of angle $\theta$ on that subspace. So if~$y = c_1 x_1 + c_2 x_2 + c_3 x_3$
is fixed then $c_1 = -c_1$, so $c_1 = 0$. Thus $y$ is in the span of~$x_2$ and $x_3$.
Thus $y$ is fixed only if either $\theta$ is a multiple of $2\pi$, or if $y = 0$.
If $\theta$ is a multiple of $2\pi$ then $A$ is a refection with fixed space spanned by $x_2$ and $x_3$.
If $\theta$ is not a multiple of $2\pi$ then the fixed space is the zero space.
\end{proof}

\begin{remark}
Although we do not much about the topology of $\OO(3)$, we note rotations
are parameterized by three parameters: a point on the sphere for the axis and an angle of rotation.
In fact $\OO(3)$ is a 3-dimensional Lie group. The subgroup $\SO(3)$ and its other coset in $\OO(3)$ are
the connected components of $\OO(3)$.
\end{remark}

\chapter*{Appendix: Finite subgroups of $\SO(3)$}

The collection of rotations of $\bR^3$ fixing the origin can be identified with $\mathrm{SO}(3)$, the group of orthogonal
3-by-3 matrices with determinant~$+1$. 
Subgroups of $\mathrm{SO}(3)$ include cyclic and dihedral groups, but the other finite subgroups are related in an interesting
way to the Platonic solids (tetrahedron, cube, octahedron, dodecahedron, icosohedron). In fact, the 
finite subgroups of $\SO(3)$ are as follows:
\begin{itemize}
\item
Cyclic subgroups of all orders. We use the denotation $C_n$ for any such group of order $n$
\item
Dihedral subgroups of all even orders $\ge 4$. Any dihedral subgroup of order~$2n > 4$ is a rotational symmetry group of a regular $n$-gon centered at the origin.
We also consider here the rotational symmetry group of a nonsquare rectangle centered at the origin. It has order $4$, and is isomormorphic
to the Klein 4-group. We use the denotation $D_{n}$ for any such group of order $2n$.
\item
Tetrahedral subgroups of order 12. Each of these is the subgroup of  rotational symmetries of a regular tetrahedron centered at the origin. This group is isomorphic to 
the alternating group $A_4$. This can be seen by looking at the action on the four vertices.
We use the denotation $T$ for any such group.
\item
Octahedral subgroups of order 24. Each of these is the subgroup of rotational symmetries of a regular octahedron centered at the origin. 
Each such group is also the subgroup of rotational symmetries of a cube centered at the origin
This group is isomorphic to 
the alternating group $S_4$. This can be seen by looking at the action on the set of four pairs of opposite faces.
We use the denotation $O$ for any such group.
\item
Icosahedral subgroups of order 60. Each of these is the subgroup of rotational symmetries of a regular icosahedron centered at the origin. 
Each such group is also the subgroup of rotational symmetries of a regular dodecahedron centered at the origin.
This group is isomorphic to 
the alternating group $A_5$. We use the denotation $I$ for any such group.
\end{itemize}

Another interesting fact is that two finite subgroups of $\mathrm{SO}(3)$ are isomorphic as abstract groups if and only if they
are conjugate subgroups of $\mathrm{SO}(3)$.

See Artin~\cite{artin_algebra} (Sections 4.5 and 5.9), Goodman~\cite{goodman_algebra} (Chapter 4 and Section 11.3),
and Sternberg~\cite{sternberg1994} (Section 1.8) for  proofs of these facts.

\chapter*{Appendix: The  Quaternion Division Ring $\bH$}

We assume some familiarity with the ring of quaternions~$\bH$
(see Chapter 7 of Dummit and Foote~\cite{dummit_foote} and Exercise 1 on page~306 of Artin~\cite{artin_algebra}).
Here we highlight the connection between $\bH^\times$ and $\mathrm{SO}(3)$  which plays an important role in this document.
The ring $\bH$ is a non-commutative ring, and it is also a four-dimensional $\bR$-vector space. In other words
it is an \emph{algebra} over $\bR$. Every non-zero element is invertible, so it is a division ring, in fact a division algebra over~$\bR$.
If $\alpha = a + b \mathbf{i} + c \mathbf{j} + d \mathbf{k}$ is in~$\bH$ then we define its 
conjugate~$\overline{\alpha}$ to be~$a - b \mathbf{i} - c \mathbf{j} - d \mathbf{k}$. 
Observe that conjugation satisfies the law $\overline{\alpha\beta} = \overline{\beta}\, \overline{\alpha}$
(use~$\bR$-linearity to reduce to the case where~$\alpha, \beta$, and hence $\alpha \beta$, are are in~$\left\{\pm 1, \pm \mathbf{i}, \pm \mathbf{j}, \pm \mathbf{k} \right\}$).

We consider $\bH$ to contain $\bR$ as a subring (we can even think of $\bC$ as a subring as the span of~$1, \mathbf{i}$).
Note that $\bR$ commutes with any element of $\bH$.
For $\alpha \in \bH$, observe that ~$\alpha \in \bR$ if and only if $\overline{\alpha} = \alpha$.
Let $\bH_0$ be the $\bR$-span of
the quaternions~$\mathbf{i}, \mathbf{j}, \mathbf{k}$.
We can and will identify $\bH_0$ with~$\bR^3$.
An element $\alpha \in \bH$ is in~$\bH_0$ if and only if~$\overline{\alpha} = -\alpha$.
Every element of $\bH$ can be written uniquely as $c + v$ where $c \in \bR$ and $v \in \bH_0$;
we call $c$ the \emph{real part} of $\alpha = c + v$, and $v$ the \emph{imaginary part} of $\alpha$.

Observe that if $u, v \in \bR^3$, then the standard inner product is  related to the product $uv$ or $u\overline{v}$ in $\bH$.
In fact,
$$
\left< u, v \right>_{\bR^3} = - \mathit{Real} (uv) = \mathit{Real} (u \overline v).
$$
(Check using~$\bR$-linearity to reduce to the case where $u, v \in \left\{\mathbf{i}, \mathbf{j}, \mathbf{k} \right\}$).
This relation is even more direct if $u=v$;
observe that if $u \in \bH_0$ then 
$$\overline{u^2} = {\overline u}^2 = (-u)^2 = u^2,$$ so $u^2$ is real.
This means we can drop the symbol $\mathit{Real}$ and write
$$
|u|^2_{\bR^3} = \left< u, u\right>_{\bR^3} = - u^2 = u \overline u.
$$
We can extend the definition of $|\alpha|$ 
to all $\alpha \in \bH$ by defining
$$
|\alpha|^2 \defeq \alpha \overline \alpha.
$$
(Warning, this does not equal to $-\alpha^2$ in general).
Observe that if we write $\alpha$ as~$c + v$ in terms of its real and imaginary parts, then
$$
|\alpha|^2 = (c + v) \overline{(c+v) } = (c+v)(c-v) = c^2 - v^2 = c^2 + |v|^2 \ge 0.
$$
Note that this is the value of the standard norm when we think of $\bH$ as $\bR^4$.

Observe that this norm gives a homomorphism $\bH^\times \to \bR_+^\times$ defined by $\alpha \mapsto \alpha \overline \alpha$
where $\bR_+^\times$ is the multiplicative group of positive real numbers.
So the elements~$\bH_1$ of norm one is forms a subgroup of~$\bH^\times$. 
Note that geometrically,  $\bH_1$ is the $3$-sphere in $\bR^4$. (As a group it is a compact Lie group isomorphic to~$\mathrm{SU}_2$.)

When we define $\bH$ we seem to put special emphasis on~$\mathbf{i}, \mathbf{j}, \mathbf{k}$. However, the following shows
we can really have chosen other orthonormal vectors to play the role of $\mathbf{i}, \mathbf{j}, \mathbf{k}$. More precisely,
all unit length vectors in $\bH_0$ behave like $\mathbf{i}, \mathbf{j}, \mathbf{k}$ in having a square of $-1$,
and any orthonormal pair behaves like any two of $\mathbf{i}, \mathbf{j}, \mathbf{k}$:

\begin{proposition}
If  $x \in \bH_0$ has unit length then $x^2 = -1$.
If $x, y \in \bH_0$ are orthonormal then $x y = - y x$. 

Now suppose $x, y \in \bH_0$ are orthonormal  and define $z$ to be $x y$.
Then $x, y, z$ form an orthonormal basis of $\bH_0$
which satisfy the laws  
$$xy = z, \quad y z = x, \quad z x  = y.$$
\end{proposition}

\begin{proof}
Above we observed that if  $x \in \bH_0$ then $-x^2$ is the norm of $x$. So, due to the assumption of unit length,
$x^2$ must be $-1$. 

Suppose that $x, y$ are unit length and orthogonal. Then $x^2 = y^2 = -1$ as before. Also the real part of $z = xy$
is zero since $\left< x, y\right> = 0$. So $z$ is in $\bH_0$. Finally $z$ has length $1\cdot 1 = 1$, so we also have $z^2 = -1$.
Next we show that $z$ is orthogonal to~$x$ and $y$.
Observe that $x z = x x y  = - y$, which has real part zero, so $\left<x, z\right> = 0$. 
Similarly, $z y = x y y = -x$ has real part zero and so $\left<z, y\right> = 0$.
So $x, y, z$ are are orthonormal, and they must be a basis since~$\bH_0$ has dimension three over~$\bR$.

Next we verify that $x y = - y x$. Since $x, y,$ and $z = xy$ are in $\bH_0$,
$$
y x = (-y) (-x) = \overline y \, \overline x = \overline{x y} = \overline {z} = - z = - xy.
$$
This result was based only on the assumption that $u, v$ are orthonormal, so it applies to $y, z$ or $z, x$ as well:
$z y = - y z$ and $x z  = - z x$.

We have $x y = z$ by definition.  Above we noted that $z y= - x$ so $y z = x$.
We also noted that $x z = -y$ so $z x = y$.
\end{proof}

There is an action of $\bH^\times$ on $\bH$ by conjugation. More specifically, if $\Phi(h)$ is defined as the map $\alpha \mapsto h \alpha h^{-1}$,
then $\Phi(h)$ is an $\bR$-linear map $\bH \to \bH$ (in fact, it is an algebra automorphism).
What is interesting about this action is that if $|h| = 1$ then $\Phi(h)$ is an $\bR$-linear automorphism of $\bH_0$:

\begin{lemma}
Let $h \in \bH^\times$ and let $\Phi(h)$ be as above. Then
\begin{enumerate}
\item
If $c \in \bR$ then $\Phi(h) c = c$ for any $h\in \bH^\times$.
\item
If~$h \in \bH_1$ and $\alpha \in \bH$ then 
$$\overline{\Phi(h) \alpha} = \Phi(h) \overline \alpha.$$
\item
If  $h\in \bH_1$ and $v \in \bH_0$ then~$\Phi(h) v \in \bH_0$. 
\item
If $h\in \bH_1$ and $\alpha \in \bH$
then the real part of $\Phi(h) \alpha$ is equal to the real part of $\alpha$.
\end{enumerate}
\end{lemma}

\begin{proof}
The statement (1) is clear. For (2) observe that $\Phi(h) \alpha =h  \alpha \overline{h} $. So
$$
\overline{\Phi(h) \alpha} = \overline{h \alpha \overline{h}} = 
\overline{\overline h} \, \overline \alpha \, \overline h =
h \, \overline \alpha \, \overline h  = \Phi(h) \overline \alpha.
$$
For (3), recall that a vector $v$ is in $\bH_0$ if and only if $\overline v = - v$. Observe that 
$$
\overline{\Phi(h) v} = \Phi(h) \overline v  = \Phi(h) (-v) = - \Phi(h) v
$$
so $\Phi(h) v \in \bH_0$. 

For (4), write $\alpha = c + v$. Then 
$$
\Phi(h) \alpha = \Phi(h) c + \Phi(h) v = c + \Phi (h) v.
$$
Since $\Phi(h)v \in \bH_0$ the result follows.
\end{proof}

In particular, $\bH_1$ acts on $\bH_0$ by conjugation.
This action preserves the inner-product:

\begin{lemma}
Suppose $\Phi$ is as above. If $h\in \bH_1$ and if $x, y \in \bH_0$ then
$$
\left< \Phi(h) x, \Phi(h) y \right> = \left<x, y\right>.
$$
\end{lemma}

\begin{proof}
Observe that
$$
-\mathit{Real} \left( \left( \Phi(h) x \right) \left( \Phi(h) y \right) \right) = -\mathit{Real} \left( \Phi(h) (x y) \right) = - \mathit{Real} (xy).
$$
\end{proof}

In particular, if $h\in \bH_1$ then $\Phi(h)$ acts on $\bH_0$ as an orthogonal transformation. 
For convenience we identify $\bH_0$ with $\bR^3$ using the standard orthonormal basis $\mathbf{i}, \mathbf{j}, \mathbf{k}$.

\begin{proposition}
The map $h \mapsto \Phi(h)$ gives a 
homomorphism $\Phi \colon \bH_1 \to \mathrm{O}(3)$.
\end{proposition}

We will now identify the kernel and the image of this homomorphism. For the kernel we need the following:

\begin{proposition}\label{H_center_prop}
The elements of $\bH$ commuting with a nonzero $x \in \bH_0$ are the elements in the $\bR$-span of $1, x$.
The elements of $\bH$ commuting with all of $\bH_0$ is $\bR$. In particular, the center of the ring $\bH$ is $\bR$.
\end{proposition}

\begin{proof}
Suppose $x \in \bH_0$. Normalize so $x$ has unit length. Let $y \in \bH_0$ be such that~$x, y$ are orthonormal.
Let $z = x y$. As above, $x, y, z$ is an orthonormal basis of~$\bH_0$. 
Observe that for $a, b, c, d \in \bR$
$$
x(a + b x + c y + d z) = a x - b + c z - d y,\quad (a + b x + c y + d z) x = a x - b - c z + d y.
$$
The first result follows by comparing coefficients.
In particular, $\bR$ is the set of elements commuting with both~$\mathbf{i}$ and $\mathbf{j}$.
The remaining claims are now clear.
\end{proof}

\begin{proposition}
The kernel of $\Phi \colon \bH_1 \to \mathrm{SO}_3(\bR)$ is $\left\{ \pm 1 \right\}$.
\end{proposition}

\begin{proof}
It is clear from the definition of $\Phi$ that $1$ and $-1$ are in the kernel. If $h$ is in the kernel,
then $h u  = u h$ for all $u\in \bH_0$. So $h \in \bR$ from the previous result.
Since~$h$ has unit length, $h = \pm 1$.
\end{proof}

This is a good point to mention the following:

\begin{proposition}
The group $\bH^\times$, and hence $\bH_1$, has $-1$ as its a unique element of order $2$.
\end{proposition}

\begin{proof}
Suppose that $x^2 = 1$ with $x \in \bH$. Then $(x-1) (x + 1) = 0$. Since $\bH$ is a division ring, $x - 1 = 0$ or $x + 1 = 0$.
So $x = \pm 1$. The result follows.
\end{proof}

\begin{remark}
As we saw earlier, $x^4 = 1$ for all $x \in \bH_0 \cap \bH_1$. Note that $\bH_0 \cap \bH_1$ is the unit sphere in $\bH_0 \cong \bR^3$.
So $x^4 - 1$ has an infinite number of roots in $\bH$.
\end{remark}

Finally we present an argument that the image of our homomorphism is
the subgroup~$\SO(3)$.
We use the following (see the earlier appendix):

\begin{proposition}
Suppose that~$A \in \mathrm{O}(3)$ is not the identity matrix. Then the following are equivalent:
\begin{itemize}
\item
$A$ is a rotation of $\bR^3$.
\item
$A$ fixes exactly a one-dimensional subspace of $\bR^3$ (called the axis of rotation).
\item
$\det A = +1$, so $A \in \mathrm{SO}(3)$.
\end{itemize}
\end{proposition}

\begin{lemma}
Let $v\in \bH_0$ be of unit length. Let $h = r + s v$ 
where $r, s \in \bR$ where~$r^2 + s^2 = 1$ and such that~$-1 < r < 1$
so that $h \in \bH_1$.
Then the orthogonal transformation $\Phi(h)$ fixes~$v$.
Furthermore if~$u$ is a unit vector orthogonal to $v$ then the cosine of the angle formed by $u$ and~$\Phi(h) u$ is~$2r^2 - 1$.
In particular, $\Phi(h)$ is a rotation with axis of rotation spanned by $v$.
\end{lemma}

\begin{proof}
It is clear that $h$ has unit length since $r^2 + s^2 = 1$. By Proposition~\ref{H_center_prop},
the given vector $v$ commutes with $h = r + s v$. So
$$
\Phi(h) (v) = h^{-1} v h =  h^{-1} h  v= v.
$$ 
Let $u$ be orthogonal to $v$ and of unit length, and let $w = uv$. Then
$$
\Phi(h)  u  = (r + sv) u (r - sv) = r^2 u  - rs uv + rs v u - s^2 vuv = r^2 u -2 rs w - s^2 u.
$$
Thus
$$
\left(  \Phi(h)  u  \right) u =  (r^2 u -2 rs w - s^2 u) u = - r^2 - 2rs v + s^2 = (s^2 - r^2) + 2rs v.
$$
So
$$
\left< \Phi(h)  u, u \right> = 
- \mathit{Real}\left( \left(  \Phi(h)  u  \right) u \right)=
r^2 - s^2 = r^2 -  (1-r^2)  = 2 r^2 - 1.
$$
Since $u$, and hence $\Phi(h) u$, have unit length, the above inner product is just the cosine of the angle formed
by $u$ and $\Phi(h) u$.

Note this cosine is not $1$ since $r^2 < 1$. Thus $u$ is not fixed by $\Phi(h)$.
A general vector of $\bH_0$ can be written as $a v + b u$ for some $u$ of unit length orthogonal to~$v$.
The image of $av + bu$ under $\Phi(h)$ is $a v + b u'$ where $u' \ne u$. Thus if $b\ne 0$, the vector $a v + b u$ is not fixed by $\Phi(h)$.
This means that the space of vectors fixed by $\Phi(h)$ is one-dimensional.
By the above proposition, $\Phi(h)$ must be a rotation.
\end{proof}

Now we are ready for the  identification of the image.

\begin{proposition}
The image of $\Phi \colon \bH_1 \to \mathrm{O}(3)$ is $\mathrm{SO}(3)$.
\end{proposition}

\begin{proof}
Every element not equal to $\pm 1$ in $\bH_1$ can be written as $r + s v$ for some unit vector $v \in \bH_0$ and some $r,s \in \bR$
with $r^2 + s^2 = 1$ and $-1< r < 1$. By the above result, its image is a rotation, and so has determinant $1$.

Conversely, suppose $A$ is in $\mathrm{SO}(3)$. Then we need to find an element $h$ of $\bH_1$ with~$\Phi(h) = A$.
If $A$ is the identity, then $h = \pm 1$ works. So we can assume that~$A$ is a rotation. Let $v$ be a unit vector in $\bH_0$
in the axis of rotation of $A$, and let $\theta$ be the angle of rotation around this axis. Choose $r$ so that $\cos \theta = 2 r^2 - 1$,
and choose~$s$ so that $r^2 + s^2 = 1$.
Then by the above result, if $h = r + s v$ then $\Phi(h)$ will also be a rotation with axis $v$ and angle $\theta$.
This means that either $\Phi(h)$ or its inverse~$\Phi(h^{-1})$ must be $A$.
\end{proof}


\bibliographystyle{plain}

\bibliography{FreelyRepresentable}

\begin{thebibliography}{10}

\bibitem{allcock2018}
Daniel Allcock.
\newblock Spherical space forms revisited.
\newblock {\em Trans. Amer. Math. Soc.}, 370(8):5561--5582, 2018.

\bibitem{amitsur1955}
S.~A. Amitsur.
\newblock Finite subgroups of division rings.
\newblock {\em Trans. Amer. Math. Soc.}, 80:361--386, 1955.

\bibitem{artin_algebra}
Michael Artin.
\newblock {\em Algebra}.
\newblock Prentice Hall, Inc., Englewood Cliffs, NJ, 1991.

\bibitem{biasse2020norm}
Jean-Fran{\c{c}}ois Biasse, Claus Fieker, Tommy Hofmann, and Aurel Page.
\newblock Norm relations and computational problems in number fields.
\newblock {\em arXiv preprint arXiv:2002.12332}, 2020.

\bibitem{burnside1905b}
William Burnside.
\newblock On a general property of finite irreducible groups of linear
  substitutions.
\newblock {\em Messenger of Mathematics}, 35:51--55, 1905.

\bibitem{burnside1905a}
William Burnside.
\newblock On finite groups in which all the sylow subgroups are cyclical.
\newblock {\em Messenger of Mathematics}, 35:46--50, 1905.

\bibitem{dummit_foote}
David~S. Dummit and Richard~M. Foote.
\newblock {\em Abstract algebra}.
\newblock Prentice Hall, second edition, 1999.

\bibitem{goodman_algebra}
Frederick~M Goodman.
\newblock Algebra: Abstract and concrete, edition 2.6 (may, 2015).

\bibitem{hall1959}
Marshall Hall, Jr.
\newblock {\em The theory of groups}.
\newblock Dover (2018 reprint), 1959.

\bibitem{mazurov2003}
Victor Mazurov.
\newblock A new proof of {Z}assenhaus theorem on finite groups of
  fixed-point-free automorphisms.
\newblock {\em J. Algebra}, 263(1):1--7, 2003.

\bibitem{parry1977}
Charles~J. Parry.
\newblock Class number formulae for bicubic fields.
\newblock {\em Illinois J. Math.}, 21(1):148--163, 1977.

\bibitem{savage2019}
Anthony Savage.
\newblock An algorithm for computing units in multicyclic number fields.
\newblock Master's thesis, California State University San Marcos, 2019.

\bibitem{sternberg1994}
S.~Sternberg.
\newblock {\em Group theory and physics}.
\newblock Cambridge University Press, Cambridge, 1994.

\bibitem{suzuki1955}
Michio Suzuki.
\newblock On finite groups with cyclic {S}ylow subgroups for all odd primes.
\newblock {\em Amer. J. Math.}, 77:657--691, 1955.

\bibitem{wada1966}
Hideo Wada.
\newblock On the class number and the unit group of certain algebraic number
  fields.
\newblock {\em J. Fac. Sci. Univ. Tokyo Sect. I}, 13:201--209, 1966.

\bibitem{wall2013}
C.~T.~C. Wall.
\newblock On the structure of finite groups with periodic cohomology.
\newblock In {\em Lie groups: structure, actions, and representations}, volume
  306 of {\em Progr. Math.}, pages 381--413. Birkh\"{a}user/Springer, New York,
  2013.

\bibitem{wolf2011}
Joseph~A. Wolf.
\newblock {\em Spaces of constant curvature}.
\newblock AMS Chelsea Publishing, sixth edition, 2011 (first edition 1967).

\end{thebibliography}

\end{document}